\documentclass[12pt]{amsart}
\usepackage{amsmath,amsthm,amsfonts,amssymb,amscd}
\date{}

\newtheorem{theorem}{Theorem}[section]
\newtheorem{claim}[theorem]{Claim}
\newtheorem{remark}[theorem]{Remark}
\newtheorem{corollary}[theorem]{Corollary}
\newtheorem{proposition}[theorem]{Proposition}
\newtheorem{example}[theorem]{Example}
\newtheorem{definition}[theorem]{Definition}
\newtheorem{question}[theorem]{Question}
\newtheorem{problem}[theorem]{Problem}

\newcommand{\closed}{cl${}_1\kern-1pt$-osed}
\newcommand{\Nw}{\mathrm{Nw}}
\newcommand{\Hom}{\mathrm{Hom}}

\newcommand{\so}{\mathrm{so}}
\newcommand{\cs}{\mathrm{cs}}
\newcommand{\wcs}{\mathrm{cs}}

\newcommand{\cl}{\mathrm{cl}}
\newcommand{\scl}{\mathrm{cl}_1}

\newcommand{\C}{\mathcal{C}}
\newcommand{\I}{\mathcal{I}}

\newcommand{\pr}{\mathrm{pr}}
\newcommand{\supp}{\mathrm{supp}}
\newcommand{\weak}{\mathrm{weak}}
\newcommand{\A}{{\mathcal A}}
\newcommand{\BB}{{\mathcal B}}
\newcommand{\FF}{{\mathcal F}}

\newcommand{\MM}{{\mathcal M}}
\newcommand{\LL}{{\mathcal L}}

\newcommand{\KK}{{\mathcal K}}

\newcommand{\IR}{\mathbb{R}}
\newcommand{\IN}{\mathbb{N}}
\newcommand{\IQ}{\mathbb{Q}}

\newcommand{\U}{{\mathcal U}}
\newcommand{\V}{{\mathcal V}}
\newcommand{\W}{{\mathcal W}}

\newcommand{\Haus}{{\mathcal Haus}}
\newcommand{\SG}{{\mathcal S}}
\newcommand{\ind}{\mathrm{ind}}
\newcommand{\id}{\mathrm{id}}

\newcommand{\DD}{{\mathcal D}}
\newcommand{\e}{\varepsilon}
\newcommand{\w}{\omega}
\newcommand{\Ra}{\Rightarrow}
\newcommand{\uplim}{\overline{\lim}}
\newcommand{\ext}{ext}
\newcommand{\ord}{\mathrm{ord}}
\newcommand{\Ord}{\mathrm{Ord}}

\begin{document}

\title{$k^*$-Metrizable Spaces and their
Applications}

\author{T.O.~Banakh, V.I.~Bogachev, A.V.~Kolesnikov}

\begin{abstract}
In this paper we introduce and study so-called {\em $k^*$-metrizable spaces} forming a  new class of generalized metric spaces, and display various applications of such spaces in topological algebra, functional analysis, and measure theory.

By definition, a Hausdorff topological space $X$ is {\em $k^*$-metrizable} if $X$ is the image of a metrizable space $M$ under a continuous map $f:M\to X$ having a section $s:X\to M$ that preserves precompact sets in the sense that the image $s(K)$ of any compact set $K\subset X$ has compact closure in $X$.
\end{abstract}
\maketitle

\tableofcontents
\newpage

\section*{Introduction}

In this paper we introduce and study so-called {\em $k^*$-metrizable spaces} forming a  new class of generalized metric spaces, and display various applications of such spaces in topological algebra, functional analysis, and measure theory.

By definition, a Hausdorff topological space $X$ is {\em $k^*$-metrizable} if $X$ is the image of a metrizable space $M$ under a subproper map $\pi:M\to X$. A map $\pi:M\to X$ is called {\em subproper} if it admits a section $s:X\to M$ that preserves precompact sets in the sense that the image $s(K)$ of any compact set $K\subset X$ has compact closure in $M$. Since each subproper map is compact-covering, all compact subsets of a $k^*$-metrizable space are metrizable. Conversely, if all compact subsets of a space $X$ are metrizable, then $X$ is $k^*$-metrizable if and only if $X$ is $\cs^*$-metrizable in the sense that $X$ is the image of a metrizable space $M$ under a continuous map $\pi:M\to X$ having a section $s:X\to M$ that is $\cs^*$-continuous in the sense that for any convergent sequence $(x_n)$ in $X$ the sequence $(s(x_n))$ has a convergent subsequence. Thus the $k^*$-metrizability decomposes into  two weaker properties: the $\cs^*$-metrizability and the metrizability of all compact subsets.

The class of $k^*$-metrizable spaces is closed under many countable (and some uncountable) topological operations. 
Regular $k^*$-metrizable spaces can be characterized as spaces with $\sigma$-compact-finite $k$-network, see Theorem~\ref{n7.4}. 
 This characterization shows that the class of $k^*$-metrizable spaces is sufficiently wide and contains all La\v snev spaces (closed images of metrizable spaces), all $\aleph_0$-spaces (images of metrizable separable spaces under compact-covering maps) and all $\aleph$-spaces (regular spaces possessing a $\sigma$-locally finite $k$-network). Many results known for latter classes of spaces (like metrizability criteria) still hold in a more general framework of $k^*$-metrizable spaces. The choice of the term ``$k^*$-metrizable space'' was motivated by a characterization of regular $k^*$-metrizable spaces as spaces $X$ possessing a metric $\rho$ such that (i) each $\rho$-convergent sequence converges in $X$; (ii) a $\rho$-Cauchy sequence converges in $X$ if and only if it contains a subsequence convergent in $X$ and (iii) each compact subset $K\subset X$ is totally bounded  with respect to the metric $\rho$. Replacing the last condition with (iii)'{\em each convergent sequence  in $X$ contains a $\rho$-Cauchy subsequence}, we obtain a metric characterization of $\cs^*$-metrizable spaces.

Luckily, our motivation for a thorough study of $k^*$-metrizable spaces was outside the theory of generalized metric spaces and came
from probability theory.
According to a celebrated result of A.V.~Skorohod~\cite{Sk},
for every  sequence of Borel probability
measures $\mu_n$ on a complete separable metric space $X$ that
is weakly convergent to a Borel probability measure~$\mu_0$,
one can find Borel functions $\xi_n\colon\, [0,1]\to X$, $n=0,1,\ldots$,
such that $\lim\limits_{n\to\infty} \xi_n(t)=\xi_0(t)$ for almost all
$t\in [0,1]$ and  the image of Lebesgue measure $\lambda$ under
$\xi_n$ is $\mu_n$ for every $n\ge 0$. Various extensions of this result
have been found since then (see, e.g., \cite{BD},
\cite{BK}, \cite{CM},
\cite{Dudley}, \cite{F}, \cite{Jak0}, \cite{Schief},
and the references therein).
The most important for us is the extension discovered
independently by Blackwell and Dubbins \cite{BD} and Fernique \cite{F},
according to which all Borel probability measures on $X$ can be
parametrized simultaneously by mappings from $[0,1]$ with the preservation
of the above correspondence.
 It was shown in \cite{BK}
 that this result can be derived from its
 simple 1-dimensional case and certain deep topological selection theorems.

 More precisely, it was shown in \cite{BK} that to every Borel probability
Radon measure $\mu$ on a metrizable $X$ one can associate a Borel function
 $\xi_\mu\colon \, [0,1]\to X$ such that $\mu$ is the image  of
Lebesgue measure $\lambda$ under $\xi_\mu$ and if measures $\mu_n$ on $X$
converge weakly
 to $\mu$, then $\lim\limits_{n\to\infty}\xi_{\mu_n}(t)=\xi_\mu(t)$
for almost all $t\in [0,1]$. This property of the space $X$
was called the strong Skorohod property in \cite{BK}. Thus, each
metrizable space has the strong Skorohod property. The situation changes
beyond the class of metrizable spaces, see \cite{BBK1}, \cite{BBK1a}, \cite{BBK2}. It was observed in \cite[4.1]{BK}
that the inductive limit $\IR^\infty=\underset{\longrightarrow}{\lim}\,\IR^n$ of finite-dimensional Euclidean
spaces (in some sense, the simplest non-metrizable locally convex space)
fails to have the strong Skorohod property. 

On the other hand, it was noticed in \cite[4.4]{BK} that the space
$\IR^\infty$ possesses a weakened version of the strong Skorohod
property which was called the weak Skorohod property. Namely,
a topological space $X$ is defined to have the weak Skorohod property if
to each probability Radon measure $\mu$ on $X$ one can assign a Borel
function $\xi_\mu:[0,1]\to X$ such that $\mu$ is the image of the Lebesgue
measure $\lambda$ under $\xi_\mu$ and for every uniformly tight sequence
$(\mu_n)$ of probability Radon measures on $X$ the function sequence
$(\xi_{\mu_n})$ contains a subsequence that  converges almost surely.

In Lemma 4.2(ii) of \cite{BK} it was shown that proper maps preserve the weak Skorohod property. In Theorem~\ref{2.2} we make one step further and show that this property is preserved by subproper maps, thus generalizing
the aforementioned lemma from \cite{BK}. Since metrizable spaces have the weak
Skorohod property we conclude that any $k^*$-metrizable space also has
that property.

In light of this result it was natural to study subproper
maps more deeply. This will be done in Section~\ref{s1}. In Section~\ref{s2} we show that subproper maps are preserved by many topological constructions.
In Section~\ref{s3} we introduce $k^*$-metrizable spaces and present their metric characterization in Theorem~\ref{n3.4}. In this section we also show that $k^*$-metrizability is preserved by many topological operations.

In Section~\ref{s4} we decompose the $k^*$-metrizability into the $\cs^*$-metrizability plus the sequential compactness of all compact subsets, and study $\cs^*$-metrizable spaces in more details. Section~\ref{s5} is devoted to cardinal characteristics of $k^*$-metrizable spaces. The main result here is Theorem~\ref{n5.3} asserting that many cardinal characteristics of $k^*$-metrizable $k$-spaces (intermediate between  extent and  $k$-network weight) coincide. 

In Section~\ref{s7} we present a characterization of $\cs^*$-metrizable spaces in terms of $\sigma$-cs-finite \closed\  $\wcs^*$-networks and derive from this characterization a characterization of $k^*$-metrizable spaces as topological spaces having a $\sigma$-compact-finite \closed\  $k$-network. In Section~\ref{s8} we apply the characterizations of $k^*$-metrizable spaces to study the interplay between $k^*$-metrizable spaces and other  generalized metric spaces such as metrizable, stratifiable, semi-stratifiable, monotonically normal, La\v snev, $\aleph_0$- or $\aleph$-spaces. In particlar, we show that the class of $\aleph_0$-spaces coincides with the class of regular $k^*$-metrizable spaces having countable network.

In Section~\ref{s9} we apply the $k$-network characterization of $k^*$-metrizable spaces to detect  such spaces among function spaces $C_{\mathcal K}(X,Y)$.
In Sections~\ref{s10} and \ref{s11}  we pay tribute to our initial motivation and apply subproper maps and $k^*$-metrizable spaces to spaces of measures. In particular, we show that under some mild restrictions the functor of probability Radon measures preserves subproper maps as well as $k^*$-metrizable spaces. In Theorem~\ref{2.2} we notice that the weak Skorohod property is preserved by subproper maps and derive from this that submetrizable $k^*$-metrizable spaces have the weak Skorohod property, thus giving many natural examples of non-metrizable spaces with that property.

In Section~\ref{s12} we observe that $k^*$-metrizable spaces naturally appear in the theory of locally convex spaces as results of application of some operations to metrizable locally convex spaces. In particular we show that often operator spaces lead to $k^*$-metrizable spaces. Section~\ref{s13} is devoted to very special operator spaces, namely Banach spaces endowed with the weak topology. We observe that such a  space $(X,\weak)$ is $k^*$-metrizable if either the dual Banach  space $X^*$ is separable or else $X$ has the Shur property (= weakly convergent sequences are norm convergent). These two opposite cases are near to exhaust all $k^*$-metrizable spaces of the form $(X,\weak)$ where $X$ is a separable Banach space: if $(X,\weak)$ is $k^*$-metrizable, then either the dual $X^*$ is separable or else $X$ contains a closed infinite-dimensional subspace with the Shur property, see Theorem~\ref{4.16}. This result allows us to construct an open linear operator $T:X\to Y$ between separable Banach spaces such that $(X,\weak)$ is $k^*$-metrizable while $(Y,\weak)$ is not. This example shows that the $k^*$-metrizability is not preserved by open maps.

In the final Section~\ref{s14} we penetrate into the structure of sequential $k^*$-metrizable groups. In Theorem~\ref{n7.3} we show that any such a group $G$ either is metrizable or else contains an open $k_\w$-subgroup.  Applying this classification to locally convex spaces, we show in Theorem~\ref{n15.3} that there are only two topological types of non-metrizable sequential $k^*$-metrizable locally convex spaces: $\IR^\infty=\underset{\longrightarrow}{\lim}\,\IR^n$ and $[0,1]^\w\times \IR^\infty$.
A similar characterization holds also for non-metrizable zero-dimensional sequential $k^*$-metrizable groups, see Theorem~\ref{n15.2}.

\section{Subproper maps}\label{s1}

In this section we define subproper maps and study their relationship with other known classes of maps.

 Throughout, the term ``map'' always means a continuous map
 unlike the term ``function''. We consider only
 Hausdorff topological spaces.  As usual,
 $\bar A$ or $\cl(A)$ stands for the closure of a subset $A$
of a topological space~$X$.

A subset $A$ of a topological space $X$ is defined to be {\em precompact} if
it has compact closure in $X$. Observe that a map $f\colon\, X\to Y$ between
topological spaces is {\em proper} if and only if the preimage $f^{-1}(K)$ of any
precompact set $K\subset Y$ is a precompact subset of $X$.

\begin{definition} A map $f\colon\, X\to Y$ is defined to be {\em subproper} if there exists a subset
$Z\subset X$ such that $f(Z)=Y$ and for any precompact subset $K\subset Y$
the set $Z\cap f^{-1}(K)$ is precompact in $X$.
\end{definition}

There is a simple characterization of subproper maps in terms of sections preserving precompact sets.

A (possibly discontinuous) function $s\colon\, Y\to X$ is called a {\em section} of a
map $f\colon\, X\to Y$ if $f\circ s(y)=y$ for each point $y\in Y$. We say that a
function $f:X\to Y$ between topological spaces {\em preserves precompact sets} (or else is {\em precompact-preserving}) if the image $s(K)$ of any
precompact set $K\subset Y$ is precompact in $X$.

A map $f:X\to Y$ is called {\em compact-covering} if each compact subset $K\subset Y$ is the image of some compact set $C\subset X$ under $f$. Equivalently, $f$ is compact-covering if it induces a surjective function $\KK(f):\KK(X)\to\KK(Y)$, $\KK(f):K\mapsto f(K)$, between the families of compact subsets of $X$ and $Y$.

\begin{theorem}\label{n1.2}
For a surjective map $f\colon\, X\to Y$ between Hausdorff topological spaces the
following conditions are equivalent:

{\rm1)} $f$ is subproper;

{\rm2)} $f$ has a section preserving precompact sets;

{\rm3)} the induced function $\KK(f):\KK(X)\to\KK(Y)$ admits a section $s:\KK(Y)\to\KK(X)$ which is monotone in the sense that $s(K)\subset s(K')$ for any compact subsets $K\subset K'$ of $Y$;

{\rm4)} the induced function $\KK(f):\KK(X)\to\KK(Y)$ admits a section $s:\KK(Y)\to\KK(X)$ which is additive in the sense that $s(K\cup K')=s(K)\cup s(K')$ for any compact subsets $K,K'\subset Y$.
\end{theorem}

\begin{proof}
We shall prove the implications
$(1)\Rightarrow (4)\Rightarrow (3)\Rightarrow (2)\Rightarrow(1)$.

$(1)\Ra(4)$ Assuming that $f:X\to Y$ is subproper, find a subset $Z\subset X$ such that $f(Z)=Y$ and for any compact subset $K\subset Y$ the set $s(K)=\cl_X(Z\cap f^{-1}(K))$ is compact. The continuity of $f$ implies that $\KK(f)\circ s(K)=K$ and hence $s:\KK(Y)\to\KK(X)$ is a section of the induced function $\KK(f):\KK(X)\to\KK(Y)$. It is clear that for any  compact subsets $A,B\subset Y$ we get
$$\begin{aligned}
s(A\cup B)=&\;\cl(Z\cap f^{-1}(A\cup B))=
\cl\big(Z\cap (f^{-1}(A)\cup f^{-1}(B))\big)=\\
=&\;\cl\big((Z\cap (f^{-1}(A))\cup (Z\cap f^{-1}(B))\big)=\\
=&\;\cl\big(Z\cap (f^{-1}(A)\big)\cup \cl\big(Z\cap f^{-1}(B))\big)=s(A)\cup s(B)
\end{aligned}
$$which means that $s:\KK(Y)\to \KK(X)$ is an additive section of $\KK(f)$.
\medskip

The implication $(4)\Ra(3)$ is trivial since for an additive section $s:\KK(Y)\to\KK(X)$ of $\KK(f)$ and two compacta $A\subset B$ in $Y$ we get $s(B)=s(A\cup B)=s(A)\cup s(B)$ which implies $s(A)\subset s(B)$.
\medskip

$(3)\Ra(2)$ Assume that $l\colon\, \KK(Y)\to\KK(X)$ is a monotone section of $\KK(f)$.
Given a point $y\in Y$ pick any point $s(y)$ in the compact set
$l(\{y\})\subset X$. Since $f(l(\{y\}))=\{y\}$, the so-defined function
$s\colon\, Y\to X$ is a section of $f$. Next, assume that $K\subset Y$ is a
precompact subset of $Y$. Then $s(K)\subset\bigcup_{y\in K}l(\{y\})\subset
l(\bar K)$ lies in the compact subset $l(\bar K)$ of $X$ and thus $s(K)$
is a precompact subset of $X$.
\medskip

$(2)\Ra(1)$ Assume that $f$ admits a precompact-preserving section
$s\colon\, Y\to X$.
Let $Z=s(Y)$. Then for any precompact subset $K\subset Y$
the intersection $s(K)=Z\cap f^{-1}(K)$ is a precompact subset in $X$ and
thus $f$ is a subproper map.
\end{proof}

Precompact-preserving functions are tightly connected with $\cs^*$-continuous functions.

A function $s:Y\to X$ between topological spaces is called
\begin{itemize}
\item {\em sequentially continuous} (or else {\em $\cs$-continuous}) if for any convergent sequence $(y_n)$ in $Y$ the sequence $(s(y_n))$  converges in $X$;
\item {\em $\cs^*$-continuous} if for any convergent sequence $(y_n)$ in $Y$ the sequence $(s(y_n))$ has an accumulation point $x_\infty$ in $X$ (the latter means that each neighborhood of $x_\infty$ contains infinitely many points $s(y_n)$).
\end{itemize}

It is clear that each precompact-preserving function $f:X\to Y$ is $\cs^*$-continuous.
The converse is true if compact subsets of $Y$ are sequentially compact and $X$ is $\mu$-complete.
\smallskip

We recall that a topological space $X$ is called
 {\em sequentially compact} if
 each sequence $(x_n)$ in $X$  contains a convergent subsequence.
\smallskip

A topological space $X$ is called {\em $\mu$-complete} if each bounded subset of $X$ has compact closure, where a subset $B\subset X$ is {\em bounded} if for any locally finite collection $\U$ of open sets in $X$ only finitely many sets $U\in\U$ meet $B$, see \cite{Bla}. It is easily seen (and well-known) that a
subset $B$ of a Tychonoff  space is bounded if and only if for any continuous
real-valued function $f\colon\, X\to\IR$ the image $f(B)$ is bounded in $\IR$. According to \cite[8.5.13]{En}, each
Dieudonn\'e-complete space
(in particular, each paracompact space) is $\mu$-complete. On the other hand, the ordinal segment $[0,\omega_1)$
endowed with the natural interval topology is not $\mu$-complete.

\begin{proposition}\label{n1.3} Let $X,Y$ be topological spaces such that $X$ is $\mu$-complete and each compact subset of $Y$ is sequentially compact. A function $s:Y\to X$ is precompact-preserving if and only if $s$ is $\cs^*$-continuous.
\end{proposition}

\begin{proof} The ``only if'' part is trivial and holds without any assumptions on $X$ and $Y$. To prove the ``if'' part, assume that $s:Y\to X$ is a $\cs^*$-continuous function from a space $Y$ whose all compact subsets are sequentially compact into a $\mu$-complete space $X$. To show that $s$ is precompact-preserving, take any compact subset $K\subset Y$.  We claim that the image $s(K)$ is precompact. Assuming that it is not so, we get that $s(K)$ is not bounded by the $\mu$-completeness of $X$.
Consequently there is an infinite locally finite family $\U=\{U_n:n\in\w\}$ of open subsets of $X$ such that for every $n\in\w$ the intersection $U_n\cap s(K)$ contains a point $x_n$. Pick any point $y_n\in K$ with $s(y_n)=x_n$ and use the sequentially compactness of $K$ to find a convergent subsequence $(y_{n_k})$ of the sequence $(y_n)$. Since $s$ is $\cs^*$-continuous, the sequence $(s(y_{n_k}))$ has an accumulation point $x_\infty$. Then each neighborhood $W$ of $x_\infty$ contains infinitely many points $s(y_{n_k})=x_{n_k}$ and thus meets infinitely many sets $U_{n_k}\ni x_{n_k}$ which contradicts the local finity of $\U$.
\end{proof}

This proposition implies another characterization of subproper maps.

\begin{theorem}\label{n1.4} A map $f:X\to Y$ from a $\mu$-complete space $X$ into a space $Y$ whose all compact subsets are sequentially compact is subproper if and only if it has a $\cs^*$-continuous section $s:Y\to X$.
\end{theorem}

It turns out that each closed map has a $\cs^*$-continuous section.

\begin{proposition}\label{n1.5} Any section $s:Y\to X$ of a closed map $f:X\to Y$ is  $\cs^*$-continuous. Moreover if the space $X$ is $\mu$-complete, then $s$ preserves precompact sets.
\end{proposition}

\begin{proof} Assume that $f:X\to Y$ is closed and take any section $s:Y\to X$ of $f$. We claim that $s$ is $\cs^*$-continuous. Given a convergent sequence $(y_n)_{n\in\w}$ in $Y$ we should find an accumulation point for its image $(s(x_n))$ in $X$.
Without loss of generality, $y_n\ne y_\infty=\lim_{n\to\infty}y_n$ for all $n$. Assuming that  the sequence $(s(y_n))$ has no accumulation point in $X$, we will get that $\{s(y_n):n\in\w\}$ is a closed discrete subset in $X$. Since $f$ is closed, the set $\{y_n:n\in\w\}=f(\{s(y_n):n\in\w\})$ is closed in $Y$ which is not possible because this set has an accumulation point $y_\infty\notin \{y_n:n\in\w\}$.
\smallskip

Next, assume additionally that the space $X$ is $\mu$-complete. We shall show that the
 the image $s(K)$ of any compact set $K\subset Y$ is precompact in $X$. Assuming the converse, use the $\mu$-completeness of $X$ to find a locally finite countable collection $\U$ of open subsets of $X$ meeting the image $s(K)$.
For each $U\in\U$ pick a point $x_U\in U\cap s(K)$. Then the set
$\{x_U\colon\, U\in\U\}\subset X$ is closed and discrete in $X$. Since the map
$f$ is closed, $\{f(x_U)\colon\, U\in\U\}$ is a closed countable subset of the
compactum $\bar K$. Consequently, this set has a non-isolated point and
there is an infinite subcollection $\W\subset\U$ such that the set
$\{f(x_U)\colon\, U\in\W\}$ is not closed in~$Y$, which contradicts the
closedness of the set $\{x_U\colon\, U\in\W\}$ in $X$ and the closedness of the
map~$f$.
\end{proof}

This proposition can be partly reversed. First we recall some definitions.
\smallskip

A map $f\colon\, X\to Y$ is defined to be {\em inductively closed} (resp. {\em inductively
perfect\/}) if there is a closed subset $Z\subset X$ such that $f(Z)=Y$ and
the restriction $f|Z\colon\, Z\to Y$ of $f$ is a closed (resp. perfect) map.
(A map $f\colon\, X\to Y$ is {\em perfect} if it is closed and the
preimage $f^{-1}(y)$ of each point $y\in Y$ is compact, see \cite[\S3.7]{En}).
It is well-known that each perfect map is proper \cite[3.7.2]{En} and each
proper map into a $k$-space is perfect \cite[3.7.18]{En}.

A topological space $X$ is called a {\em Fr\'echet--Urysohn space} if for each
subset $A\subset X$ and each point $a\in \bar A$ there is a sequence
$(a_n)\subset A$ convergent to $a$. It is clear that each metrizable space is
Fr\'echet--Urysohn. A standard example of a non-metrizable Fr\'echet--Urysohn
space of arbitrarily large character is the {\em sequential fan} $S_\kappa$ where $\kappa$ is an infinite cardinal $\kappa$. By definition, $S_\kappa=\{2^{-n}:n\le\infty\}\times \kappa/\{0\}\times\kappa$ is the quotient space of the discrete sum of $\kappa$ many convergent sequences with their limit point glued together.

\begin{theorem}\label{n1.6} For a surjective map $f:X\to Y$ from a $\mu$-complete space $X$ into a Fr\'echet-Urysohn space $Y$ the following conditions are equivalent:
\begin{enumerate}
\item[1)] $f$ is subproper;
\item[2)] $f$ has a section preserving precompact sets;
\item[3)] $f$ has a $\cs^*$-continuous section;
\item[4)] $f$ is inductively closed.
\end{enumerate}
Moreover, if $Y$ contains no closed subspace homeomorphic to the sequential fan
$S_\w$, then (1)--(4) are equivalent to
\begin{enumerate}
\item[5)] $f$ is inductively perfect.
\end{enumerate}
\end{theorem}

\begin{proof} The equivalences $(1)\Leftrightarrow(2)\Leftrightarrow(3)$ follow from Theorem~\ref{n1.4}  and the fact that each Fr\'echet-Urysohn compact space is sequentially compact. The implication $(4)\Ra(3)$ has been proved in Proposition~\ref{n1.5} and $(5)\Ra(4)$ is trivial.

$(3)\Ra(4)$  Assume that $f$ has a $\cs^*$-continuous section $s:Y\to X$ and let $Z=\cl_X(s(Y))$. We claim that the
restriction $f|Z\colon\, Z\to Y$ is a closed map. Assume that this
 is not true. Then we
can find a closed subset $A\subset Z$ whose image $f(A)$ is not closed
in $Y$. Since $Y$ is a Fr\'echet--Urysohn space, there is a sequence
$(a_n)\subset A$ such that the sequence $(f(a_n))$ converges to a point
$y_0\in Y\setminus f(A)$. It can be easily shown that $B=\{a_n\colon\, n\in\IN\}$ is
a closed discrete subset of $A$. Since $B$ is not compact and the space
$X$ is $\mu$-complete, the set $B$ is not bounded in $X$. Consequently, there is a
locally finite countable collection $\U=\{U_n:n\in\w\}$ of open sets of $X$ meeting the set $B$. For every $i\in\IN$ pick a point $b_i\in B\cap
U_i$. Since $\U$ is an infinite locally finite family, the set $\{b_i\colon\, i\in\IN\}\subset B$ is
infinite (and discrete in $X$). Then the sequence $(f(b_i))$ converges to
$y_0$, being a subsequence of $(a_n)$.

For each $i\in\IN$ fix an open neighborhood $O(f(b_i))$ of $f(b_i)$ in $Y$
whose closure does not contain the point $y_0$ and let $W_i=U_i\cap
f^{-1}(O(f(b_i)))$. Then $\{W_i\colon\, i\in\IN\}$ is a locally finite family of open
subsets of $X$ and for every $i\in\IN$ the closure of $f(W_i)$ in $Y$ does
not contain the point $y_0$. Since $b_i\in\cl_X(W_i\cap s(Y))$, by the
continuity of~$f$, we get that $f(b_i)\in\cl_Y(f(W_i\cap s(Y)))$. Then
$y_0$ lies in the closure of the union $\bigcup_{i\in\IN}f(W_i\cap s(Y))$
in~$Y$. Since the space $Y$ is Fr\'echet--Urysohn, there is a sequence
$y_n\in\bigcup_{i\in\IN}f(W_i\cap s(Y))$ convergent to $y_0$. For each
$n\in\IN$ find $i(n)\in\IN$ such that $y_n\in f(W_{i(n)}\cap s(Y))$. Passing to
a subsequence, we can assume that $i(n)\ne i(n')$ for any $n\ne n'$. Then
$s(y_n)\in W_{i(n)}$ and hence  $\{s(y_n)\colon\, n\in\IN\}$ is an infinite closed
discrete subset of $X$ which has no accumulation point in $X$, a contradiction with the $\cs^*$-continuity of $s$ and the convergence of $(y_n)$.
\medskip

$(4)\Ra(5)$ Assume $f$ is inductively closed and $Y$
contains no closed subspace homeomorphic to the sequential fan $S_\w$. Let $Z$ be a closed subset of $X$, such that $f(Z)=Y$ and the restriction $f|Z:Z\to Y$ is a closed map. We shall show that the map $f|Z$ is perfect. Take any section $s:Y\to Z$ of the map $f$. We loose no generality assuming that $s(Y)$ is dense in $Z$. Fix any point $y_0\in Y$ and
assume that the pre-image $f^{-1}(y_0)\cap Z$ is not compact. Since the
space $X$ is $\mu$-complete, the set $f^{-1}(y_0)\cap Z$ is not bounded and thus there
is an infinite locally finite family $\U=\{U_n:n\in\w\}$ of open sets in $X$ meeting $f^{-1}(y_0)\cap Z$. Without
loss of generality, we may assume that
$s(y_0)\notin\cup\U$. For every $n\in\w$ pick a point $z_n\in U_n\cap f^{-1}(y)\cap Z$. Since
$z_n\in\cl_X(U_n\cap s(Y))$, the continuity of $f$ ensures that
$f(z_n)=y_0\in\cl_Y(f(U_n\cap s(Y)))$. Since the space $Y$ is
Fr\'echet--Urysohn, there is a sequence $T_n=\{y_{n,i}\}_{i=1}^\infty\subset
f(U_n\cap s(Y))$ convergent to $y_0$. This sequence is not trivial because
$s(y_0)\notin U_n\cap s(Y)$ and hence $y_0\notin f(U\cap s(Y))$.
Passing to a suitable subsequence we may assume that $y_{n,i}\ne y_{n,j}$ for all $i\ne j$.

We claim that there is an increasing number sequence $(n_k)$ such that the sets $T_{n_k}$, $k\in\w$, are pairwise disjoint. This sequence will be constructed by induction. Let $n_0=0$.
Assume that for some number $k\in\w$ the numbers $n_0<n_1<\dots<n_{k-1}$ have been constructed so that the sets $T_{n_i}$ are pairwise disjoint for $i<k$. It follows from Proposition~\ref{n1.3} that the section $s:Y\to Z$ preserves precompact sets. Consequently the set $K=\bigcup_{i<k}s(T_{n_i})$ is precompact in $X$ and meets only finitely many of sets of the locally finite family $\U$. Then we can find a number $n_k>n_{k-1}$ such that $K\cap U_{n_k}=\emptyset$. For this number $n_k$ the intersection $T_{n_k}\cap T_{n_i}$ is empty for all $i<k$. This completes the inductive construction.

We claim that $T=\{y_0,y_{n_k,i}\colon\, k,i\in\IN\}$ is a closed subset of $Y$ homeomorphic to the sequential fan $S_\w=S_1\times\w/\{0\}\times \w$, where $S_1=\{2^{-i}:i\in\w\}$ is a convergent sequence.
For this consider the map $h\colon\, S_\w\to T$ assigning to the non-isolated point of $S_\w$ the point $y_0$ and to each point $(2^{-i},k)\in S_\w$ the point $y_{n_k,i}$.

First we check that $T$ is closed in $Y$.
Since $Y$ Fr\'echet-Urysohn, it suffices to verify that for each compact subset $K\subset Y$ the intersection $T\cap K$ is compact. Since the section $s:Y\to Z$ preserves precompact sets, the image $s(K)$ of $K$ is precompact in $X$ and hence meets only finitely many sets $U_{n_k}\supset s(T_{n_k})$. Consequently, $K$ meets only finitely many sets $T_{n_k}$ with $k\le l$ for some $l$ and $K\cap T=\bigcup_{k\le l}K\cap\cl(T_{n_k})$ is compact.

The above argument shows also that each convergent sequence in $T$ lies in the union of finitely many sets $\cl(T_{n_k})$, which implies the continuity of the map $h:T\to S_\w$ The analogous property of the sequential fan ensures the continuity of the inverse map $h^{-1}:S_\w\to T$.
\end{proof}

\begin{remark}
{\rm The implication $(4)\Ra(5)$ is due to A.Arkhangelski \cite{Ar1}.
If $X$ and $Y$ are separable metrizable and $X$ is
Polish or $Y$ is $\sigma$-compact, then the five equivalent conditions of
Theorem~\ref{n1.6} hold if and only if the map $f$ is compact-covering,
see \cite{Chr}, \cite{SR}, \cite{Os}, \cite{Os2},
\cite{JW}, \cite{Mi1}. Under certain
additional set-theoretic assumptions (namely, the determinacy of all
analytic games) each compact-covering map $f\colon\, X\to Y$ between coanalytic
spaces is inductively perfect, see \cite{DSR1}. On the other hand, there
is a model of ZFC (namely, G\"odel's Constructive Universe) in which there
exists a compact-covering map $f\colon\, X\to\IN^\omega$ from an $F_\sigma$-subset
$X$ of the Baire space $\IN^\omega$, which is not inductively perfect and
thus is not subproper, see \cite{DSR2}, \cite{DSR3}.
This shows that
the interplay between compact-covering and subproper maps
is highly non-trivial and subtle even in the realm of separable
metrizable spaces.
}\end{remark}

\section{Operations over subproper maps}\label{s2}

In this section we show that subproper maps are preserved by many
topological constructions. We start from three simple observations.

\begin{proposition}\label{1.9} If $f\colon\, X\to Y$ is a subproper map, then for every subspace $Z\subset Y$ the map
$f|f^{-1}(Z)\colon\, f^{-1}(Z)\to Z$ is subproper.
\end{proposition}

\begin{proposition}\label{n2.2} The composition $f\circ g:X\to Z$ of two subproper maps
$g:X\to Y$ and $f:Y\to Z$ is subproper.
\end{proposition}

\begin{proposition}\label{1.10} For a family $\{f_i:X_i\to Y_i\}_{i\in\I}$ of subproper maps the induced map
$\prod_{i\in\I}f_i\colon\, \prod_{i\in\I}X_i\to \prod_{i\in\I}Y_i$ between Tychonoff products is  subproper.
\end{proposition}

Subproper maps are also preserved by the construction of
hyperspace. We recall that the {\em hyperspace} $\exp(X)$ of a topological
space $X$ is the space of all non-empty compact subsets of $X$ endowed
with the Vietoris topology generated by the base consisting of the sets
$$\langle U_1,\dots,U_n\rangle=\{K\in\exp(X)\colon\,  K\subset
\bigcup_{i=1}^nU_i,\; K\cap U_i\ne\emptyset\mbox{ for all }i\le n\},$$ where
$U_1,\dots, U_n$ run over all open subsets of $X$. Observe that
$\KK(X)=\exp(X)\cup\{\emptyset\}$. It is well-known \cite[3.12.26]{En}
that for any
compact Hausdorff space $X$ its hyperspace $\exp(X)$ is compact as well. If the
topology of a space $X$ is generated by a metric $d$, then the Vietoris
topology on $\exp(X)$ is generated by the Hausdorff metric
$$
d_H(A,B)=\max\{\max_{a\in A}d(a,B),\;\max_{b\in B}d(A,b)\},\;\;
A,B\in\exp(X).
$$

Given a map $f\colon\, X\to Y$ between topological spaces let
$\exp(f)\colon\, \exp(X)\to\exp(Y)$ be the map between their hyperspaces acting as
$\exp(f)(K)=f(K)$ for $K\in\exp(X)$.

\begin{proposition}\label{1.11} A map $f\colon\, X\to Y$ is subproper if and only if so is the
map $\exp(f)\colon\, \exp(X)\to\exp(Y)$.
\end{proposition}

\begin{proof} Assuming that the map $f\colon\, X\to Y$ is subproper, let $s\colon\, Y\to X$ be a section of $f$ preserving precompact
sets. Next, define a section $l\colon\, \exp(Y)\to\exp(X)$ of $\exp(f)$ letting
$l(K)=\cl_X(s(K))$ for $K\in\exp(Y)$. We claim that the section $l$
preserves precompact sets. Let $\KK$ be any compact subset of $\exp(Y)$.

First we show that the union $\cup\KK$ is a compact subset of $Y$. Let $\U$ be a cover of $\cup\KK$ by open subsets of $Y$. Assuming that no finite subfamily $\FF\subset\U$ covers $\cup\KK$, find a compact set $K_\FF\in\KK$ with $K_\FF\not\subset\cup\FF$. We can consider the family $[\U]^{<\w}$ of finite subsets of $\U$ as a partially ordered set ordered by the inclusion relation. The compactness of $\KK$ ensures the existence of an accumulation point $K_\infty\in\KK$ of the net $(K_\FF)_{\FF\in[\U]^{<\w}}$.
The latter means that for any open neighborhood $O(K_\infty)\subset\exp(Y)$ and any finite subset $\FF\subset\U$ there is another finite subset $\mathcal E\subset\U$ containing $\FF$ and such that $K_{\mathcal E}\in O(K_\infty)$. Using the compactness of $K_\infty$ in $\cup\KK\subset\cup\U$ find a finite family $\FF\subset\U$ with $K_\infty\subset\cup\FF$. Then $\langle \cup\FF\rangle=\{K\in\exp(Y):K\subset\cup \FF\}$ is an open neighborhood of $K_\infty$ in $\exp(Y)$. The accumulation property of $K_\infty$ ensures the existence of a finite set $\mathcal E\subset\U$ containing $\FF$ and such that $K_{\mathcal E}\in\langle\cup\FF\rangle$. Then $K_{\mathcal E}\subset\cup\FF\subset\cup\mathcal E$. On the other hand,  $K_{\mathcal E}\not\subset\cup\mathcal E$ by the choice of $K_{\mathcal E}$. This contradiction completes the proof of the compactness of $\cup\KK$.

Since the section $s$ is precompact-preserving, the set $C=\cl_X(s(\cup\KK))$ is compact in
$X$ and $\exp(C)$ is a compact subset of $\exp(X)$. Observing that
$$
l(K)\subset\cl_X(s(K))\subset\cl_X(s(\cup\KK))=C
$$
 for every $K\in\KK$ , we
see that $l(\KK)\subset\exp(C)$ and thus $l(\KK)$ is a precompact subset of
$\exp(X)$.
\smallskip

Next, assume conversely that the map $\exp(f)\colon\, \exp(X)\to\exp(Y)$ is subproper. Then there is a section $l\colon\, \exp(Y)\to\exp(X)$ of
$\exp(f)$ preserving precompact sets. For each $y\in Y$ let $s(y)$ be any
point of the compact subset $l(\{y\})$ of $X$. Since
$\exp(f)(l(\{y\}))=f(l(\{y\}))=\{y\}$, we get that the so-defined function
$s\colon\, Y\to X$ is a section of the map $f$. To verify that $s$ preserves
precompact subsets, fix any compact subset $K$ of $Y$. It follows that
$\KK=\{\{y\}\colon\, y\in K\}$ is a compact subset of $\exp(Y)$ and hence
$l(\KK)=\{l(\{y\})\colon\, y\in K\}$ is a precompact subset of $\exp(X)$ contained
in some compact subset $\C$ of $\exp(X)$. Then $\cup\C$ is a compact
subset of $X$ containing all sets $l(\{y\})$ for $y\in K$. Hence
$s(K)\subset\cup\C$ is a precompact subset of $X$.
\end{proof}

The constructions of the hyperspace $\exp$ is an example of a functorial construction on the
category $\Haus$ of Hausdorff  spaces and their continuous maps. The above
results suggest the following general

\begin{problem} Which functorial topological constructions do preserve
subproper maps?
\end{problem}

We shall  answer  this question for functors with compact continuous
support. For basic concepts of categorial topology we refer the reader to
\cite{TZ}. Let $F\colon\, \Haus\to\Haus$ be a functor on the category $\Haus$ of
Hausdorff spaces and let
$X$ be a Hausdorff space. We say that a point $a\in
F(X)$ is supported by a subset $K\subset X$ if $a\in F(e_K)(F(K))$, where
$e_K\colon\, K\to X$ stands for the natural inclusion. We say that a functor
$F\colon\, \Haus\to\Haus$ has {\em compact continuous support} if there is a natural
transformation $\supp\colon\, F\to\exp$ of the functor $F$ into the hyperspace
functor $\exp$ such that for every Hausdorff space $X$ the component
$\supp_X\colon\, F(X)\to\exp(X)$ is a continuous map such that each element $a\in
F(X)$ is supported by the compact set $\supp_X(a)\subset X$. We say that a
functor $F\colon\, \Haus\to\Haus$ {\em preserves surjective maps between compact
spaces} if for every surjective map $f\colon\, X\to Y$ between compact
Hausdorff spaces the spaces $F(X)$ and $F(Y)$ are compact and the
map $F(f)\colon\, F(X)\to F(Y)$ is surjective.

\begin{proposition}\label{1.20}
Suppose $F\colon\, \Haus\to \Haus$ is a functor
with compact continuous support that preserves
 surjective maps between
compact spaces. Then for every subproper map $f\colon\, X\to Y$
the map $F(f)\colon\, F(X)\to F(Y)$ is subproper.
\end{proposition}

\begin{proof} Let $f\colon\, X\to Y$ be a subproper map and
let $\supp\colon\, F\to \exp$ be the
 natural transformation such that for each Hausdorff
space $Z$ the component $\supp_Z\colon\, F(Z)\to\exp(Z)$ is a continuous map and
each $a\in F(Z)$ is supported by the compact set $\supp_Z(a)$. Let $s\colon\, Y\to
X$ be a section preserving precompact sets. Define a section $l\colon\, F(Y)\to
F(X)$ of the map $F(f)$ as follows. Given a point $a\in F(Y)$ let
$C=\supp_Y(a)$ and $K=\cl_X(s(C))$. Find a point $c\in F(C)$ such that
$F(e_C)(c)=a$ where $e_C\colon\, C\to Y$ is the identity inclusion. Since the
functor $F$ preserves surjective maps between compacta, there is a point
$b\in F(K)$ such that $F(f|K)(b)=c$, where $f|K\colon\, K\to C$ is the restriction
of $f$. Let finally $l(a)=F(e_K)(b)\in F(X)$ where
$e_K\colon\, K\to X$ is the identity embedding.

It follows that
\begin{multline*}
F(f)(l(a))=F(f)\circ F(e_K)(b)=F(f\circ e_K)(b)=F(e_C\circ
f|K)(b)
\\
=F(e_C)\circ F(f|K)(b)=F(e_C)(c)=a
\end{multline*}
 and thus the so-defined
function $l\colon\, F(Y)\to F(X)$ is a section of the map $F(f)$.

To show that
this section preserves precompact sets, fix any compact subset $\KK\subset
F(Y)$. By the continuity of the map $\supp_Y$, the image $\supp_Y(\KK)$ of
$\KK$ is a compact subset of the hyperspace $\exp(Y)$. Then the union
$A=\cup\supp_Y(\KK)$ is a compact subset of $Y$. Let $B=\cl_X(s(A))$.
Since $F$ preserves surjective maps between compacta, the space $F(B)$ is
compact and so is its image $F(e_B)(F(B))$ in $F(X)$ where $e_B\colon\, B\to X$
stands for the identity inclusion. It follows from the construction that
$l(a)\in F(e_B)(F(B))$ for each $a\in\KK$ which yields that $l(\KK)$ is a
precompact subset of $F(X)$.
\end{proof}

\begin{remark}
{\rm
There are many examples of functors satisfying the
requirements of Proposition~\ref{1.20}, see \cite[VII.\S1]{FF} or \cite[\S2.7]{TZ}. Among them
there are the functor $(\cdot)^n$ of finite power, the functor $SP^n$ of
symmetric power, the subfunctors $\exp_n$ of the hyperspace functor, etc.
}\end{remark}

We finish this section by establishing an ``upper semicontinuity''
property of $\cs^*$-continuous functions. Let us recall that the
upper limit $\overline{\lim}_{n\to\infty}F_n$ of a sequence
$(A_n)_{n\in\w}$ of subsets of a topological space $X$ is the set
of all points $x\in X$ such that any neighborhood $U_x$ of $X$
meets infinitely many sets $A_n$. We shall write that
$(A_n)\nearrow \uplim_{n\to\infty}A_n$ if each neighborhood of
$\uplim_{n\to\infty}A_n$ contains all but finitely many sets $A_n$.

\begin{proposition}\label{n2.8} Let $s:Y\to X$ be a $\cs^*$-continuous function and $(A_n)_{n\in\w}$ be a sequence of subsets of $Y$ such that $K=\uplim_{n\to\infty}A_n$ is metrizable compact and $(A_n)\nearrow K$. Then the upper limit $B=\uplim_{n\to\infty}s(A_n)$ is closed and bounded in $X$ and $\big(s(A_n)\big)\nearrow B$. If the space $X$ is $\mu$-complete, then $B$ is compact.
\end{proposition}

\begin{proof} It is clear that $B=\uplim_{n\to\infty}s(A_n)$ is closed. Assuming that this set is unbounded, find an infinite locally finite family $\U=\{U_k:k\in\w\}$ of open subsets of $X$ that intersect $B$. By induction construct an increasing number sequence $(n_k)_{k\in\w}$ and a sequence of points $y_k\in A_{n_k}$ such that $s(y_k)\in U_k$. It follows from the compactness and metrizability of $K$ and the convergence $(A_n)\nearrow K$ that the set $L=K\cup\{y_k:k\in\w\}$ is compact and metrizable. Consequently, the sequence $(y_k)$ has a convergent subsequence $(y_{k_i})$. By the $\cs^*$-continuity of $s$ the sequence $\big(s(y_{k_i})\big)$ has an accumulation point $x_\infty$ in $X$. Any neighborhood of this point meets infinitely many sets $U_{k_i}$ which contradicts the local finity of $\U$.

Therefore $B$ is closed and bounded. If $X$ is $\mu$-complete, then $B$ is compact by the definition of $\mu$-completeness.

Next, we show that $\big(s(A_n)\big)\nearrow B$. Assuming the converse, we can find a neighborhood $W$ of $B$ in $X$ such that the set $J=\{n\in\w:s(A_n)\not\subset W\}$ is infinite. Then we can construct a sequence $(y_n)_{n\in J}$  such that each $y_n\in A_n$ and $s(y_n)\notin W$. It follows from $(A_n)\nearrow K$ that the set $K\cup \{y_n:n\in J\}$ is compact and metrizable and the sequence $(y_n)_{n\in J}$ has a convergent subsequence $(y_{n_i})$. By the $\cs^*$-continuity of the section $s$, the sequence $\{s(y_{n_i})\}$ has an accumulating point $x_\infty$ in $X$. By definition, $x_\infty\in \uplim_{n\to\infty}s(A_n)=B$. On the other hand, $x_\infty\in\cl\{s(y_n):n\in J\}\subset X\setminus W\subset X\setminus B$, which is a contradiction.
\end{proof}

\section{$k^*$-Metrizable spaces}\label{s3}

It is well-known that the image of a metrizable space under a perfect map is metrizable \cite[[4.4.15]{En}. The images of metrizable spaces under (sub)proper need not be metrizable, which allows us to introduce two new classes of generalized metric spaces.

\begin{definition} A topological space $X$ is defined to be
{\em $k$-metrizable} \textup{(}resp. {\em $k^*$-metrizable}\textup{)} if $X$ is the image of a metrizable space $M$ under a proper \textup{(}resp. subproper\textup{)} map $\pi:M\to X$.
\end{definition}

In fact, the $k$-metrizability of a space $X$ is equivalent to the metrizability of the $k$-coreflexion of $X$. By the {\em $k$-coreflexion} \ $kX$ of a topological space $X$ we understand the set $X$ endowed with the strongest topology inducing the original topology on each compact subset $K\subset X$. It is clear that the identity map $i:kX\to X$ is continuous. Moreover, for any map $f:Z\to X$ from a $k$-space $Z$ the composition $i^{-1}\circ f:Z\to kX$ still is continuous. Thus $kX$ carries the weakest  $k$-space topology that is stronger than the original topology of $X$.
We recall that a topological space $X$ is a {\em $k$-space} if a subset $F\subset X$ is closed if and only if for any compact set $K\subset X$ the intersection $K\cap F$ is closed in $K$.

$k$-Metrizable spaces admit a simple characterization.

\begin{proposition}\label{n3.2} A topological space $X$ is $k$-metrizable if and only if its $k$-coreflexion $kX$ is metrizable.
\end{proposition}

\begin{proof} Since the identity map $i:kX\to X$ is proper, the space $X$ is $k$-metrizable if $kX$ is metrizable. Now assume conversely that $X$ is $k$-metrizable and find a proper map $\pi:M\to X$ from a metrizable space $M$. It follows that the composition $i^{-1}\circ \pi:M\to kX$ is proper and hence perfect because $kX$ is a $k$-space, see \cite[3.7.17]{En}. Then the space $kX$ is metrizable, being a perfect image of a metrizable space $M$, see \cite[4.4.15]{En}.
\end{proof}

The $k^*$-metrizability also can be reduced to considering the $k$-coreflexion.

\begin{proposition}\label{n3.3} A space $X$ is $k^*$-metrizable if and only if its $k$-coreflexion $kX$ is $k^*$-metrizable.
\end{proposition}

\begin{proof} Let $i:kX\to X$ denote the identity map. Assuming that $X$ is $k^*$-metrizable, find a subproper map $\pi:M\to X$ for a metrizable space $M$. Since $M$ is a $k$-space, the composition $i^{-1}\circ \pi:M\to kX$ is continuous and subproper, which implies that $kX$ is $k^*$-metrizable.

Assuming conversely that $kX$ is $k^*$-metrizable, find a subproper map $\pi:M\to kX$ from a metrizable space and consider the composition $i\circ \pi:M\to X$ which is a subproper map because the function $i^{-1}:X\to kX$ preserves compacta.
\end{proof}

It turns out that sequentially regular $k^*$-metrizable spaces carry a nice metric.
We define a topological space to be {\em sequentially regular} if for each point $x\in X$ and a neighborhood $U\subset X$ of $x$ there is another neighborhood $V\subset X$ of $x$ such that $\cl_1(V)\subset U$, where $\cl_1(V)$ is the set of the limit points of sequences $\{x_n\}_{n\in\w}\subset V$, convergent in $X$. It is clear that each regular space is sequentially regular.

\begin{theorem}\label{n3.4}
 A (sequentially regular) space $X$ is $k^*$-metrizable (if and) only if $X$ has a metric $\rho$ such that
\begin{enumerate}
\item[(a)] each compact subset $K\subset X$ is totally bounded with respect to the metric $\rho$;
\item[(b)] each $\rho$-convergent sequence converges in $X$;
\item[(c)] a $\rho$-Cauchy sequence converges in $X$ if and only if it has a convergent subsequence in $X$.
\end{enumerate}
\end{theorem}

\begin{proof} To prove the ``only if'' part, assume that $X$ is the image of a metric space $(M,d)$ under a subproper map $\pi:M\to X$. Find a section $s:X\to M$ of $\pi$ that preserves precompact sets and consider the metric $\rho(x,x')=d(s(x),s(x'))$ on the space $X$, induced by the metric $d$ of $M$. It is easy to check that the so-defined metric $\rho$ on $X$ has the three properties indicated in the theorem.
\smallskip

Now assume conversely that a sequentially regular space $X$ admits metric $\rho$ with properties (a)--(c). Consider the metric space $Z=(X,\rho)$ and its completion $\tilde Z$. The  condition (b) implies that the identity map $\pi:Z\to X$ is continuous. Let $M$ be the set of all points $z\in\tilde Z$ such that for any sequence $(z_n)\subset Z$ convergent to $z$ the sequence $(\pi(z_n))$ converges to some point $x$. Letting $\bar \pi(z)=x$ one defines a map $\bar \pi:M\to X$ extending the identity map $\pi:Z\to X$.
The sequential regularity of $X$ can be used to show that the map $\bar \pi:M\to X$ is continuous. We claim that $\bar \pi$ is subproper. Fix any compact subset $K\subset X$. This set is totally bounded in the metric space $(Z,\rho)$ and thus has compact closure $\overline{K}$ in $\tilde Z$. It remains to check that $\overline{K}\subset M$. Assuming that this is not so, find a point $z\in \overline{K}\setminus M$. The compact space $K$, being the continuous image of a totally bounded (and hence separable) metric space, has countable network and thus is metrizable.
Let $(x_{2n})_{n\in\w}\subset K$ be a sequence convergent to the point $z$ in the metric space $\tilde Z$. The compactness of $K$ implies that this sequence contains a subsequence convergent to some point $x$ in $K$.
The condition $z\notin M$ implies the existence of a sequence $(x_{2n+1})_{n\in\w}\subset Z$ convergent to $z$ in $\tilde Z$ whose image $(\pi(x_{2n+1}))_{n\in\w}$ diverges in $X$. Then the sequence $(x_n)_{n\in\w}$ is $\rho$-Cauchy and contains a subsequence convergent to $x$ in $X$. The property (c) of the metric $\rho$ implies that the sequence $(x_n)_{n\in\w}$ converges to $x$ in the space $X$ which is not possible because $(x_{2n+1})_{n\in\w}$ does not converge to $x$. This contradiction shows that the set $\pi^{-1}(K)=Z\cap \bar \pi^{-1}(K)\subset\overline{K}\subset M$ is precompact in $M$, which means that $X$ is $k^*$-metrizable, being the image of the metrizable space $M$ under the subproper map $\bar \pi:M\to X$.
\end{proof}

Let us remark some obvious properties of $k^*$-metrizable spaces.

\begin{theorem}\label{3.3}
{\rm 1.} A subspace of a $k^*$-metrizable space is $k^*$-metrizable.

{\rm 2.} A countable product of $k^*$-metrizable spaces is $k^*$-metrizable.

{\rm 3.} The topological sum of arbitrary family of
$k^*$-metrizable spaces is $k^*$-metrizable.

{\rm 4.} The image of a $k^*$-metrizable space under a subproper map is $k^*$-metrizable.

{\rm 5.} Each sequentially compact subset of a $k^*$-metrizable space is metrizable.

{\rm 6.} A $k$-metrizable space is sequential if and only if it is a $k$-space.
\end{theorem}

Therefore the $k^*$-metrizability is preserved by subproper maps. In the class of $k$-spaces the same is true for closed maps.

\begin{proposition}\label{n3.6} If $f:X\to Y$ is a closed map from a $k^*$-metrizable $k$-space $X$, then $Y$ is $k^*$-metrizable as well.
\end{proposition}

\begin{proof} Let $\pi:M\to X$ be a subproper map of a metrizable space $M$ onto the $k^*$-metrizable $k$-space $X$. The proposition will be proven as soon as we check that the composition $f\circ \pi:M\to Y$ is subproper. Let $s:Y\to X$ be any section of the map $\pi$ and $\sigma:X\to M$ be a precompact-preserving section of $\pi$. We claim that $\sigma\circ s:Y\to M$ also preserves precompact sets. Fix any compact subset $K\subset Y$. Assuming that $\sigma\circ s(K)$ is not precompact in $M$, we may find a subset $D\subset \sigma\circ s(K)$ which is closed and discrete in $M$. Then the image $\pi(D)$ is closed and discrete in the $k$-space $X$ because $\pi(D)$ has finite intersection with each compact subset of $X$. Since $f|\pi(D)$ is a closed injective map, the image $f\circ \pi(D)\subset K$ is closed and discrete in $Y$ which is not possible as $K$ is compact.
\end{proof}

It is clear that each sequentially compact $k^*$-metrizable space $X$ is compact and metrizable.
In fact, a bit more general result is true. We recall that a space $X$ is called {\em countably compact} if each sequence $(x_n)$ in $X$ has an accumulating point $x_\infty\in X$.

\begin{proposition}\label{n3.7} A $k^*$-metrizable space $X$ is compact and metrizable iff the $k$-coreflexion $kX$ of $X$ is countably compact.
\end{proposition}

\begin{proof}
Let
$\pi: M\to X$ be a subproper map from a metrizable space $M$ onto $X$
and $s\colon\, X\to M$ is a section of $\pi$ preserving precompact sets. Suppose that $kX$ is countably compact. We claim
that the closure $Z=\overline{s(X)}$ of $s(X)$ in $M$ is compact. Assuming
that this is not true, we can find an infinite subset $D\subset s(X)$ that is closed and discrete in $X$. Then the image $\pi(D)$ has finite intersection with each compact subset of $X$ and hence closed and discrete in the $k$-coreflexion $kX$, which is not possible as $kX$ is countably compact. Therefore $Z$ is compact metrizable space and $X$, being the continuous image of $Z$, is compact and metrizable too.
\end{proof}

\begin{remark}\label{frolik}
{\rm
Surprisingly, but a countably compact $k^*$-metrizable space
 need not be metrizable.
It was shown by Frolik
\cite{Fr} (see also \cite[3.23]{Gru}) that $\beta \IN$ contains a
countably compact subspace $X\supset\IN$ of size $2^{\mathfrak c}$ and (network) weight $\mathfrak c$ such that
all compact subsets of $X$ and $\beta\IN\setminus X$ are finite. It
follows that the $k$-coreflexion of $X$ is discrete and hence both $kX$ and $X$ are $k^*$-metrizable. On the other hand $X$ is a non-metrizable  separable, countably compact space. 
}\end{remark}

Next, we show that a
space $X$ is $k^*$-metrizable if $X$ admits a special cover by
$k^*$-metrizable spaces.

We define a subset $A$ of a topological space $X$ to be {\em
$k$-closed} if for each compact subset $K\subset X$ the
intersection $A\cap K$ is closed in $K$. This is equivalent to saying that $A$ is closed in the $k$-coreflexion $kX$ of $X$.

 A collection $\mathcal F$
of subsets of a space $X$ is defined to be {\em compact-finite} if
for every compact subset $K\subset X$ the set $\{F\in\FF:F\cap
K\ne\emptyset\}$ is finite. It is clear that each locally finite
collection is compact-finite and each compact-finite collection is
point-finite. In Fr\'echet-Urysohn spaces each compact-finite family is locally-finite.

\begin{theorem}\label{3.3a} Let $X$ be a topological space. Then
\begin{enumerate}
\item[1)] $X$ is $k^*$-metrizable if and only if
there is a compact finite cover $\C$ of $X$ such that each
$C\in\C$ lies in a $k$-closed $k^*$-metrizable subspace of
$X$;
\item[2)] $X$ is $k^*$-metrizable if and only if
there is a countable collection $(X_n)_{n=1}^\infty$ of $k$-closed $k^*$-metrizable subspaces of $X$ such that every
convergent sequence $S\subset X$ is contained in some $X_n$.
\end{enumerate}
\end{theorem}

\begin{proof}
1. Let $\C$ be a compact-finite cover of $X$ such that for each
set $C\in\C$ there is a $k$-closed subspace $\tilde C\supset
C$ of $X$ such that $\tilde C$ is $k^*$-metrizable and hence
$\tilde C$ is the image of a metrizable space $M_C$ under a
continuous map $\pi_C\colon\, M_C\to\tilde C$ admitting a section
$s_C\colon\, \tilde C\to M_C$ that preserves precompact sets. It
is clear that the topological sum $M=\oplus_ {C\in\C}M_C$ is a
metrizable space and the map $\pi=\bigcup_{C\in\C}\pi_C\colon\,
M\to\bigcup_{C\in\C}\tilde C=X$ is continuous. Given a point $x\in
X$, find any subset $C(x)\in\C$ containing $x$ and let
$s(x)=s_{C(x)}(x)\in M_{C(x)}\subset M$. We claim that the section $s\colon\,
X\to M$ of $\pi$ preserves precompact sets. Given a compact subset $K\subset X$ observe that the family $\C'=\{C\in\C\colon\, C\cap
K\ne\emptyset\}$ is finite and thus $s(K)$ lies in the precompact
subset $\bigcup_{C\in\C'}{s_C(K\cap\tilde C)}$ of $M$.
\smallskip

2. Suppose that $\{X_n\}_{n\in\IN}$ is a countable collection of
$k$-closed $k^*$-metrizable subspaces of $X$ such that each
convergent sequence $K\subset X$ lies in some $X_n$. For every
$n\in\IN$ let $Y_n=X_n\setminus \bigcup_{i<n}X_i$. The
$k^*$-metrizability of $X$ will follow from the preceding item as
soon as we check that the family $\{Y_n\}_{n\in\IN}$ is
compact-finite. Take any compact subset $K\subset X$. Assuming
that $K\not\subset X_n$ for all $n$, construct a sequence $(x_n)$
in $K$ such that $x_n\notin X_n$ for all $n$. We claim that the
sequence $(x_n)$ has a convergent subsequence. Since each set
$X_n$ is $k$-closed, the intersection $K\cap X_n$, being a
compact subset of the $k^*$-space $X_n$ is metrizable. Then
$K=\bigcup_n K\cap X_n$, being the countable union of metrizable
compacta, has countable network of the topology and hence is
metrizable and sequentially compact. So we can find a convergent
subsequence $(x_{n_k})$ of $(x_n)$ lying in no subset $X_n$, which
contradicts the choice of the sets $X_n$.
\end{proof}

We saw in Theorem~\ref{3.3}(2) that the countable product of
$k^*$-metrizable spaces is $k^*$-metrizable. The same is true for
the box products, that is Cartesian products $\prod_{i\in \I}X_i$
endowed with the box-product topology generated by sets of the
form $\prod_{i\in \I}U_i$ where $U_i\subset X_i$ are open sets.

\begin{theorem}\label{n3.10} Let $\{X_i:i\in\I\}$ be a family of $k^*$-metrizable spaces. Then the box-product $\square_{i\in \I}X_i$ is $k^*$-metrizable.
\end{theorem}

\begin{proof} Elements of the box-product $\square_{i\in \I}X_i$ can be thought as functions
$x:\I\to \bigcup_{i\in \I}X_i$ such that $x(i)\in X_i$ for $i\in\I$. For two functions $x,y\in\square_{i\in\I}X_i$ let $$\{x\ne y\}=\{i\in I:x(i)\ne y(i)\}\mbox{  and }\sigma(x)=\{y\in\square_{i\in \I}X_i:|\{y\ne x\}|<\w\}.$$
Let us show that each set $\sigma(x)$ is closed
in $\square_{i\in \I}X_i$. Indeed,  any point $y\in\square_{i\in\I}X_i\;\setminus \sigma(x) $ can be separated from $\sigma(x)$ by the box-neighborhood $\prod_{i\in\I}U_i$ of $y$, where $U_i=X_i\setminus \{x(i)\}$ if $i\in\{x\ne y\}$ and $U_i=X_i$ otherwise.
Observe also that for two points
$x,x'\in \square_{i\in I}X_i$ the sets $\sigma(x)$ and
$\sigma(x')$ either coincide or are disjoint. Then
$\Sigma=\{\sigma(x):x\in \square_{i\in I}X_i\}$ is a disjoint closed cover
of the box-product $\square_{i\in I}X_i$. We claim that $\Sigma$
is a compact-finite cover of $X$ consisting of closed
$k^*$-metrizable subspaces of $\square_{i\in I}X_i$.
\smallskip

First we check that the cover $\Sigma$ is compact-finite.
 Take any
compact set $K$ and assume that it intersects infinitely many sets
$\sigma(x_n)$, where $x_n\notin\sigma(x_m)$ for all $n\ne m$. By
the compactness of $K$ the sequence $(x_n)$ has an accumulation point
$x_\infty$ in $K$. This point $x_\infty$ can belong to at most one
set $\sigma(x_n)$. Without loss of generality, $x_\infty\notin
\sigma(x_n)$ for all $n$, which means that the sets $\{x_n\ne x_\infty\}$ are infinite. This allows us to construct an injective
sequence $(i_n)$ in $\I$ such that $x_\infty(i_n)\ne x_n(i_n)$ for
all $n$. For every index $i\in \I$ let $$U_i=
\begin{cases} X_{i_n}\setminus \{x_n(i_n)\}&\mbox{ if $i=i_n$ for some $n$}\\
X_i&\mbox{otherwise} \end{cases}$$ Then $\prod_{i\in I}U_i$ is an
open neighborhood of $x_\infty$ in $\square_{i\in I} X_i$ that
contains no point of the sequence $(x_n)$, which is impossible as
$x_\infty$ is an accumulation point of $(x_n)$. Therefore $\Sigma$ is a
compact-finite cover of $\square_{i\in I}X_i$.
\smallskip

Next, we check that each set $\sigma(x)\in\Sigma$ is
$k^*$-metrizable.  Given a finite subset $F\subset \I$ let $$
\sigma_F(x)=\{y\in\sigma(x):\{x\ne y\}=F\}\mbox{ and } \Pi_F(x)=\{y\in\sigma(x):\{x\ne y\}\subset F\}. $$
 It is clear that $\sigma_F(x)\subset\Pi_F(x)$ and $\Pi_F(x)$ is
a closed subspace of $\sigma(x)$ homeomorphic to the finite
product $\prod_{i\in F}X_i$ which $k^*$-metrizable according to
Theorem~\ref{3.3}(2). Repeating the preceding argument, we may show for each compact subset $K\subset\sigma(x)$ the set $F=\bigcup_{y\in K}\{y\ne x\}$ is finite and hence $K\subset\Pi_F(x)$. This implies that the cover $\{\sigma_F(x)\colon\, F\subset\I,\;|F|<\infty\}$ of
$\sigma(x)$ is compact-finite. Since each $\sigma_F(x)$ lies in
the closed $k^*$-metrizable subspace $\Pi_F(x)$ of $\sigma(x)$,
we can apply Theorem \ref{3.3a}(1) to conclude that the space $\sigma(x)$ is $k^*$-metrizable.
\smallskip

By Theorem~\ref{n3.10}, the box-product $\square_{i\in\I}X_i$ is
$k^*$-metrizable being the union of a compact-finite cover
consisting of closed $k^*$-metrizable subspaces.
\end{proof}

\begin{remark}
{\rm The $\sigma$-product $\sigma(x)\subset\prod_{i\in\I}X_i$ of
metrizable spaces $X_i$ endowed with the Tychonoff
product-topology is $k^*$-metrizable if and only if it is
metrizable. This follows from the fact that any nonmetrizable
space $\sigma(x)$ contains a compact subspace homeomorphic to the
one-point compactification of an uncountable discrete space. Then
$\sigma(x)$ cannot be the image of a metrizable space under a
compact-covering map. }\end{remark}

The $k^*$-metrizability is also preserved by the construction of
hyperspace. This follows from Proposition~\ref{1.11}.

\begin{proposition}\label{n3.12} The hyperspace $\exp(X)$ of any
$k^*$-metrizable space is $k^*$-metrizable.
\end{proposition}

Applying Theorem~\ref{1.16}, we get a more general

\begin{proposition}\label{n3.13} Let $F\colon\, \Haus\to\Haus$ be a functor with compact
continuous support that preserves metrizable  spaces and
surjective maps between compact spaces. Then for every
$k^*$-metrizable space  the space $F(X)$ is $k^*$-metrizable.
\end{proposition}

\section{$\cs^*$-Metrizable spaces}\label{s4}

In fact, the $k^*$-metrizability decomposes into two weaker properties: the $\cs^*$-metrizability and the sequential compactness of all compact subsets.

\begin{definition} A topological space $X$ is defined to be
{\em $\cs^*$-metrizable} \textup{(}resp. {\em $\cs$-metrizable}\textup{)} if $X$ is the image of a metrizable space $M$ under a map $\pi:M\to X$ having a $\cs^*$-continuous \textup{(}resp. sequentially continuous\textup{)} section.
\end{definition}

The $k$-metrizability (resp. $k^*$-metrizability) can be characterized via $\cs$-metrizability (resp. $\cs^*$-metrizability) as follows.

\begin{proposition}\label{n4.2} A topological space $X$ is $k^*$-metrizable \textup{(}resp. $k$-metrizable\textup{)} if and only if $X$ is $\cs^*$-metrizable \textup{(}resp. $\cs$-metrizable\textup{)} and all compact subsets of $X$ are sequentially compact.
\end{proposition}

\begin{proof} Assuming that a space $X$ is $k^*$-metrizable, find a subproper map $\pi:M\to X$ from a metrizable space. The space $M$, being metrizable is $\mu$-complete. Then Proposition~\ref{n1.3} implies that $\pi$ has a $\cs^*$-continuous section, which means that $\pi$ is $\cs^*$-metrizable. Since $\pi$ is compact-covering, all compact subsets of $X$ are metrizable and hence sequentially compact.

Now assume conversely that $X$ is $\cs^*$-metrizable and all compact subsets of $X$ are sequentially compact. Find a map $\pi:M\to X$ possessing a $\cs^*$-continuous section. By Proposition~\ref{n1.3} this map is subproper and thus $X$ is $k^*$-metrizable.

The equivalence of the $k$-metrizability to the $\cs$-metrizability plus the sequential compactness of all compact subsets can be proved by analogy.
\end{proof}

By analogy with the $k$-metrizability (which is equivalent to the metrizability of the $k$-coreflexion), the $\cs$-metrizability of a space $X$ is equivalent to the metrizability of the sequential coreflexion $sX$ of $X$.

We recall that a topological space $X$ is {\em sequential} if each sequentially closed  subset of $X$ is closed. A subset $F\subset X$ is {\em sequentially closed} in $X$ if it contains the limits of all sequences $(x_n)\subset F$ convergent in $X$.

By the {\em sequential coreflexion} $sX$ of a space $X$ we understand $X$ endowed with the topology consisting of sequentially open subsets of $X$. A subset $U\subset X$ is called {\em sequentially open\/} if its complement $X\setminus U$ is sequentially closed  in $X$. It is clear that the identity maps $sX\to kX\to X$ are continuous and all these maps are homeomorphisms if the space $X$ is sequential. In fact, the identity map $sX\to X$ always is a sequential homeomorphism. A bijective function $h\colon\, X\to Y$ between topological spaces is called a
{\em sequential homeomorphism} if both $h$ and $h^{-1}$ are sequentially
continuous. Topological spaces $X,Y$ are {\em sequentially homeomorphic}
if there is a sequential homeomorphism $h\colon\, X\to Y$.

The following theorem is  a counterpart of Propositions~\ref{n3.2}, \ref{n3.3} and Theorem~\ref{n3.4} and can be proved by analogy.

\begin{theorem}\label{n4.3}
\begin{enumerate}
\item
A topological space $X$ is $\cs$-metrizable iff its sequential coreflexion $sX$ is metrizable.
\item A topological space $X$ is $\cs^*$-metrizable iff its sequential coreflexion $sX$ is a $\cs^*$-metrizable iff $sX$ is $k^*$-metrizable.
\item
 A regular space $X$ is $\cs^*$-metrizable if and only if $X$ has a metric $\rho$ such that
\begin{enumerate}
\item each $\rho$-convergent sequence converges in $X$;
\item each convergent sequence in $X$ contains a $\rho$-Cauchy subsequence;
\item a $\rho$-Cauchy sequence converges in $X$ if and only if it has a convergent subsequence in $X$.
\end{enumerate}
\end{enumerate}
\end{theorem}

Properties of $\cs^*$-metrizable spaces are analogous to those of
$k^*$-metrizable spaces. We define a family $\mathcal F$ of
subsets {\em cs-finite} if each convergent sequence meets only
finitely many sets $F\in\FF$.

\begin{theorem}\label{n4.4}
{\rm 1.} A subspace of a $\cs^*$-metrizable space is $\cs^*$-metrizable.

{\rm 2.} The topological sum of arbitrary family of
$\cs^*$-metrizable spaces is $\cs^*$-metrizable.

{\rm 3.} A space $X$ is $\cs^*$-metrizable if $X$ is the union of
a cs-finite cover $\C$ such that each $C\in\C$ lies in a
sequentially closed $\cs^*$-metrizable subspace of $X$.

{\rm 4.}  A space $X$ is $\cs^*$-metrizable if $X$ is the union
$X=\bigcup_{n\in\w}X_n$ of a sequence of sequentially closed
$\cs^*$-metrizable subspaces such that each convergent sequence
lies in some $X_n$.

{\rm 5.} The countable product of $\cs^*$-metrizable spaces is
$\cs^*$-metrizable.

{\rm 6.} The box-product $\square_{i\in \I}X_i$ of arbitrary
family of $\cs^*$-metrizable spaces is $\cs^*$-metrizable.

{\rm 7.} Each sequentially compact subset of a $\cs^*$-metrizable
space is metrizable.

{\rm 8.} The image of a sequential $\cs^*$-metrizable space under
a closed or subproper map is $\cs^*$-metrizable.

{\rm 10.} A $\cs^*$-metrizable space $X$ is compact and metrizable
if its sequential coreflexion $sX$ is countably compact.
\end{theorem}

Thus the $k$-, $k^*$-, $\cs$-, and $cs^*$-metrizability properties relate as follows:
$$
\begin{CD}
\mbox{metrizable}@ >>>\mbox{$k$-metrizable}@>>>\mbox{$k^*$-metrizable}\\
@. @VVV @VVV\\
@. \mbox{$\cs$-metrizable}@>>>\mbox{$\cs^*$-metrizable}
\end{CD}
$$
\smallskip

\begin{remark} None arrow in this diagram can be reversed:
\begin{itemize}
\item[(a)] the Banach space $(l^1,\mathrm{weak})$ endowed with the weak topology is an example of a $k$-metrizable space that is not metrizable;
\item[(b)] the inductive limit $\IR^\infty=\underset{\longrightarrow}{\lim}\,\IR^n$ is $k^*$-metrizable but not $k$-metrizable, see Theorem~\ref{4.1};
\item[(c)] the Stone-\v Cech compactification $\beta\w$ of $\w$ contains no non-trivial convergent sequences and hence is $\cs$-metrizable but fails to be $k^*$-metrizable. \item[(d)] The product $\beta\w\times\IR^\infty$ is $\cs^*$-metrizable but is neither $\cs$-metrizable nor $k^*$-metrizable.
\end{itemize}
\end{remark}

\section{Cardinal invariants of $k^*$-metrizable spaces}\label{s5}

In this section we calculate some cardinal invariants of $k^*$-metrizable spaces.
We recall that for a topological space $X$
\begin{itemize}
\item $w(X)$, the {\em weight} of $X$, is the smallest size of a base of the topology of $X$;
\item $nw(X)$, the {\em network weight} of $X$, is the smallest size of a network $\mathcal N$ for $X$ ($\mathcal N$ is a {\em network} if for any open set $U\subset X$ and a point $x\in U$ there is $N\in\mathcal N$ with $x\in N\subset U$);
\item $knw(X)$, the {\em $k$-network weight} of $X$, is the smallest size of a $k$-network $\mathcal N$ ($\mathcal N$ is a {\em network} if for any open set $U\subset X$ and a compact set $K\subset U$ there is a finite subfamily $\mathcal F\subset\mathcal N$ with $K\subset\cup\FF\subset U$);
\item $l(X)$, the {\em Lindel\"of number}, is the smallest cardinal $\kappa$ such that each open cover $\U$ of $X$ has a subcover $\V\subset\U$ of size $|\V|\le\kappa$;
\item $ml(X)$, the {\em meta-Lindel\"of number}, is the smallest cardinal $\kappa$ such that each open cover of $X$ has an open refinement $\U$ such that $|\{U\in\U:x\in U\}|\le\kappa$ for every point $x\in X$;
\item $s(X)=\sup\{|D|:D$ is a discrete subspace of $X\}$ is the {\em spread} of $X$;
\item $\ext(X)=\sup\{|D|:D$ is a closed discrete subspace of $X\}$ is the {\em extent} of $X$;
\item $d(X)$, the {\em density} of $X$, is the smallest size of a dense set in $X$.
\end{itemize}

For any metrizable space $X$ the cardinals $w(X)$, $knw(X)$, $nw(X)$, $l(X)$, $s(X)$, $\ext(X)$, $d(X)$ coincide (see \cite[4.1.15]{En}) while for any topological space $X$ we have the inequalities:
$$\ext(X)\le s(X),l(X)\le nw(X)\le knw(X)\le w(X)$$
and $$d(X)\le nw(X)\mbox{ and }ml(X)\le l(X)\le ml(X)\cdot d(X).$$


In addition to the density $d(X)$ we consider its sequential versions $d_\alpha(X)$ dependent on the ordinal parameter $\alpha$.
By the {\em sequential closure} of a set $A$ in a topological space $X$ we understand the set $\cl_1(A)$ consisting of the limits of all sequences $(a_n)\subset A$ that converge in $X$. Next, for any ordinal $\alpha$ define the {\em $\alpha$-th sequential closure $\cl_\alpha(A)$} of $A$ by transfinite induction: $\cl_{\alpha}(A)=\cl_1(\cl_\beta(A))$ if $\alpha=\beta+1$ is a successor ordinal and $\cl_{\alpha}(A)=\bigcup_{\beta<\alpha}\cl_\beta(A)$ if $\alpha$ is a limit ordinal. It is well-known (and easy to see) that $\cl_\alpha(A)=\cl_{\w_1}(A)$ for all ordinals $\alpha\ge\w_1$ and $\cl_{\w_1}(A)$ coincides with the closure of $A$ in the sequential coreflexion $sX$ of $X$.

The {\em sequential order} $so(X)$ of a topological space $X$ is the smallest ordinal $\alpha$ such that $\cl_\alpha(A)=\cl_{\omega_1}(A)$ for any subset $A\subset X$. Observe that a sequential space $X$ is Fr\'echet-Urysohn if and only if $so(X)\le 1$.

By the {\em $\alpha$-th sequential density} $d_\alpha(X)$ of a space $X$ we understand the smallest size $|D|$ of a subset $D\subset X$ whose $\alpha$-th sequential closure $\cl_\alpha(D)=X$. It is clear that $d_\alpha(X)\ge d_\beta(X)$ for any ordinals $\alpha\le \beta$ and $d_{so(X)}(X)=d_{\w_1}(X)$ equals the density $d(sX)$ of the sequential coreflexion of $X$.

Let also $d_{<\w_1}(X)=\min_{\alpha<\w_1}d_\alpha(X)$. Note that the cardinal $d_{<\w_1}(X)$ can be strictly larger than $d_{\w_1}(X)$. A suitable counterexample yields the space $B(\IR)\subset\IR^\IR$ of all Borel-measurable functions. By the classical Lebesgue-Hausdorff Theorem (see \cite[11.6]{Ke}), $B(\IR)=\cl_{\w_1}(C(\IR))$ where $C(\IR)$ stands for the set of all continuous functions and hence $d_{\w_1}(B(\IR))=\aleph_0$. On the other hand, it can be shown that $d_{<\w_1}(B(X))>d_{\w_1}(B(X))$.

Since a space $X$ and its sequential coreflexion $sX$ have the same convergent sequences, the following equality holds.

\begin{proposition}\label{n5.1} For any space $X$ and any ordinal $\alpha$ we get $d_\alpha(X)=d_\alpha(kX)=d_\alpha(sX)$.
\end{proposition}

It turns out that $\cs^*$-continuous sections of continuous maps do not increase the cardinal invariant $d_\alpha$ for $\alpha<\w_1$.

\begin{proposition}\label{n5.2} Let $\pi:M\to X$ be a map from a metrizable space $M$ and $s:X\to M$ be a $\cs^*$-continuous section of $\pi$. Then for any subset $A\subset X$ and any ordinal $\alpha$ we get
\begin{enumerate}
\item $d\big(s(\cl_1(A))\big)\le d(s(A))$;
\item $d\big(s(\cl_\alpha(A))\big)\le\max\{|\alpha|,|A|\}$;
\item $d(s(X))\le \min\{d_{<\omega_1}(X), \so(X)\cdot d_{\w_1}(X)\}$.
\end{enumerate}
\end{proposition}

\begin{proof} 1. Assuming that $d\big(s(\cl_1(A))\big)>d\big(s(A)\big)$, find a subset $D\subset s(\cl_1(A))$ of size $|D|>d(s(A))$ that is closed and discrete in $M$.

For every point $y\in D$ use the inclusion $s^{-1}(y)\in\cl_1(A)$ to find a sequence $(a^y_n)\subset A$ convergent to the point $x_y=s^{-1}(y)=\pi(y)$. By the $\cs^*$-continuity of $s$ the sequence $(s(a^y_n))$ has a limit point $z_y\in \cl(s(A))$ in $M$. By the continuity of $\pi$ the point $z_y$ projects into $x_y$. The set $Z=\{z_y:y\in D\}$ has size $|Z|=|D|>d(s(A))=d(\overline{s(A)})$ and hence contains a convergent sequence $S\subset Z$ whose image $\pi(S)$ is a convergent sequence in $\pi(Z)=\pi(D)$. Then the $\cs^*$-continuity of $s$ implies that the sequence $s\circ \pi(S)\subset s\circ \pi(D)=D$ has a limit point in $M$ which is not possible as $D$ is closed and discrete.
\medskip

2. The second item can be proved by transfinite induction. For $\alpha=1$ it follows from the first item. Assume that for some ordinal $\beta\le\w_1$ the inequality $d\big(s(\cl_\alpha(A))\big)\le \max\{|\alpha|,|A|\}$ has been proved for all ordinals $\alpha<\beta$. If $\beta=\alpha+1$ is a successor ordinal, then $\cl_\beta(A)=\cl_1(\cl_\alpha(A))$ and by the first item $d(s(\cl_\beta(A)))\le d(s(\cl_\alpha(A)))\le\max\{|\alpha|,|A|\}\le\max\{|\beta|,|A|\}$ by the inductive hypothesis. If $\beta $ is a limit ordinal, then $\cl_\beta(A)=\bigcup_{\alpha<\beta}\cl_\alpha(A)$ and $$d\big(s(\cl_\beta(A))\big)=d\big(\bigcup_{\alpha<\beta}s(\cl_\alpha(A))\big)\le
\sum_{\alpha<\beta}d\big(s(\cl_\alpha(A))\big)\le \sum_{\alpha<\beta}\max\{|\alpha|,|A|\}\le\max\{|\beta|,|A|\}.$$
\medskip

3. The inequality $d(s(X))\le d_{<\w_1}(X)$ follows from the preceding item. To show that $d(s(X))\le  so(X)\cdot d_{\w_1}(X)$, consider two cases. If $so(X)<\w_1$, then $d_{\w_1}(X)=d_{<\w_1}(X)$ and hence $d(s(X))\le d_{<\w_1}(X)\le so(X)\cdot d_{\w_1}(X)$.
If $so(X)=\w_1$, then $d(s(X))\le \max\{\aleph_1,d_{\w_1}(X)\}\le so(X)\cdot d_{\w_1}(X)$ by item (2).
\end{proof}

Cardinal characteristics of $k^*$-metrizable $k$-spaces are evaluated in the following

\begin{theorem}\label{n5.3} If $X$ is a $k^*$-metrizable $k$-space, then $$
d(X)\le knw(X)=nw(X)=s(X)=\ext(X)=l(X)=ml(X)\cdot d(X)=d_1(X)\le \so(X)\cdot d(X).
$$
\end{theorem}

\begin{proof} Assuming that $X$ is a $k^*$-metrizable $k$-space, find a subproper map $\pi:M\to X$ of a metrizable space $M$ onto $X$ and take a section $\sigma:X\to M$ of $\pi$ that preserves precompact sets. Without loss of generality, $\sigma(X)$ is dense in $M$ and hence
$$w(M)=d(M)=d(\sigma(X))\le\min\{d_{<\w_1}(X),\so(X)\cdot d_{\w_1}(X)\}$$ by Proposition~\ref{n5.2}. The continuity of the map $\pi$ and the metrizability of $M$ imply the inequalities
$$d(X)=d_{\w_1}(X)\le d_{<\w_1}(X)\le d_1(X)\le d_1(M)=d(M)=w(M).$$ Therefore, $$w(M)=d_{<\w_1}(X)=d_1(X)\le \so(X)\cdot d_{\w_1}(X)=\so(X)\cdot d(X).$$
The latter equality holds because $X$, being a $k^*$-metrizable $k$-space, is sequential, which yields $d(X)=d_{\w_1}(X)$.

Next, we show that $w(M)\le \ext(X)$. Otherwise we may take any $\sigma$-discrete base $\mathcal B$ of the topology of $M$ and find a discrete family $\mathcal U\subset \mathcal B$ of size $|\U|> \ext(X)$. Using the density of $\sigma(X)$ in $M$, for every $U\in\U$ pick a point $x_U\in X$ with $\sigma(x_U)\in U$. Since $\sigma$ preserves precompact sets, the  set $D=\{x_U:U\in\U\}$ has finite intersection with compact subsets of $X$ and thus is closed and discrete in the $k$-space $X$. Consequently $\ext(X)\ge|D|>\ext(X)$, which is a contradiction. This proves the inequality $w(M)\le \ext(X)$.

Then for any base $\mathcal B$ of the topology of $M$ with size $|\mathcal B|=w(M)$ the family $\mathcal N=\{f(B):B\in\mathcal B\}$ is a $k$-network for  $X$ and hence $knw(X)\le |\mathcal N|\le|\mathcal B|=w(M)\le \ext(X)$. This inequality combined with the trivial inequalities $\ext(X)\le knw(X)$ and $ml(X)\le l(X)\le ml(X)\cdot d(X)\le nw(X)$ implies the equalities $\ext(X)=s(X)=nw(X)=knw(X)=l(X)=ml(X)\cdot d(X)$. 
\end{proof}

In light of Theorem~\ref{n5.3} the following problem is natural.

\begin{problem} Is $d(X)=nw(X)$ for any regular $k^*$-metrizable $k$-space?
\end{problem} 

We shall give an affirmative answer to this problem assuming (CH), the Continuum Hypothesis.

We shall need $+$-modifications of the cardinal characteristics $\ext(X)$, $l(X)$, and $ml(X)$.
For a topological space $X$ let 
\begin{itemize}
\item $\ext^+(X)$ be the smallest cardinal $\kappa$ such that no closed discrete subset of $X$ has cardinality $\kappa$;
\item $l^+(X)$ be the smallest cardinal $\kappa$ such that each open cover of $X$ has a subcover having size $<\kappa$;
\item $ml^+(X)$ be the smallest cardinal $\kappa$ such that each open cover $\U$ of $X$ has an open refinement $\V$ such that $|\{V\in V:x\in V\}|<\kappa$ for every point $x\in X$.
\end{itemize}

It is easy to see that the cardinals $\ext^+(X)$, $l(X)$, and $ml^+(X)$ completely determine the values of the related cardinal characteristics $\ext(X)$, $l(X)$, and $ml(X)$:
$$\begin{aligned}
\ext(X)=&\sup\{\kappa:\kappa<\ext^+(X)\},\\
l(X)=&\sup\{\kappa:\kappa<l^+(X)\},\\
ml(X)=&\sup\{\kappa:\kappa<ml^+(X)\}.
\end{aligned}
$$

It turns out that the meta-Lindel\"of number of a regular $k^*$-metrizable $k$-space is bounded by a local version of the extent called the extent-character of a topological space.

For a point $x$ of a topological space $X$  by the {\em extent-character}  $\ext_\chi(x;X)$ (resp. {\em extent$^+$-character} $\ext^+_\chi(x;X)$) of  $X$ at $x$ we shall understand the largest cardinal $\kappa$ such that for any family $\U$ of neighborhoods of $x$ having size $|\U|\le\kappa$ (resp. $|\U|<\kappa$) there is an injective map $f:\U\to X$ such that the image $f(\U)$ is a closed discrete subset of $X$ and $f(U)\in U$ for every set $U\in\U$. 

The value of the cardinal $\ext_\chi(x;X)$ is fully determined by the value of $\ext^+_\chi(x;X)$:
$$\ext_\chi(x;X)=\sup\{\kappa:\kappa<\ext_\chi^+(x;X)\}.$$

Let also $$\ext_\chi(X)=\sup_{x\in X}\ext_\chi(x;X),\quad \ext^+_\chi(X)=\sup_{x\in X}\ext^+_\chi(x;X)$$ and observe that $\ext_\chi(X)\le\ext(X)$ and $\ext^+_\chi(X)\le\ext^+(X)$.

The extent$^+$-character $\ext^+_\chi(x;X)$ is also bounded by the usual {\em character} $\chi(x;X)$ equal to the smallest size of a neighborhood base at $x$. 

\begin{proposition}\label{nn5.4} Let $x$ be a point of a topological space. Then $$\ext_\chi(x;X)\le\ext^+_\chi(x;X)\le\chi(x;X)\mbox{ and }\ext_\chi(X)\le\ext_\chi^+(X)\le\chi(X).$$ 
\end{proposition}

Applications of the extent-character to $k^*$-metrizable spaces will rely on the following 

\begin{proposition}\label{nn5.5} Let $\pi:M\to X$ be a map possessing a $\cs^*$-continuous section $s:X\to M$ and let $\U$ be a point-finite open cover of $M$. If the space $X$ is sequential, then each point $x\in X$ has a neighborhood $O_x$ whose image $s(O_x)$ meets less than $\ext^+_\chi(x;X)$ sets $U\in\U$.
\end{proposition} 

\begin{proof} Assume the converse: the image $s(O_x)$ of any neighborhood $O_x$ of some point $x$ meets at least $\ext^+_\chi(x;X)$ sets $U\in\U$.
The definition of the cardinal $\kappa=\ext^+_\chi(x;X)$ yields a family $\mathcal B$
of open neighborhoods of $x$ such that $|\mathcal B|=\kappa$ and for each injective map $f:\mathcal B\to X$ with $f(B)\in B$ for $B\in\mathcal B$ the image $f(\mathcal B)$ is not closed and discrete in $X$.

Enumerate the family $\mathcal B$ as $\mathcal B=\{B_\alpha:\alpha<\kappa\}$. By transfinite induction we shall construct a transfinite sequence of points $\{x_\alpha:\alpha<\kappa\}\subset X$ such that $s_\alpha\in B_\alpha$ and  
$$s(x_\alpha)\notin\bigcup_{\beta<\alpha}\{U\in\U:s(x_\beta)\in U\}$$
for every $\alpha<\kappa$.
Let $x_0=x$. Assuming that for some ordinal $\alpha<\kappa$ the points $x_\beta$, $\beta<\alpha$, have been constructed, we shall find a point $x_\alpha$. Use the point-finity of the cover $\U$ to conclude that the family $\U_{<\alpha}=\bigcup_{\beta<\alpha}\{U\in\U: s(x_\beta)\in U\}$ has size $<\kappa$. By our hypothesis, the image $s(B_\alpha)$ meets at least $\kappa$ sets $U\in\U$. Consequently, there is a point $x_\alpha\in B_\alpha$ with $s(x_\alpha)\notin \cup\U_{<\alpha}$.
This completes the inductive construction.

The choice of the family $\{B_\alpha\}_{\alpha<\kappa}$ implies that the set $D=\{x_\alpha:\alpha<\kappa\}$ is not closed and discrete in the sequential space $X$. The $\cs^*$-continuity of the section $s:X\to M$ implies that the image $s(D)$ is not closed and discrete in $M$. Consequently, there is a point $z\in M$ whose any neighborhood contains infinitely many points $s(x_\alpha)$. In particular, any neighborhood $U\in\U$ of $z$ contains two points $s(x_\alpha)$, $s(x_\beta)$ with $\alpha<\beta$, which is not possible because $s(x_\beta)\notin\{U\in\U:s(x_\alpha)\in U\}$. 
\end{proof}

The proof of the following theorem essentially is due to S.Lin \cite{Lin90a}.

\begin{theorem}\label{nn5.6} If $X$ is a $k^*$-metrizable $k$-space, then $ml(X)\le\ext_\chi(X)$. If the cardinal $\ext^+_\chi(X)$ is regular, then also $ml^+(X)\le \ext_\chi^+(X)$.
\end{theorem}

\begin{proof} Let $\kappa=\ext_\chi(X)$. 
Let $\pi:M\to X$ be a subproper map from a metric space $(M,d)$ and $s:X\to M$ be a precompact-preserving section of $\pi$ having dense image $s(X)$ in $M$. Let $\U_{-1}=\{X\}$ and by induction select a sequence $(\U_k)_{k\in\w}$ of locally finite open covers of $M$ such that each cover $\U_k$ refines $\U_{k-1}$ and consists of sets with diameter $<2^{-k}$.  By Proposition~\ref{nn5.5}, each point of $X$ has a neighborhood meeting at most $\kappa$ sets $s^{-1}(U)$, $U\in\U_k$. 

For a family $\FF$ of subsets of $X$ let
$\ord(\FF)$ (resp. $\Ord(\FF)$) be the smallest cardinal $\lambda$ such that each point of $X$ belongs to (resp. has a neighborhood meeting) at most $\lambda$ sets $F\in\FF$.

\begin{claim} For each family $\FF$ of subsets of $X$ with $\Ord(\FF)\le \kappa$ there is a family $O(\FF)=\{O(F):F\in\FF\}$ of open neighborhoods $O(F)$ of the sets $F\in \FF$ in $X$ with $\ord O(\FF)\le\kappa$.
\end{claim}

Let $\w^{<\w}$ denote the set of all finite sequences of non-negative integer numbers. For each sequence $\sigma\in\w^{<\w}$ we define a family $O_\sigma(\FF)=\{O_{\sigma}(F):F\in\FF\}$ with $\Ord(O_\sigma(\FF))\le\kappa$, by induction. For the empty sequence $\emptyset $ we put $O_\emptyset(F)=F$ for all $F\in\FF$. Assume that for some finite sequence $\sigma=(n_0,\dots,n_k)$ the sets $O_\sigma(F)$, $F\in\FF$, have been defined so that the family $O_\sigma(\FF)=\{O_\sigma(F):F\in\FF\}$ has 
$\Ord(O_\sigma(\FF))\le\kappa$. Given a number $i\in\w$ consider the sequence $\sigma\hat{}i=(n_0,\dots,n_k,i)$ and let
$$
\begin{aligned}
&\mathcal N(\sigma\hat{}i)=\{N\in s^{-1}(\U_i):|\{F\in\FF:N\cap O_{\sigma}(F)\ne\emptyset\}|\le\kappa\}\\
&O_{\sigma\hat{}i}(F)=\cup\{N\in \mathcal N(\sigma\hat{}i):N\cap O_\sigma(F)\ne\emptyset\}
\end{aligned}
$$
To show that $\Ord(O_{\sigma\hat{}i}(\FF))\le\kappa$, take any point $x\in X$ and find a neighborhood $W_x$ meeting at most $\kappa$-sets of the family $O_\sigma(\FF)$. By Proposition~\ref{nn5.5} the neighborhood $W_x$ can be chosen so small that it meets at most $\kappa$ sets of the family $s^{-1}(\U_i)=\{s^{-1}(U):U\in\U_i\}$. Then the definition of the family $O_{\sigma\hat{}i}(\FF)$ implies that $W_x$ meets at most $\kappa$ sets from this family.

After completing the inductive construction, put $O(F)=\bigcup_{\sigma\in\w^{<\w}}O_\sigma(F)$ for $F\in\FF$ and consider the family $O(\FF)=\{O(F):F\in\FF\}$. It follows from $\Ord(O_\sigma(\FF))\le\kappa$ that $\ord(O(\FF))\le \kappa$.

It remains to show that each set $O(F)$, $F\in\FF$ is open in $X$. Since $X$ is a sequential space, it suffices to check that $O(F)$ is sequentially open. Assuming the converse, we could find a sequence $\{x_n\}_{n\in\w}$ in $X\setminus O(F)$, convergent to a point $x_\infty\in O(F)$. Using the $\cs^*$-continuity of the section $s$, we may assume that the sequence $s(x_n)$ converges in $M$ to some point $z_\infty$. The continuity of the map $\pi$ implies that $\pi(z_\infty)=x_\infty$. Since $x_\infty\in O(F)$, there is a finite number sequence $\sigma\in\w^{<\w}$ with $x_\infty\in  O_\sigma(F)$.  Since $\Ord(O_\sigma(\FF))\le\kappa$, the point $x_\infty$ has a neighborhood $W\subset X$ meeting at most $\kappa$ sets of the family $O_\sigma(\FF)$. Then we can find $i\in\w$ so large that each set $U\in\U_i$ containing $z_\infty$ lies in $\pi^{-1}(W)$. Then the set $s^{-1}(U)$ lies in $W$ and hence meets at most $\kappa$ sets of the family $O_{\sigma}(\FF)$. The definition of the set  $O_{\sigma\hat{}i}(F)$ ensures that $s^{-1}(U)\subset O_{\sigma\hat{}i}(F)\subset O(F)$. Since $U$ contains almost all element of the sequence $(s(x_n))$, the set $s^{-1}(U)$ contains almost all elements of the sequence $(x_n)$, a contradiction with $\{x_n\}\subset X\setminus O(F)$.
This contradiction completes the proof of the claim.
\medskip

Now the proof of the theorem is quite easy. Given a an open cover $\W$ of $X$, for every $n\in\w$ consider the subcollection $\mathcal F_n=\{F\in s^{-1}(\U_n):F\subset W(F)$ for some $W(F)\in\W\}$ having $\Ord(\mathcal F_n)\le\kappa$ by Proposition~\ref{nn5.5}.  Since $\bigcup_{n\in\w}s^{-1}(\U_n)$ is a network for $X$, the family $\bigcup_{n\in\w}\FF_n$ covers $X$.

The preceding claim yields us a family $O(\FF_n)=\{O(F):F\in\FF_n\}$ of open neighborhoods of the sets $F\in\FF_n$ with $\ord(O(\FF_n))\le\kappa$. Replacing each $O(F)$ by the intersection $O(F)\cap W(F)$, if necessary, we may assume that $O(F)\subset W(F)$ for all $F\in\FF_n$. Then $\V=\bigcup_{n\in\w}O(\FF_n)$ is an open refinement of $\W$ with $\ord(\V)\le\kappa$, which completes the proof of the inequality $ml(X)\le \kappa=\ext_\chi(X)$.

A minor modification of this proof yields also the inequality $ml^+(X)\le\ext_\chi^+(X)$ in case of a regular cardinal $\ext_\chi^+(X)$.
\end{proof}

The preceding theorem will help us to evaluate the $k$-network weight $knw(X)$ of a  $k^*$-metrizable $k$-space via the extent-characters and the density. 

\begin{theorem}\label{nn5.8} If $X$ is a $k^*$-metrizable $k$-space, then 
\begin{enumerate}
\item $knw(X)=ml(X)\cdot d(X)=\ext_\chi(X)\cdot d(X)$;
\item $knw(X)<2^{d(X)}$ if $X$ is a regular space.
\end{enumerate}
\end{theorem}

\begin{proof} 1. The inequality $\ext_\chi(X)\cdot d(X)\le knw(X)$
follows immediately from the trivial inequalities $\ext_\chi(X)\le\ext(X)\le nw(X)\le knw(X)$ and $d(X)\le nw(X)\le knw(X)$.
On the other hand, Theorems~\ref{n5.3} and \ref{nn5.6} yield $knw(X)=ml(X)\cdot d(X)\le \ext_\chi(X)\cdot d(X)$.
\medskip

2. By Theorem~\ref{n5.3}, $d(X)\le knw(X)\le\aleph_1\cdot d(X)\le 2^{d(X)}$. So it remains to show that the assumption $knw(X)=2^{d(X)}$ leads to a contradiction. Observe that the latter equality implies that $d(X)=\aleph_0$ and $knw(X)=2^\aleph_0=\aleph_1$. By Theorem~1.5.6 \cite{En}, the space $X$, being regular, has weight $w(X)\le 2^{d(X)}$. Consequently, $\ext^+_\chi(X)\le\chi(X)\le w(X)\le 2^{\aleph_0}=\aleph_1$ and hence $\ext_\chi(X)\le\aleph_0$. Now the first item implies $knw(X)=\ext_\chi(X)\cdot d(X)\le\aleph_0$, which contradicts $knw(X)=\aleph_1$.
\end{proof} 

Combining Theorem~\ref{n5.3} with the second item of Theorem~\ref{nn5.8}, we get

\begin{corollary}\label{nn5.9} Under (CH), every regular $k^*$-metrizable $k$-space satisfies the equality $knw(X)=d(X)$.
\end{corollary}

This corollary is specific for $k$-spaces and does not hold in general.

\begin{remark} The Frolik's space $X$ from Remark~\ref{frolik} has cardinal characteristics$$\aleph_0=d(X)< nw(X)=knw(X)=w(X)=2^{\aleph_0}<2^{2^{\aleph_0}}=|kX|=knw(kX)$$which show that even for a $k$-metrizable space $X$ the gap between $d(X)$ and $knw(X)$ can be very large.
\end{remark}

Nonetheless we can estimate the cardinal characteristics of $k^*$-metrizable spaces via the cardinal characteristics of their $k$-coreflexions.

\begin{corollary}\label{nn5.11} If $X$ is a $k^*$-metrizable space, then
$$d(X)\le \,nw(X)\le knw(X)\le d_1(X)=d_{<\w_1}(X)=knw(kX)\le \so(X)\cdot d(kX).$$
\end{corollary}

The preceding discussion displays the importance of the local cardinal invariant $\ext_\chi(X)$. It turns out that this cardinal characteristic can be bounded from below by another local cardinal invariant, which is a quantitative version of the Arkhangelski's property $(\alpha_4)$.

 Following \cite{Ar} we say that a topological space $X$ has the {\em  property $(\alpha_4)$} at a point $x\in X$ if for any  countable family $\mathcal S$  consisting of non-trivial sequences convergent to $x$  there is a sequence $T\subset X$ convergent to $x$ and intersecting infinitely many sequences $S\in\mathcal S$. (Besides the usual meaning, by a {\em non-trivial convergent} sequence in a space $X$ we shall understand a countable infinite subset $S\subset X$ whose closure in $X$ is compact that has a unique non-isolated point $x=\lim S$ called the limit of $S$).
\smallskip

A quantification of the property $(\alpha_4)$ yields two local cardinal characteristics.
\smallskip

For a point $x$ of a topological space $X$ let $\alpha_4(x;X)$  be the smallest cardinal $\tau$ such that for any cardinal $\kappa<\tau$ there is a  family $\mathcal S$ of  size $|\mathcal S|=\kappa$  that consist of non-trivial sequences convergent to $x$ and is such that each sequence $T\subset X$ convergent to $x$ meets only finitely many sequences $S\in\mathcal S$. If such a cardinal $\tau$ does not exist (which happen if $x$ is an isolated point of the sequential coreflexion $sX$), then we put $\alpha_4(x;X)=0$.

Another way of quantifying the property $(\alpha_4)$ leads to a cardinal invariant $\alpha_4^+(x;X)$ equal to the smallest cardinal $\kappa\ge1$ such that
such that for any family $\mathcal S$ of  size $|\mathcal S|=\kappa$  that consist of non-trivial sequences convergent to $x$ there is a sequence $T\subset X$ that converges to $x$ and meets infinitely many sequences $S\in\mathcal S$. We put $\alpha_4^+(x;X)=1$ if $x$ an isolated point of $sX$.
\smallskip

Observe  that $\alpha_4(x;X)=\sup\{\kappa:\kappa<\alpha_4^+(x;X)\}$ and hence the cardinal $\alpha^+_4(x;X)$ completely determined the value of $\alpha_4(x;X)$.

We put $\alpha_4(X)=\sup_{x\in X}\alpha_4(x;X)$ and $\alpha_4^+(X)=\sup_{x\in X}\alpha_4^+(x;X)$.
\smallskip

Observe that $\alpha_4^+(X)=\aleph_0$ if and only if $x$ is a non-isolated point in $sX$ and $X$ has property $(\alpha_4)$ at $x$.
 On the other hand,
for any infinite cardinal $\kappa$ the sequential fan $S_\kappa$ with $\kappa$ nodes has $\alpha_4(S_\kappa)=\kappa$ and $\alpha^+_4(S_\kappa)=\kappa^+$.

The cardinal $\alpha^+_4(X)$ will be used to detect the presence of a sequential copies of the sequential fan $S_\kappa$ in the space. For a point $x$ of a topological space $X$ we write $(S_\kappa,*)\subset_{\mathrm{cl}} (X,x)$  if there is a  closed subset $F\subset X$ containing point $x$ and a  homeomorphism $h:S_\kappa\to F$ mapping the unique non-isolated point of the fan $S_\kappa$ onto $x$.

\begin{theorem}\label{n6.1} Let $\pi:M\to X$ be a map from a metrizable space and $s:X\to M$ be a $\cs^*$-continuous section of $\pi$ such that $s(X)$ is dense in $M$. Then for any point $x\in X$ that is not isolated in the sequential coreflexion $sX$ we get
\begin{enumerate}
\item $\alpha_4(x;X)=\sup\{\kappa:(S_\kappa,*)\subset_{\mathrm{cl}} (sX,x)\}=\max\{\aleph_0,\ext(\pi^{-1}(x))\}\le\ext_\chi(x;sX)$;
\item $\alpha^+_4(x;X)=\min\{\kappa:(S_\kappa,*)\not\subset_{\mathrm{cl}} (sX,x)\}=\max\{\aleph_0,\ext^+(\pi^{-1}(x))\}\le\ext_\chi^+(x;sX)$.
\end{enumerate}
\end{theorem}

\begin{proof} Since the values of the cardinals in the second item determine the values of the corresponding cardinals in the first item, it suffices to prove item (2) only.

Let $$\begin{aligned}
\kappa_1=&\,\alpha_4^+(x;X),\\
\kappa_2=&\max\{\aleph_0,\ext^+(\pi^{-1}(x))\},\\
\kappa_3=&\min\{\kappa: (S_\kappa,*)\not\subset_{\mathrm{cl}} (sX,x)\}.
\end{aligned}
$$
We shall prove the inequalities $\kappa_1\le\kappa_2\le\kappa_3\le\kappa_1$ and $\kappa_3\le\ext_\chi^+(x;sX)$.

$(\kappa_1\le\kappa_2)$. Assume conversely that $\alpha_4^+(x;X)=\kappa_1>\kappa_2$. The definition of $\alpha_4^+(x;X)$ yields an infinite family $\mathcal S$ consisting of $\kappa_2$ many sequences convergent to $x$ such that any other sequence $T\subset X$ with $\lim T=x$ meets only finitely many sequences $S\in\mathcal S$. Since the section $s:X\to M$ is $\cs^*$-continuous, we may replace each sequence $S\in\mathcal S$ with a suitable subsequence and assume that the sequence $s(S)$ converges to some point $z_S\in M$. The continuity of $\pi$ implies that $z_S\in \pi^{-1}(x)$. We claim that the set $\{z_S:S\in\mathcal S\}$ is a closed discrete subset in $\pi^{-1}(x)$. Observe that any convergent sequence $T\subset M$ meets only finitely many sequences $s(S)$, $S\in\mathcal S$ (because its image $\pi(T)$ meets only finitely many sequences $S\in\mathcal S$). Then the indexed set $\{z_S:S\in\mathcal S\}$ cannot have accumulating points in $M$ and hence $D=\{z_S:S\in\mathcal S\}$ is a closed discrete subset of $\pi^{-1}(x)$ with size $|\mathcal S|=\kappa_2\ge \ext^+(\pi^{-1}(x))$. But this contradicts the definition of the cardinal $\ext^+(\pi^{-1}(x))$. This proves the inequality $\kappa_1\le\kappa_2$.
\medskip

$(\kappa_2\le\kappa_3)$ also will be proved by contradiction. Assume conversely that $\ext^+(\pi^{-1}(x))=\kappa_2>\kappa_3$. The definition of $\ext^+(\pi^{-1}(x))$ provides a closed discrete subset $D\subset \pi^{-1}(x)$ having size $|D|=\kappa_3$.  We may additionally assume that $s(x)\notin D$. The space $M$, being metrizable, is collectively normal. Hence for each point $z\in D$ we can find a neighborhood $O_z$ such that the family $\{O_z:z\in D\}$ is discrete. The density of $s(X)$ in $M$ implies that for any point $z\in D$ we can find a sequence $S(z)\subset O_z\cap s(X)$ convergent to $z$. Using the $\cs^*$-continuity of the section $s$ it can be shown that the set $F=\{x\}\cup\bigcup_{z\in D}\pi(S(z))$ with distinguished point $x$ is a closed copy of the fan $(S_{\kappa_3},*)$ in the sequential coreflexion $sX$ of $X$. But this contradicts the definition of the cardinal $\kappa_3$.
\medskip

To prove the inequality $\kappa_3\le\kappa_1$, assume conversely that $\kappa_3>\kappa_1=\alpha_4^+(x;X)$. Then the definition of the cardinal $\kappa_3$ yields us a topological copy of the fan $(S_{\kappa_1},*)$ in $(sX,x)$. Since each sequence convergent to $x$ meets only finitely many nodes of the fan $S_{\kappa_1}$, the space $X$ has $\alpha^+_4(X;x)>\kappa_1$, a contradiction with the definition of the cardinal $\kappa_1$.
\medskip

The final inequality $\kappa_3\le\ext_\chi(x;sX)$ follows from the observation that $\ext^+_\chi(S_\kappa)=\kappa^+$ for each infinite cardinal $\kappa$.
\end{proof}

Theorem~\ref{n6.1} gives an upper bound $\ext_\chi(x;X)$ for $\alpha_4(x;X)$. A lower bound is given by another local cardinal invariant called the $\cs$-character of $X$ at $x$.

A family $\mathcal N$ of subsets of a space $X$ is called a {\em $\cs$-network} (resp. {\em $\wcs^*$-network}) at $x$ if for any neighborhood $U\subset X$ of $X$ and any sequence $\{x_n\}_{n\in\w}\subset U$ convergent to $x$ there is a set $N\subset U$ in $\mathcal N$ such that $N$ contains all but finitely many (resp. infinitely many) points $x_n$.

The {\em $\cs$-character} (resp. {\em $\cs^*$-character}) of a space $X$ at a point $x\in X$ is the smallest size $|\mathcal N|$ of a $\cs$-network (resp. $\wcs^*$-network) at $x$. The cardinal $\cs_\chi(X)=\sup_{x\in X}\cs_\chi(x,X)$ (resp.$\cs_\chi(X)=\sup_{x\in X}\cs_\chi(x,X)$) is called the {\em $\cs$-character} (resp.  {\em $\cs^*$-character}) if $X$.

\begin{proposition}\label{n6.3} If $X$ is a $\cs^*$-metrizable space, then $\cs^*_\chi(x;X)\le\cs_\chi(x;X)\le \alpha_4(x;X)\le\ext_\chi(x;X)$. If $cs^*_\chi(x;sX)\le\aleph_0$, then $\alpha_4(x;X)=cs^*_\chi(x;X)=\cs_\chi(x;X)\in\{1,\aleph_0\}$.
\end{proposition}

\begin{proof} Since $X$ is $\cs^*$-metrizable, there is a  map $\pi:M\to X$ from a metrizable space and a $\cs^*$-continuous section $s:X\to M$ such that $s(X)$ is dense in $M$. By Theorem~\ref{n6.1}, $\ext(\pi^{-1}(x))\le\alpha_4(x;X)$. Let $\mathcal B$ be a $\sigma$-discrete base of the topology of the metrizable space $M$. It follows that the subfamily $\mathcal C=\{U\in\mathcal B:{U}\cap \pi^{-1}(y)\}$ has size $$|\mathcal C|\le d(\pi^{-1}(x))=\ext(\pi^{-1}(x))\le\alpha_4(x;X).$$ Then the family $\C_\cup=\{\cup\FF:\FF\subset\C,\;|\C|<\infty\}$ also has size $\le \alpha_4(x;X)$. We claim that $\mathcal N=\{\pi(C):C\in\C_\cup\}$ is a $\cs$-network at $x$. Indeed, take any neighborhood $O_x\subset X$ of $x$ and a sequence $S=\{x_n\}\subset  U$ convergent to $x$. The $\cs^*$-continuity of the section $s$ implies that the set $s(S)$ has compact closure $K$ in $M$. Let $\FF\subset\mathcal B$ be a finite subcover of the set $K\cap \pi^{-1}(x)$ such that $\cup \FF\subset \pi^{-1}(O_x)$. The continuity of $\pi$ implies that all points of $K\setminus \pi^{-1}(x)$ are isolated and hence the set $K\setminus\cup\FF$ is finite. Then $\pi(\cup\FF)\subset O_x$ is an element of the family $\mathcal N$, containing almost all elements of the sequence $(x_n)$. This shows that $\mathcal N$ is a $\cs$-network at $x$ and hence $\cs_\chi(x;X)\le|\mathcal N|\le \alpha_4(x;X)$. The inequality $\alpha_4(x;X)\le\ext_\chi(x;X)$ was proved in Theorem~\ref{n6.1}.
\medskip

Now assume that $\cs^*_\chi(x;sX)\le\aleph_0$. If $x$ an isolated point in $sX$, then $\alpha_4(x;X)=\cs^*_\chi(x;X)=\cs_\chi(x;X)=1$. So, assume that $x$ is not isolated in $sX$, which means that there is a non-trivial convergent sequence in $X$.
In light of the preceding paragraph, it suffices to check that $\alpha_4(x;X)\le\aleph_0$. Assuming the converse and applying Theorem~\ref{n6.1}, find a closed embedding $(S_{\w_1},*)$ into $(sX,x)$. Then $\cs^*_\chi(S_{\w_1})\le \cs_\chi^*(x;sX)\le\aleph_0$ and hence there is a countable $\wcs^*$-network $\mathcal N$ for $S_{\w_1}$. Write $S_{\w_1}=\{*\}\cup\{x^\alpha_n:n\in\w,\; \alpha<\w_1\}$ where each sequence $T_\alpha=\{x^\alpha_n\}_{n\in\w}$ converges to the non-isolated point $*\notin T_\alpha$ of $S_{\w_1}$, the sequences $T_\alpha$, $\alpha\in\w_1$, are pairwise disjoint and determine the topology of $S_{\w_1}$ in the sense that a subset $F\subset T_{\w_1}$ is closed in $S_{\w_1}$ if for any $\alpha<\w_1$ the intersection $F\cap\overline{T}_\alpha$ is compact. For a set $A\subset S_{\w_1}$ let $\supp(A)=\{\alpha\in\w_1:A\cap T_\alpha\ne\emptyset\}$. Let $\mathcal N_\infty$ be the subfamily of the $\wcs^*$-network $\mathcal N$, consisting of all sets $N\in\mathcal N$ with infinite support $\supp(N)$. Since $\mathcal N_\infty$ is at most countable, we can inductively construct an at most countable subset $D\subset S_{\w_1}$ such that $D\cap N\ne\emptyset$ for any $N\in\mathcal N_\infty$ and $|D\cap T_\alpha|\le 1 $ for any $\alpha\in\w_1$. By definition of the topology of $S_{\w_1}$, the set $D$ is closed and discrete in $S_{\w_1}$. Then $U=S_{\w_1}\setminus D$ is an open neighborhood of the non-isolated point $*$ of $S_{\w_1}$. Pick any countable ordinal $\alpha$ that does not belong to the countable set $\supp(D)\cup\bigcup_{N\in\mathcal N\setminus\mathcal N_\infty}\supp(N)$. Then $T_\alpha$ is a convergent sequence in $U$ and hence there is a set $N_0\subset U$ in $\mathcal N$ having infinite intersection   with $T_\alpha$. This set $N_0$ cannot belong to $\mathcal N_\infty$ because no set $N\in\mathcal N_\infty$ lies in $U$. Then $N_0\in\mathcal N\setminus\mathcal N_\infty$, which also is not possible because $\alpha\notin \supp(N_0)$.
The obtained contradiction completes the proof of the inequality $\alpha_4(x;X)\le\aleph_0$.
\end{proof}

\begin{question} Does $\cs^*_\chi(x;X)\le\aleph_0$ imply $\alpha_4(x;X)=\cs_\chi^*(x;X)$? (The answer is affirmative if $X$ is a sequential space).
\end{question}

\begin{question} Is $\alpha_4(x;X)=\ext_\chi(x;X)$ for a point $x$ of a regular $k^*$-metrizable $k$-space?
\end{question}

\section{Characterizing $k^*$-metrizable spaces in terms of $k$-networks}\label{s7}

In this section for each space $X$ we describe a canonical construction of a map $\lim:\Nw(X)\to X$ from the so-called Network Hyperspace over $X$ such that $X$ is $k^*$-metrizable (resp. $\cs^*$-metrizable) if and only if the map $\lim$ has a precompact-preserving (resp. $\cs^*$-continuous) section.

Given a  topological space $X$ by $\mathcal P(X)$ we denote the power-set of $X$ endowed with the discrete topology. For each point $x\in X$ consider the subset $\Nw(x)$ of $\mathcal P(X)^\w$ consisting of all sequences $(A_n)_{n\in\w}$ of non-empty subsets of $X$ such that
\begin{itemize}
\item $A_{n+1}\subset A_n$ for all $n$;
\item for any neighborhood $U$ of $x$ there is a number $n\in\w$ with $\cl_1(A_n)\subset U$.
\end{itemize}
Therefore $(\{x\}\cup A_n)_{n\in\w}$ is a decreasing network at $x$.
We recall that $\cl_1(A)$ stands for the 1-st sequential closure of a set $A$ in a topological space $X$. By definition, $\cl_1(A)$ consists of limit points of sequences $(a_n)\subset A$ that converge in $X$.

The subspace $\Nw(X)=\bigcup_{x\in X}\Nw(x)$ of the Tychonoff product $\mathcal P(X)^\w$
 is called the {\em Network Hyperspace} over $X$. Since $X$ is Hausdorff, $\Nw(x)\ne\Nw(y)$ for distinct points $x,y$ of $X$, which allows us to define a map $\lim:\Nw(X)\to X$ letting $\lim^{-1}(x)=\Nw(x)$. It is easy to see that the so-defined map $\lim$ is continuous.

The construction of the Network Hyperspace is functorial in the sense that each
continuous map $f:X\to Y$ between Hausdorff spaces induces a continuous map $f^\w:\Nw(X)\to\Nw(Y)$, $f^\w:(A_n)\mapsto \big(f(A_n)\big)$ making the diagram
$$
\begin{CD}
\Nw(X)@>{\lim}>>X\\
@VV{f^\w}V @V{f}VV\\
\Nw(Y)@>{\lim}>>Y
\end{CD}
$$commutative.

The Network Hyperspace will be our principal tool in characterizing $k^*$-metrizable spaces in terms of $k$-networks.

 A family $\mathcal N$ of subsets of a space $X$ is called a {\em $k$-network} (resp. {\em \closed\  $k$-network}) for $X$ if for any open set $U\subset X$ and a compact set $K\subset U$ there is a finite subfamily $\mathcal F\subset\mathcal N$ such that $K\subset\cup\mathcal F\subset U$ (resp. $K\subset\cup\FF\subset\cl_1(\cup\FF)\subset U$).

A family $\mathcal N$ is a {\em $\wcs^*$-network} (resp. {\em \closed\  $\wcs^*$-network}) for $X$ if for any open set $U\subset X$ and a sequence $(x_n)_{n\in\w}\subset U$, convergent to a point $x_\infty\in U$ there is a set $N\in\mathcal N$ containing infinitely many points $x_n$ and such that $N\subset U$ (resp. $\cl_1(N)\subset U$).

More detail information on $k$- and $\wcs^*$-networks can be found in \cite{Lin90}, \cite{LiTa96}, \cite{LiTa2}, \cite{Ta1}.
It is clear that a family $\mathcal N$ is a $k$-network (resp. $\wcs^*$-network) for a sequentially regular space $X$ if and only if $\mathcal N$ is a \closed\  $k$-network (resp. \closed\  $\wcs^*$-network) for $X$. We recall that a topological space $X$ to be {\em sequentially regular} if for each point $x\in X$ and a neighborhood $U\subset X$ of $x$ there is another neighborhood $V\subset X$ of $x$ with $\cl_1(V)\subset U$.

A family $\mathcal A$ of subsets of a space $X$ is called {\em
compact-finite} (resp. $cs$-finite) if for any  compact
subset (resp. convergent sequence) $K\subset X$ the family $\mathcal
F=\{A\in\mathcal A:A\cap K\ne\emptyset\}$ is finite. A family
$\mathcal A$ is {\em $\sigma$-compact-finite} (resp. {\em
$\sigma$-$cs$-finite}) if it can be written as the countable union
$\mathcal A=\bigcup_{n\in\w}\mathcal A_n$ of compact-finite (resp.
$cs$-finite) subfamilies.

\begin{proposition}\label{n7.1} If $\mathcal A$ is a $\sigma$-compact-finite (resp. $\sigma$-cs-finite) family of subsets of a space $X$, then the family $\A_\infty=\{\cap \FF:\FF\subset \A,\;|\FF|<\infty\}$ also is $\sigma$-compact-finite (resp. $\sigma$-cs-finite).
\end{proposition}

\begin{proof} Write $\mathcal A$ as the countable union $\mathcal A=\bigcup_{k\in\w}\A_k$ of compact-finite families. For any finite subset $F\subset \w$ consider the compact-finite family $\mathcal A_F=\{\bigcap_{k\in F}A_k:A_k\in\A_k,\; k\in F\}$ and observe that $\mathcal A_\infty=\bigcup_F\mathcal A_F$ is $\sigma$-compact-finite, being the countable union of the compact-finite families $\mathcal A_F$.

By analogy we can prove the cs-finite case.
\end{proof}

It is clear that each $\sigma$-compact-finite $k$-network for $X$
is a $\sigma$-cs-finite $cs^*$-network.  If all compact subsets of
$X$ are countably compact, then the converse is also true.

\begin{proposition}\label{n7.2} Let $X$ be a topological space whose all compact
subsets are sequentially compact. Then each $\sigma$-cs-finite \textup{(}\closed\textup{)}
$\wcs^*$-network $\mathcal N$ for $X$ is a $\sigma$-compact-finite
\textup{(}\closed\textup{)} $k$-network for $X$.
\end{proposition}

\begin{proof} First we check that each $cs$-finite family $\mathcal F$ of subsets of $X$ is compact-finite. Assuming the converse, find a sequence $(F_k)_{k\in\w}$ of pairwise distinct sets from $\mathcal F$ meeting some compact set $K$. For every $k\in\w$ pick a point $x_k\in K\cap F_k$ and use the sequential compactness of $K$ to find a convergent subsequence $(x_{k_i})_{i\in\w}$ of $(x_k)$. Since the set $\{x_{k_i}:i\in\w\}$ intersects infinitely many sets from $\mathcal F$, the family $\mathcal F$ cannot be $cs$-finite, which is a contradiction.

Therefore, each  $\sigma$-cs-finite $\wcs^*$-network $\mathcal N$ for $X$ is $\sigma$-compact-finite. Next, we show that $\mathcal N$ is a  $k$-network for $X$. Fix any open set $U\subset X$ and a compact subset $K\subset U$. Since $\mathcal N$ is $\sigma$-compact-finite, the subfamily $\mathcal N'=\{N\in\mathcal N: N\subset U, \;N\cap K\ne\emptyset\}$ is at most countable and hence can be enumerated as $\mathcal N'=\{N_k:k\in\w\}$. We claim that $K\subset N_0\cup\dots\cup N_m$ for some finite $m$. Assuming the converse, we would construct a sequence $(x_n)_{n\in\w}$ with $x_n\in K\setminus(N_0\cup\dots\cup N_n)$ for all $n\in\w$. By the sequential compactness of $K$ this sequence has a convergent subsequence $(x_{n_i})$. The family $\mathcal N$, being a $cs^*$-network, contains a set $N\in\mathcal N$ such that $N\subset U$ and $N$ meets infinitely many points of the sequence $(x_{n_i})$. Since $N$ meets $K$ and lies in $U$, it belongs to the subfamily $\mathcal N'$ and hence equals $N_m$ for some $m$. Now the choice of the points $x_k$ implies that $x_{n_i}\notin N_m=N$ for all $i$ with $n_i>m$, which is a contradiction.

Repeating the above prove for the family $\mathcal N'_1=\{N\in\mathcal N':\cl_1(N)\subset U,\; N\cap K\ne\emptyset\}$, we can prove that each $\sigma$-cs-finite \closed\ $\wcs^*$-network for $X$ is a $\sigma$-compact-finite \closed\  $k$-network for $X$.
\end{proof}

The following theorem gives an inner characterization of $\cs^*$-metrizable spaces in terms of $\wcs^*$-networks.

\begin{theorem}\label{n7.3} For a space $X$ the following conditions are equivalent:
\begin{enumerate}
\item $X$ is $\cs^*$-metrizable;
\item the map $\lim:\Nw(X)\to X$ has a $\cs^*$-continuous section $s:X\to\Nw(X)$;
\item $X$ has a $\sigma$-cs-finite \closed\  $\wcs^*$-network.
\end{enumerate}
\end{theorem}

\begin{proof} The implication $(2)\Ra(1)$ is trivial.

$(1)\Ra(3)$ Assuming that $X$ is $\cs^*$-metrizable, find a map $\pi:M\to X$ from a metrizable space $M$ having a $\cs^*$-continuous section $s:X\to M$. Let $Z=s(X)$. The space $M$, being metrizable, admits a $\sigma$-locally-finite base $\mathcal B$ of the topology. For each $B\in\mathcal B$ let $N_B=s^{-1}(B)$. We claim that $\mathcal N=\{N_B:B\in\mathcal B\}$ is a $\sigma$-cs-finite \closed\  $\wcs^*$-network for $X$.
 First we show that $\mathcal N$ is a \closed\ $\wcs^*$-network for $X$. Take an open set $U\subset X$ and a sequence $(x_n)_{n\in\w}\subset U$, convergent to a point $x_\infty\in U$. It follows from the $\cs^*$-continuity of $\pi$ that the sequence $(s(x_n))$ has an accumulating point $z_\infty\in M$ which projects onto $x_\infty$ by the continuity of $\pi$.  Find a basic set $B\in\mathcal B$ such that $z_\infty\in B\subset\overline{B}\subset\pi^{-1}(U)$. Then $B$ contains infinitely many points $s(x_n)$ and hence $N_B=s^{-1}(B)$ contains infinitely many points $x_n$.
The $\cs^*$-continuity of $s$ implies that $\cl_1(N_B)\subset \pi(\bar B)\subset U$. This proves that $\mathcal N$ is a \closed\  $\wcs^*$-network for $X$.

To show that the $\wcs^*$-network $\mathcal N$ is $\sigma$-cs-finite, write the base $\mathcal B$ as the countable union $\mathcal B=\bigcup_{k\in\w}\mathcal B_k$ of locally finite families in $M$. Then $\mathcal N=\bigcup_{k\in\w}\mathcal N_k$ where $\mathcal N_k=\{N_B:B\in\mathcal B_k\}$. It remains to check that each family $\mathcal N_k$ is cs-finite in $X$. Fix any convergent sequence $K\subset X$. It follows from the $\cs^*$-continuity of $s$ that the image $s(K)$ is precompact in $M$ and thus meets only finitely many sets of the locally finite family $\mathcal B$. Then its image $K=\pi(s(K))$ meets only finitely many elements of the family $N_B$, $B\in\mathcal B_k$, which proves the $cs$-finity of $\mathcal N_k$.
Therefore each $\cs^*$-metrizable space admits a $\sigma$-$cs$-finite \closed\ $\wcs^*$-network.
\medskip

$(3)\Ra(2)$ Now assume that a space $X$ has a $\sigma$-cs-finite \closed\  $\wcs^*$-network $\mathcal N$. According to Proposition~\ref{n7.1} we may assume that $\mathcal N$ is closed under finite intersections.

The network $\mathcal N$, being $\sigma$-cs-finite, can be written as the countable union $\mathcal N=\bigcup_{k\in\w}\mathcal N_k$ of cs-finite families so that for any $m\in\w$ the family $\bigcup_{k\ge m} \mathcal N_k$ still is a \closed\ $\wcs^*$-network for $X$. Define an increasing sequence of ordinals $(\beta_k)_{k\in\w}$ letting $\beta_0=0$ and $\beta_{k+1}=\beta_k+|\mathcal N_k|$. Let also $\beta_\w=\sup_{k\in\w}\beta_k$. Then the family $\mathcal N$ can be enumerated as $\mathcal N=\{N_\alpha:\alpha<\beta_\w\}$ so that $\mathcal N_k=\{N_\alpha:\beta_k\le\alpha<\beta_{k+1}\}$ for all $k\in\w$. The cs-finity of the families $\mathcal N_k$ implies that for every point $x\in X$ the set $\Lambda(x)=\{\alpha<\beta_\w:x\in N_\alpha\}$ has finite intersection with each interval $[\beta_k,\beta_{k+1})$ and consequently, $\Lambda(x)$ has order type of $\w$ (here we also use the fact that for every $k\in\w$ the family $\{N_\alpha:\beta_k\le\alpha<\beta_\w\}$ is a network for $X$). Consequently, there is an increasing bijective map $\alpha_x:\w\to \Lambda(x)$.

 Finally, let $s(x)=(A_k)_{k\in\w}$, where $A_k=\bigcap_{i\le k}N_{\alpha_x(i)}$.
It is easy to see that  $s(x)\in \Nw(x)$ and $\lim\circ s(x)=x$, which means that $s$ is a section of the map $\lim:\Nw(X)\to X$.

It remains to check that this section is $\cs^*$-continuous.
Take any sequence $(x_n)$ in $X$ convergent to a point $x_\infty\in X$.
Let $K=\{x_n:n\le\infty\}$ and consider the family
$\Lambda(K)=\{\alpha<\beta_\w:K\cap N_\alpha\ne\emptyset\}$.
 The cs-finity of
the families $\mathcal N_k$ implies that for every $k\in\w$ the intersection $\Lambda(K)\cap[0,\beta_{k+1})$ is finite and hence $\Lambda(K)$ is a countable set.
Using the fact that for every $k\in\w$ the family $\{N_\alpha:\beta_k\le\alpha<\beta_\w\}$
is a $\wcs^*$-network for $X$, construct an increasing number sequence $(k_i)_{i\in\w}$ such that $K\subset \bigcup_{\beta_{k_i}\le\alpha<\beta_{k_{i+1}}}N_\alpha$.

For every $n\in\w$ consider the set $\Lambda(x_n)=\{\alpha<\beta_\w:x_n\in N_\alpha\}$ having order type $\w$ and let $\alpha_n=\alpha_{x_n}:\w\to\Lambda(x_n)$ be the increasing bijective map. It follows that for every $i\in\w$ we get $\alpha_n(i)<\beta_{k_{i+1}}$ and thus $\alpha_n(i)\in\Lambda_i(K)$ where
$\Lambda_i(K)=\Lambda(K)\cap [0,\beta_{k_{i+1}})$ for $i\in\w$.
Endow each finite set $\Lambda_i(K)$ with discrete topology and consider the countable product $\Pi=\prod_{i\in\w}\Lambda_i(K)$, which is a metrizable compact space.
For every $n$ the function $\alpha_n:\w\to \Lambda_i(K)$ can be considered as a point of the metrizable compact space $\Pi$. Then  the sequence $(\alpha_n)_{n\in \w}$ contains a subsequence $(\alpha_{n})_{n\in J}$ that converges to some point $\alpha_\infty\in\Pi$ (here $J$ is an infinite subset of $\w$).
Moreover, we can assume that $\alpha_{n}(i)=\alpha_\infty(i)$ for all $n\in J$ with $n\ge i$.

Consequently, $x_{n}\in N_{\alpha_\infty(i)}$ for all $n\in J$ with $n\ge i$ and hence $x_\infty=\lim x_n\in \cl_1(N_{\alpha_\infty(i)})$ for all $i\in\w$.
Consider the sequence $(N_{\alpha_\infty(i)})_{i\in\w}$. It follows from the convergence of the sequence $(\alpha_{n})_{n\in J}$ to $\alpha_\infty$ that the sequence $(s(x_{n}))_{n\in J}$ converges to the sequence $\vec n=(\cap_{j\le i} N_{\alpha_\infty(i)})_{i\in\w}$ in $\mathcal P(X)^\w$. It remains to check that $\vec n$ belongs to the subset $\Nw(X)\subset\mathcal P(X)^\w$. This will follows as soon as we prove that each neighborhood $U$ of $x_\infty$ contains some set $\cl_1(N_{\alpha_\infty(i)})$. Since $\mathcal N$ is a \closed\ $cs^*$-network,  there is a set $N_\gamma\in\mathcal N$ with $\gamma<\beta_\w$ and $\cl_1(N_\gamma)\subset U$ containing infinitely many points of the sequence $(x_n)_{n\in J}$. So the set $J_1=\{n\in J:x_n\in N_\gamma\}$ is infinite. For every $n\in J_1$, the ordinal $\gamma$ belongs to the set $\Lambda(x_n)$ and hence $\gamma=\alpha_{x_n}(k_n)$ for some number $k_n\le |\Lambda(K)\cap[0,\gamma]|$. Since the set $\Lambda(K)\cap [0,\gamma]$ is finite, for some number $m$ the set $J_2=\{n\in J_1: k_n=m\}$ is infinite. The convergence of the sequence $(\alpha_n)_{n\in J}$ to $\alpha_\infty$ implies that $\alpha_\infty(m)=\alpha_{x_n}(m)=\gamma$ for all sufficiently large numbers $n\in J$. Consequently, $N_{\alpha_\infty(m)}=N_{\gamma}\subset\cl_1(N_\gamma)\subset U$, which implies that  $\vec n=(N_{\alpha_\infty(i)})_{i\in\w}\in M$ by the definition of $M$.
\end{proof}

Theorem~\ref{n7.3} combined with Proposition~\ref{n7.2} implies a characterization of $k^*$-metrizable spaces in terms of \closed\ $k$-networks.

\begin{theorem}\label{n7.4} For a topological space $X$ the following conditions are
equivalent:
\begin{enumerate}
\item $X$ is $k^*$-metrizable;
\item $X$ is $cs^*$-metrizable and all compact subsets of $X$ are sequentially compact;
\item the map $\lim:\Nw(X)\to X$ has a section that preserves precompact sets;
\item all compact subsets of $X$ are sequentially compact and $X$
has a $\sigma$-cs-finite \closed\  $\wcs^*$-network;
\item $X$ has a $\sigma$-compact-finite \closed\  $k$-network.
\end{enumerate}
If the space $X$ is sequentially regular, then the conditions (1)--(5) are equivalent to
\begin{enumerate}
\item[(6)] $X$ has a $\sigma$-compact-finite $k$-network.
\end{enumerate}
\end{theorem}

\begin{proof} The equivalences $(1)\Leftrightarrow(2)\Leftrightarrow(3)\Leftrightarrow(4)$ follow from
Theorem~\ref{n7.3} and Propositions~\ref{n4.2} and \ref{n1.3}. The implication $(4)\Ra(5)$ follows
from Proposition~\ref{n7.2}. To prove that $(5)\Ra(4)$ it suffices
to check that for a space $X$ possessing a $\sigma$-compact-finite
\closed\ $k$-network $\mathcal N$ each compact subset $K\subset X$ is
sequentially compact. For this observe that the family $\A=\{N\cap
K:N\in\mathcal N,\;N\cap K\ne\emptyset\}$ is a countable
$k$-network in $K$, which implies the metrizability (and hence sequential compactness) of $K$.

The equivalence $(6)\Leftrightarrow(5)$ for a sequentially regular space $X$ is obvious.
\end{proof}

\begin{remark} The $\sigma$-compact-finite $k$-networks in the characterization of $k^*$-metrizable spaces cannot be replaced by compact-countable $k$-networks: the ordinal space $[0,\w_1)$ has a compact-countable base of the topology and hence has a compact-countable $k$-network. Yet, it is not $k^*$-metrizable, being sequentially compact but not compact.
\end{remark}

Each continuous binary operation $\circ:X\times X\to X$ on a regular space $X$ induces a continuous binary operation $$\bullet:\Nw(X)\times\Nw(X)\to\Nw(X),\; (A_n)\bullet(B_n)\mapsto (A_n\circ B_n)$$ on $\Nw(X)$, where $A\circ B=\{a\circ b:a\in A,\; b\in B\}$. If $\circ$ is an associative (commutative) operation on $X$, then $\bullet$ is an associative (commutative) operation on $\Nw(X)$. Moreover the limit operator $\lim:\Nw(X)\to X$  preserves the operation in the sense that $\lim((A_n)\bullet (B_n))=\lim(A_n)\circ\lim(B_n)$.

Combining this observation with Theorem~\ref{n7.4} we obtains an algebraic characterization of $k^*$-metrizable topological semigroups.

\begin{corollary}\label{n7.6} A  regular (commutative) topological semigroup $X$ is $k^*$-metrizable if and only if there is a continuous homomorphism $h:M\to X$ from a (commutative) metrizable topological semigroup $M$ admitting a section $s:X\to M$ that preserves precompact sets.
\end{corollary}

We recall that a {\em topological semigroup} is a pair $(S,\circ)$ consisting of a topological space $S$ and a continuous associative operation $\circ:S\times S\to S$.

\section{Interplay between $k^*$-metrizable spaces and other generalized metric spaces}\label{s8}

In this section we study the interplay between $k^*$-metrizable spaces and other classes of generalized metric space such as
La\v snev  or $\aleph$-spaces. We remind that a Hausdorff
space $X$ is called a {\em La\v snev space} if it is the image of a
metrizable space under a closed continuous map.

We need to recall the definition of the {\em Arens' fan}, which is the subset
$$S_2=\{(0,0),(\frac1n,0),(\frac1n,\frac1m):n,m\in\IN\}$$ of the plane,
endowed with the strongest topology inducing the original topology on each
compact subset $K_n=\{(0,0),(\frac1k,0),(\frac1i,\frac1j):k,j\in\IN,\;i\le
n\}$, $n\in\IN$, see \cite{Arens}.

The following theorem
is close by spirit to the results of \cite{Lin}, \cite{LD}, \cite{NT},
\cite{Ta2}.

\begin{theorem}\label{n8.1} Let $X$ be a Hausdorff space. Then
\begin{enumerate}
\item  $X$ is a La\v snev space iff $X$ is a Fr\'echet-Urysohn $k^*$-metrizable space iff $X$ is a Fr\'echet-Urysohn $\cs^*$-metrizable space iff $X$ is a sequential $\cs^*$-metrizable space not containing a topological copy of the Arens' fan $S_2$;
\item  $X$ is $\cs$-metrizable if and only if $X$ is a $\cs^*$-metrizable space
containing no (sequentially closed)
subspace sequentially homeomorphic to the fans $S_\w$ or $S_2$;
\item $X$ is $k$-metrizable  if and only if $X$ is a $k^*$-metrizable space
containing no (sequentially closed)
subspace sequentially homeomorphic to the fans $S_\w$ or $S_2$;
\item $X$ is metrizable iff $X$ is a sequential $\cs$-metrizable space iff $X$ is a $k$-metrizable $k$-space.
\end{enumerate}
\end{theorem}

\begin{proof} 1. If $X$ is a La\v snev space, then there is a closed map $\pi:M\to X$ of a metrizable space onto $X$ and $X$ is a Fr\'echet-Urysohn space by \cite[2.4.G]{En}. By Proposition~\ref{n1.5} each section $s:M\to X$ is $\cs^*$-continuous. Consequently, $X$ is $\cs^*$-metrizable. Being Fr\'echet-Urysohn, $X$ is a $k^*$-metrizable space. Since the Arens' fan $S_2$ is not Fr\'echet-Urysohn, $X$ cannot contain a topological copy of $S_2$.

Now assume conversely that $X$ is a Fr\'echet-Urysohn $\cs^*$-metrizable space and find a metrizable space $M$ and a map $\pi:M\to X$ having a $\cs^*$-continuous section. By Theorem~\ref{n1.6} the map $\pi$ is inductively closed and hence $X$ is a La\v snev space. If $X$ is a sequential space containing no copy of the Arens' fan, then $X$ is Fr\'echet-Urysohn by \cite{Si}.
\smallskip

2. Assume that $X$ is $\cs$-metrizable, which is equivalent to the metrizability of the sequential coreflexion $sX$ of $X$. We need to show that  $X$ contains no subspace $W$
sequentially homeomorphic to the Fr\'echet--Urysohn or Arens fan. Actually,
we can prove more:  $X$ contains no subspace sequentially homeomorphic to
a non-metrizable sequential space $Z$. Suppose that this is not true.
Then we can find a
sequential homeomorphism $h\colon\, W\to Z$ of a subspace $W\subset X$ onto $Z$. Since the identity map $i:sX\to X$ is a sequential homeomorphism, the composition $h\circ i|i^{-1}(W)\colon\, i^{-1}(W)\to Z$ is a sequential homeomorphism. Since
both spaces $i^{-1}(W)$ and $Z$ are sequential, the function
$h\circ i|i^{-1}(W)$ is a
homeomorphism. Then the space $Z$, being homeomorphic to the metrizable
space $i^{-1}(W)$, is metrizable, which is a contradiction.
\smallskip

Now assume conversely that $X$ is a $\cs^*$-metrizable but not $\cs$-metrizable space. We will find a sequentially closed subspace in $X$ sequentially homeomorphic to a Fr\'echet-Urysohn or Arens' fan.
 By Theorem~\ref{n4.3}(1) the sequential coreflexion $sX$ of $X$ is $\cs^*$-metrizable but not metrizable. Since all compact subsets of $sX$ are sequentially compact (by the sequentiality of $sX$), the space $sX$ is $k^*$-metrizable and hence $sX$ is the image of a metrizable space $M$ under a subproper map $\pi:M\to sX$. Let $\sigma:sX\to M$ be a section of $\pi$ that preserves precompact sets. Without loss of generality, $\sigma(sX)$ is dense in $M$.

 The space $sX$ is not $k$-metrizable, being a non-metrizable $k$-space. Consequently, the map $\pi$ fails to be proper and for some compact set $K\subset sX$ the preimage $\pi^{-1}(K)$ is not compact and hence contains an
infinite closed discrete subset $\{z_n\colon\, n\in\IN\}$. The sequentiality of $X$ implies the sequential compactness of the compactum $K$. Passing to a suitable
subsequence we can assume that the sequence $(\pi(z_n))_{n=1}^\infty\subset
K$ converges to a point $x_\infty$ of the sequentially compact space $K$.
Again passing to a subsequence we can assume that either $\pi(z_n)=x_\infty$
for all $n$ or $\pi(z_m)\ne \pi(z_n)\ne x_\infty$ for all $n\ne m$.

Since the set $\{z_n\colon\, n\in\IN\}$ is closed and discrete in $Z$ while $\sigma(K)$
is precompact, only finitely many points $z_n$ belong to $\sigma(K)$. Without
loss of generality, we can assume that $z_n\notin \sigma(K)$ for all $n$ and
thus $\{z_n\colon\, n\in\IN\}\subset Z\setminus \sigma(K)$. Pick pairwise disjoint
neighborhoods $O(z_n)$ of the points $z_n$ in $Z$ such that the collection
$\{O(z_n)\colon\, n\in\IN\}$ is discrete in $Z$. Since $\sigma(sX)$ is dense in $Z$
we can find for every $n\in\IN$ a sequence $(z_{n,m})_{m=1}^\infty\subset
\sigma(X)\cap O(z_n)$ convergent to $z_n$. Then the sequence
$(\pi(z_{n,m}))_{m=1}^\infty$ converges to $\pi(z_n)$. It can be shown that
the set $W=\{x_\infty,\pi(z_n),\pi(z_{n,m})\colon\, n,m\in\IN\}$ is
closed in $sX$. We claim that $W$ is sequentially homeomorphic either to
the Fr\'echet--Urysohn fan or to the Arens fan.

Indeed, if $\pi(z_n)=x_\infty$ for all $n$, we define a
homeomorphism $h\colon\, S_\omega\to W$ from the Fr\'echet--Urysohn fan
letting $h(*)=x_\infty$ and
$h(2^{-m},n)=\pi(z_{n,m})$ for $(2^{-m},n)\in S_\w\setminus\{*\}$. Here $*=\{0\}\times \w$ is the non-isolated point of the fan $S_\w=\{2^{-m}:m\le\w\}\times\w/\{0\}\times\w$. If $\pi(z_n)\ne
\pi(z_m)\ne x_\infty$ for all $n\ne m$, we define a  homeomorphism
$h\colon\, S_2\to W$ from the Arens fan
letting $h(0,0)=x_\infty$, $h(\frac1n,0)=\pi(z_n)$, and
$h(\frac1n,\frac1m)=\pi(z_{n,m})$ for all $n,m\in\IN$.

Since the identity map $i:sX\to X$ is a sequential homeomorphism, the image $i(W)$ is a sequentially closed subset of $X$, sequentially homeomorphic to the Fr\'echet-Urysohn or Arens' fans.
\smallskip

3. If $X$ is $k$-metrizable, then it is $k^*$-metrizable and $\cs$-metrizable. By the preceding item, $X$ contains no subspace sequentially homeomorphic to the Fr\'echet-Urysohn or Arens' fan. Now assume conversely that $X$ is a $k^*$-space containing no sequentially closed subspace homeomorphic to the Fr\'echet-Urysohn or Arens' fan. By the preceding item, $X$ is $\cs$-metrizable. Being a $k^*$-space, $X$ is the image of a metrizable space under a subproper map. Consequently all compact subsets of $X$ are metrizable and hence sequentially compact. Now Proposition~\ref{n4.2} implies that $X$ is $k$-metrizable.
\smallskip

4. The last item trivially follows from the fact that the (sequential) $k$-coreflexion of $X$ coincides with $X$ if $X$ is a (sequential) $k$-space.
\end{proof}

Next, we establish the interplay between $k^*$- and $\cs^*$-metrizable spaces and  generalized metric spaces defined with help of $k$- or $\wcs^*$-networks (like $\aleph_0$- or $\aleph$-spaces).
 $\aleph_0$-Spaces were introduced by E.Michael
\cite{Mi} as regular spaces with countable $k$-network. He
also proved that a regular space $X$ is an $\aleph_0$-space if and
only if $X$ is the image of a metrizable separable space under a
compact-covering map. We are going to show that a bit more is
true: each regular $\aleph_0$-space is the image of a metrizable
separable space under a subproper map.  A version of the following theorem was proved by Chuan Liu and Y.Tanaka in \cite{LiTa2}.

\begin{theorem}\label{n8.2} For a regular topological space $X$ the
following conditions are equivalent:
\begin{enumerate}
\item $X$ is an $\aleph_0$-space;
\item $X$ is the image of a metrizable separable space
$M$ under a subproper map $\pi:M\to X$;
\item $X$ is a $k^*$-metrizable space with countable network;
\item $X$ is a $\cs^*$-metrizable space with countable network;
\item the $k$-coreflexion $kX$ of $X$ is a $k^*$-metrizable
space with countable extent;
\item the sequential coreflexion $sX$ of $X$ is a
$\cs^*$-metrizable space with countable extent.
\item $kX$ is the image of a metrizable separable space $M$ under a quotient map $f:M\to kX$.
\end{enumerate}
If $sX$ is regular and (CH) holds, then the conditions (1)--(6) are equivalent to:
\begin{enumerate}
\item[(8)] $sX$ is separable and $\cs^*$-metrizable.
\end{enumerate}
\end{theorem}

\begin{proof} We shall prove the implications $(1)\Ra(3)\Ra(4)\Ra(5)\Ra(6)\Ra(2)\Ra(7)\Ra(1)\Leftrightarrow(8)$.
\smallskip

The implication $(1)\Ra(3)$ follows from Theorem~\ref{n7.4} while $(3)\Ra(4)$ is trivial.
\smallskip

$(4)\Ra(5)$ If $X$ is $\cs^*$-metrizable space with countable network, then each compact subset of $X$ has countable network and hence is metrizable. By Proposition~\ref{n4.2}, the space $X$ is $k^*$-metrizable and so is the $k$-coreflexion $kX$ of $X$, see Proposition~\ref{n3.3}.
The space $X$ has  countable network, and hence  is the image
of a separable metrizable space $Z$ under  a continuous map
$g\colon\, Z\to X$, see \cite[4.9]{Gru}. Observe that the map $g:Z\to kX$ is continuous. Consequently, the space $kX$ has countable network and cannot contain uncountable closed discrete subset.
\smallskip

$(5)\Ra(6)$ Assume that the $k$-coreflexion $kX$ of $X$ is a $k^*$-metrizable space with countable extent. Then each compact subset of $kX$ is metrizable and hence the space $kX$ is sequential and coincides with the sequential coreflexion $sX$ of $X$.
Then $sX=kX$ is a $\cs^*$-metrizable space with countable extent.
\smallskip

$(6)\Ra(2)$ Assume that $sX$ is a $\cs^*$-metrizable space with countable extent. By 
Theorem~\ref{n4.3}, the space $X$ is $\cs^*$-metrizable. So we can take  a map $\pi:M\to X$ from a metrizable space $M$ that has a $\cs^*$-continuous section $\sigma:X\to M$ with dense image $s(X)$ in $M$. We may assume that $\sigma(X)$ is dense in $M$. We claim that the space $M$ is separable. Otherwise, we could find an uncountable subset $D\subset \sigma(X)$ that is closed and discrete in $X$. Then the set $\pi(D)$ has finite intersection with each convergent sequence in $sX$ and consequently, $\pi(D)$ is an uncountable closed discrete subset in $sX$, which is not possible as $sX$ has countable extent. Therefore, the space $M$ is separable and consequently, each compact subset of $X$ has countable network and is metrizable. By Proposition~\ref{n1.3}, the map $\pi:M\to X$ is subproper.
\smallskip

$(2)\Ra(7)$ If $\pi:M\to X$ is a subproper map from a metrizable separable space, then $\pi$, considered as a map from $M$ to $kX$ is compact-covering and thus quotient.

$(7)\Ra(1)$ This is due to E.Michael \cite{Mi}. Let $f:M\to kX$ be a quotient map of a metrizable separable space $M$ and let $\mathcal B$ be a countable base of the topology of $M$. We claim that $f(\mathcal B)=\{f(B):B\in\mathcal B\}$ is a $k$-network for $X$. Take any open set $U\subset X$ and a compact subset $K\subset U$. Let $\{C_i:i\in\w\}$ be an enumeration of the family $\C=\{f(B):B\in\mathcal B,\; f(B)\subset U\}$. We claim that $K\subset C_1\cup\dots\cup C_n$ for some $n\in\w$. Assuming the converse, we can find a sequence $(x_n)$ in $K$ such that $x_n\notin C_1\cup\dots\cup C_n$ for all $n\in\w$. The compactum $K$, being the image of a separable metrizable space, has countable network and hence is metrizable. Then, passing to a subsequence, we may assume that $(x_n)$ converges to some point $x\in K$ such that $x\ne x_n$ for all $n\in\w$. Since $f:M\to kX$ is quotient and $A=\{x_n:n\in\w\}$ is not closed in $kX$, the set $\pi^{-1}(A)$ is not closed in $M$. Let $z\in \overline{\pi^{-1}(A)}\setminus \pi^{-1}(A)$. Then $z\in \pi^{-1}(K)\subset \pi^{-1}(U)$. Choose $B\in\mathcal B$ with $z\in B\subset \pi^{-1}(U)$. Then $f(B)\in\C$ and $f(B)$ contains infinitely many $x_n$'s.  This contradicts the way the $x_n$'s were chosen.
\smallskip

The equivalence $(1)\Leftrightarrow(8)$ follows from Corollary~\ref{nn5.9}.
\end{proof}

\begin{remark}
{\rm
The countably compact
space $X\subset\beta\IN$ from Remark~\ref{frolik} is a separable
$k$-metrizable space without countable network (and thus $X$ is not an
$\aleph_0$-space). Another example of this sort is the Lindel\"ofication $\{\infty\}\cup D$ of an uncountable discrete space $D$ which is Lindel\"of and $k$-metrizable but has no countable network. These examples show that
the network weight $nw(X)$ in Theorem~\ref{n8.2}(3) cannot be replaced by the density $d(X)$ or Lindel\"of number $l(X)$. We do not know if this can be done for the hereditary density or hereditary Lindel\"of number.
}\end{remark}

\begin{question} Has a regular $k^*$-metrizable space $X$
countable network if it is hereditarily separable and hereditarily Lindel\"of?
\end{question}

$\aleph_0$-Spaces form a subclass in a more general class of
$\aleph$-spaces introduced by  O'Meara \cite{OM}. We recall that
a regular space $X$ is defined to be an {\em $\aleph$-space} if it
has a $\sigma$-locally-finite $k$-network, see \cite{Lin90}. Since
each locally-finite family is compact-finite, we may apply
Theorem~\ref{n7.4} to conclude that the class of
$k^*$-metrizable spaces contains all $\aleph$-spaces.

\begin{corollary}\label{n8.5} Each $\aleph$-space is $k^*$-metrizable.
\end{corollary}

It is interesting to mention that $\aleph$- and $k$-spaces can be
characterized as regular $k$-spaces possessing a
$\sigma$-compact-finite \underline{closed} $k$-network. Such a
characterization holds because each compact-finite family of
closed subsets of a $k$-space is locally-finite. In spite of the
similarity in characterizations of $k^*$-metrizable and
$\aleph$-spaces, there are $k^*$-metrizable spaces which are
not $\aleph$-spaces. The simplest example is the uncountable
sequential fan $S_{\w_1}$ which is a regular $k^*$-metrizable
$k$-space that fails to be an $\aleph$-space. 
The following theorem helps to detect $\aleph$-spaces among $k^*$-metrizable spaces with small character.

\begin{proposition} A regular paracompact $k^*$-metrizable $k$-space $X$ with $\ext_\chi(X)\le\aleph_0$ is an $\aleph$-space.
\end{proposition}

\begin{proof}  Fix a subproper map $\pi:M\to X$ of a metric space $(M,d)$ possessing a precompact-preserving section $s:X\to M$ with dense image $s(X)$ in $M$. For every $n\in\w$ fix a locally finite open cover $\U_n$ of the metric space $M$ by sets of diameter $<2^{-n}$. By Proposition~\ref{nn5.6},  each point $x\in X$ has a neighborhood $O_n(x)$ meeting at most countably many sets $s^{-1}(U)$, $U\in\U_n$. The paracompactness of $X$ yields a $\sigma$-discrete open cover $\V_n$ of $X$ refining the cover $\{O_n(x):x\in X\}$. Write $\V_n=\bigcup_{m\in\w}\V_{n,m}$, where each collection $V_{n,m}$ is discrete in $M$. For every $V\in \V_{n,m}$ the image $s(V)$ lies in the union of a countable subcollection $\U_V\subset\U_n$. Enumerate this subcollection as $\U_V=\{U_{V,k}:k<|\U_V|\}$ and put $U_{V,k}=\emptyset$ for $k\ge|\U_V|$. Finally let 
$$\mathcal N_{n,m,k}=\{s^{-1}(U_{V,k}): V\in\V_{n,m}\}$$
and observe that the family $\mathcal N_{n,m,k}$ is discrete in $X$.
On the other hand, it can be easily shown that the union $\mathcal N=\bigcup_{n,m,k\in\w}\mathcal N_{n,m,k}$ is a $k$-network for $X$. Thus $X$ admits a $\sigma$-discrete $k$-network and hence $X$ is an $\aleph$-space.
\end{proof}

\begin{problem} Characterize $\aleph$-spaces within the class of
$k^*$-metrizable spaces. In particular, is a regular
$k^*$-metrizable $k$-space $X$ an $\aleph$-space if  $\alpha_4(X)\le\aleph_0$? (The answer is affirmative if $X$ carries a topological group structure, see Theorem~\ref{n14.2}).
\end{problem}

Next we investigate the interplay between $k^*$-metrizable and semi-stratifiable spaces.
We recall that a regular space $X$ is {\em semi-stratifiable} if there is a function $G$ which assigns to each $n\in\w$ and a closed subset $F\subset X$, an open set $G(n,F)$ containing $F$ such that
\begin{itemize}
\item[(a)] $F=\bigcap_n {G(n,F)}$;
\item[(b)] $F\subset K\;\Ra\; G(n,F)\subset G(n,K)$.
\end{itemize}
If also
\begin{itemize}
\item[(a')] $F=\bigcap_n\overline{G(n,F)}$,
\end{itemize}
then $X$ is called {\em stratifiable}, see \cite[\S5]{Gru}.
\smallskip

By Theorem~\cite[5.9]{Gru}, each regular space with countable network is semi-stratifiable. In particular, each $\aleph_0$-space is semi-stratifiable.
However, an $\aleph_0$-space need not be stratifiable.
 The simplest example is the Banach space $(l^1,\mathrm{weak})$ endowed with the weak topology, which is not stratifiable according to \cite{Ga}. Another example is the space $C_k(\IQ)$ of continuous functions on rationals, endowed with the compact-open topology. By a recent result of P.Nyikos \cite{Nyi}, $C_k(\IQ)$ fails to be a stratifiable space but is an $\aleph_0$-space according to a classical result of E.Michael \cite{Mi}. Also $k^*$-metrizable spaces need not be semi-stratifiable.

\begin{example} The Frolik's space $X$ from Example~\ref{frolik}, being separable and not Lindel\"of, fails to be subparacompact and hence is not semi-stratifiable, see \cite[5.11]{Gru}. Nonetheless, $X$ is a $k$-metrizable space.
\end{example}

In the meantime, $\cs^*$-metrizable spaces are semi-stratifiable in a weaker sequential sense. We shall say that a subset $U$ of a space $X$ is a {\em sequential barrier} for a set $F\subset X$ if for any sequence $(x_n)\subset X$ convergent to a point $x_\infty\in F$ there is $m\in\w$ such that $x_n\in U$ for all $n\ge m$. It is clear that each open set $U$ is a sequential barrier for any set $F\subset U$. The converse is true in Fr\'echet-Urysohn spaces.

\begin{theorem}\label{n8.8} For each  $\cs^*$-metrizable space $X$ there is a function $G$ which assigns to each $n\in\w$ and a closed subset $F\subset X$, a  sequential barrier $G(n,F)$ for $F$ such that
\begin{itemize}
\item $F=\bigcap_n{G(n,F)}$;
\item $F\subset K\;\Ra\; G(n,F)\subset G(n,K)$.
\end{itemize}
\end{theorem}

\begin{proof} Given a $\cs^*$-metrizable space $X$, find a map $\pi:M\to X$ of a metrizable space $M$ having a $\cs^*$-continuous section $s:X\to M$. We may additionally assume that the set $Z=s(X)$ is dense in $M$.

The space $M$, being metrizable, is stratifiable and hence has a  function $G_M$ assigning to each
 $n\in\w$ and a closed subset $F\subset X$, an open set $G_M(n,F)$ containing $F$ and such that
\begin{itemize}
\item $F=\bigcap_n \overline{G_M(n,F)}$;
\item $F\subset K\;\Ra\; G_M(n,F)\subset G_M(n,K)$.
\end{itemize}
Using the stratifying function $G_M$ define a function $G$ letting $$G(n,F)=s^{-1}( G_M(n,\pi^{-1}(F)\big),$$ where $n\in\w$ and $F$ is a closed subset of $X$. We claim that $G(n,F)$ is a sequential barrier at $F$. Indeed, assume that there is a  sequence $(x_n)\in X\setminus G(n,F)$ convergent to a point $x_\infty\in F$. The $\cs^*$-continuity of the function $s$ implies the existence of a limit point $z_\infty$ for the sequence $(s(x_n))$. Then $\pi(z_\infty)=x_\infty$ by the continuity of $\pi$.
Since $G_M(n,\pi^{-1}(F))$ is an open neighborhood of $z_\infty$, infinitely many points of the sequence $(s(x_n))$ belong to $G_M(n,\pi^{-1}(F))$ and then infinitely many points of the sequence $(x_n)$ belong to the set $G(n,F)=s^{-1}(G_M(n,\pi^{-1}(F))$ which contradicts the choice of the sequence $(x_n)$.

In obvious way the two properties of the function $G_M$  imply the corresponding two properties of the function $G$ indicated in Theorem~\ref{n8.8}.
\end{proof}

Also $\cs^*$-spaces have a property
close to the monotonical normality. We recall that a topological space $X$
is called {\em monotonically normal} if to each pair $(H,K)$ of disjoint
closed subsets of $X$ one can assign an open set $D(H,K)\subset X$ such
that
\begin{itemize}
\item $H\subset D(H,K)\subset\overline{D(H,K)}\subset X\setminus K$;
\item if $H\subset H'$ and $K\supset K'$, then $D(H,K)\subset
D(H',K')$.
\end{itemize}

The function $D$ is called a {\em monotone normality operator} for $X$,
see \cite[5.15]{Gru} or \cite{Collins}. Observe that for
such an operator $D$ we have
$$
\cl(D(H,K))\cap K=\emptyset\quad\hbox{and}\quad
\cl(X\setminus D(H,K))\cap H=\emptyset
$$
for any disjoint closed subsets $H,K$ of $X$.

According to a characterization
in \cite[5.22]{Gru}, a space $X$ is stratifiable if and only if the product
$X\times S_1$ of $X$ and a convergent sequence $S_1=\{2^{-n}:n\le\infty\}$ is monotonically normal.
\medskip

We define a topological space $X$ to be {\em monotonically sequentially
normal} if to each pair $(H,K)$ of disjoint closed subsets of $X$ one can
assign a subset $D(H,K)\subset X$ such that
\begin{itemize}
\item $\scl(D(H,K))\cap K=\emptyset$ and $\scl(X\setminus D(H,K))\cap
H=\emptyset$;
\item if $H\subset H'$ and $K\supset K'$, then $D(H,K)\subset
D(H',K')$.
\end{itemize}

Here $\scl(A)$ stands for the 1-st sequential closure of a subset
$A\subset X$. Since $\scl(A)=\cl(A)$
for every subset $A$ of a
Fr\'echet--Urysohn space, we see that each monotonically sequentially
normal Fr\'echet--Urysohn space is monotonically normal.

\begin{proposition}\label{monseq}
Each $\cs^*$-metrizable space is monotonically
sequentially normal.
\end{proposition}

\begin{proof} Given a $\cs^*$-metrizable space let $\pi:M\to X$ be a
continuous map from a metrizable space $M$ and let
$s\colon\, X\to M$ be a $\cs^*$-continuous section of
$\pi$. Without loss of generality, the image $Z=s(X)$ is dense
in~$M$. The metrizable space $M$ is monotonically normal, hence admits a
monotone normality operator $D_M$.

Given a pair $(H,K)$ of disjoint closed subsets of $X$, let
$$
D(H,K)=s^{-1}\big(D_M(\pi^{-1}(H),\pi^{-1}(K))\big).
$$
 If $(H',K')$ are disjoint closed
subsets of $X$ with $H\subset H'$ and $K\supset K'$, then
$$
D_M(\pi^{-1}(H),\pi^{-1}(K))\subset D_M(\pi^{-1}(H'),\pi^{-1}(K'))
$$
 and hence
$D(H,K)\subset D(H',K')$.

It can be shown that
$$
\scl(D(H,K))\subset \pi(\overline{Z\cap
D_M(\pi^{-1}(H),\pi^{-1}(K))})= \pi(\overline{D_M(\pi^{-1}(H),\pi^{-1}(K))}).
$$
Since
$\overline{D_M(\pi^{-1}(H),\pi^{-1}(K))}\cap \pi^{-1}(K)=\emptyset$, we get
$\scl(D(H,K))\cap K=\emptyset$.

Similarly, we obtain
$$
\scl(X\setminus D(H,K))=\scl(\pi(Z\setminus
D_M(\pi^{-1}(H),\pi^{-1}(K))))\subset \pi(\overline{Z\setminus
D_M(\pi^{-1}(H),\pi^{-1}(K))}).$$ Since $Z\setminus D_M(\pi^{-1}(H),\pi^{-1}(K))\subset
M\setminus D_M(\pi^{-1}(H),\pi^{-1}(K))$ and the latter set misses $\pi^{-1}(H)$,
we get that $\scl(X\setminus D(H,K))\cap H=\emptyset$.
\end{proof}

It is interesting to remark that there exist $\sigma$-compact
$\aleph_0$-spaces which are  monotonically sequentially normal but
not monotonically normal.

\begin{example} For every infinite-dimensional Banach space $X$
whose  dual $X^*$ is separable
the space $(X,\weak)$ is an $\aleph_0$-space
(see Proposition~\ref{4.14}) and thus is monotonically sequentially
normal, but $(X,\weak)$ is not monotonically normal, see \cite{Ga}.
\end{example}

\section{$k^*$-Metrizable function spaces}\label{s9}

Given topological spaces $X$ and $Y$ let $C(X,Y)$ be the set
of all continuous functions from $X$ to $Y$. There is a general
way of topologizing $C(X,Y)$ using a family $\mathcal K$ of
compact subsets of $X$. Namely we endow the set $C(X,Y)$ with the
topology generated by the sub-base consisting of the sets $$
[F,U]=\{f\in C(X,Y):f(F)\subset U\}$$ where $F$ is a closed subset
of a compact set $K\in\mathcal K$ and $U$ runs over open subsets
of $Y$, and denote the obtained topological  space by $C_{\mathcal
K}(X,Y)$. It is easy to see that the space $C_{\mathcal K}(X,Y)$
is Hausdorff if $Y$ is Hausdorff and $\cup\mathcal K$ is dense in
$X$. If, in addition, the space $Y$ is (completely) regular, then
so is the function space $C_{\mathcal K}(X,Y)$, see Theorems
3.4.13 and 3.4.15 of \cite{En}.

If the union $\cup\mathcal K\subset X$ is separable and  all compact subsets of a space $Y$ are metrizable, then all compact subsets of the function space $C_{\mathcal K}(X,Y)$ are metrizable too. Indeed, take any countable dense subset $D\subset \cup\mathcal K$ and consider the restriction operator $R:C_{\mathcal K}(X,Y)\to Y^D$ which is continuous and injective. Then each compact subset $K\subset C_{\mathcal K}(X,Y)$ is homeomorphic
to the compact subset $R(K)$ of the countable product $Y^D$ and hence is metrizable.

Two particular choices of the family $\mathcal K$ are especially important. If $\mathcal K$ is the family of all finite (resp. compact) subsets of $X$, then we write $C_p(X,Y)$ (resp. $C_k(X,Y)$) instead of $C_{\mathcal K}(X,Y)$ and say about the topology of pointwise convergence (the compact-open topology) on the set $C(X,Y)$.

By a classical result of E.Michael \cite{Mi} the function space $C_k(X,Y)$ is an $\aleph_0$-space for
any $\aleph_0$-spaces $X$ and $Y$. This result was generalized to $\aleph$-spaces by O'Meara \cite{OM}.

\begin{theorem}[O'Meara]\label{n9.1} The space $C_k(X,Y)$ of continuous functions from an $\aleph_0$-space to an $\aleph$-space is an $\aleph$-space, and hence $C_k(X,Y)$ is $k^*$-metrizable.
\end{theorem}

We shall derive from this theorem a bit more general result
treating subspaces of $C_{\mathcal K}(X,Y)$.

\begin{theorem}\label{n9.2} Let $X$ be an $\aleph_0$-space, $Y$ be an $\aleph$-space and $\mathcal K$ be a family of compact subsets of a space $X$ with dense union
$\cup\mathcal K$ in $X$. A
subspace $\mathcal F\subset C_{\mathcal K}(X,Y)$ is $k^*$-metrizable if the evaluation operator
$e:\mathcal F\times X\to Y$, $e:(f,x)\mapsto f(x)$ is sequentially
continuous.
\end{theorem}

\begin{proof} Let $(\FF,\tau_k)$ (resp. $(\FF,\tau_\KK)$) denote the set $\FF$ endowed with the topology inherited from the function space $C_k(X,Y)$ (resp. $C_\KK(X,Y)$). It is clear that the identity map $\id:(\FF,\tau_k)\to(\FF,\tau_\KK)$ is continuous. The sequential continuity of the evaluation operator $e:\FF\times X\to Y$ and the metrizability of compact subsets of $X$ imply that the inverse function  $\id^{-1}:(\FF,\tau_\KK)\to(\FF,\tau_k)$ is sequentially continuous.
Moreover, this function is continuous on each compact subset of $(\FF,\tau_\KK)$ because compact subsets of the function space $C_\KK(X,Y)\supset\FF$ are metrizable (this follows from the separability of $\cup\KK$ and metrizability of compact subsets of $Y$).
Therefore the idenity map $\id: (\FF,\tau_k)\to(\FF,\tau_\KK)$ is proper.

By Theorem~\ref{n9.1}, the space $C_k(X,Y)$ is $k^*$-metrizable and so is its subspace $(\FF,\tau_k)$. Since the $k^*$-metrizability is preserved by proper maps, the space $\FF=(\FF,\tau_\KK)$ is $k^*$-metrizable.
\end{proof}

Unfortunately we do not know the answer to the following intriguing 

\begin{question} Is $C_k(X,Y)$ $k^*$-metrizable for any $\aleph_0$-space $X$ and any regular $k^*$-metrizable space $Y$?
\end{question}

As an application of Theorem~\ref{n9.2} we prove the $k^*$-metrizability of spaces of continuous homomorphisms between topological groups.
For topological groups $G,H$ and a family $\mathcal K$ of compact subsets of $G$ let  $\Hom_{\mathcal K}(G,H)$ be the subspace of the function space $C_{\mathcal K}(G,H)$,  consisting of all continuous group homomorphisms.

\begin{corollary}\label{n9.6} Let $G,H$ be topological groups such that $G$ is a Baire $\aleph_0$-space and $H$ is an $\aleph$-space. Then for any family $\mathcal K$ of compact subsets of $G$ with  $\cup\mathcal K=G$ the space $\Hom_{\mathcal K}(G,H)$ is $k^*$-metrizable.
\end{corollary}

\begin{proof} This corollary will follow from Theorem~\ref{n9.2} as soon as
we check that the evaluation operator $e:\Hom_{\mathcal
K}(G,H)\times G\to H$ is sequentially continuous. Take any two
compact sets $\mathcal A\subset\Hom_{\KK}(G,H)$ and $B\subset G$.
Since $\cup\mathcal K=G$, the set $\mathcal A$ is compact in
$C_p(G,H)$. Applying Troallic's Theorem 4.4 \cite{Tro}, we
conclude that $\mathcal A$ is compact in $C_k(G,H)$ which implies
that the topology of pointwise convergence on $\mathcal A$
coincides with the compact-open topology. This implies that the
restriction of the evaluation operator $e$ to $\A\times B$ is
continuous and hence $e$ is sequentially continuous.
\end{proof}

\section{Subproper maps between spaces of measures}\label{s10}

By a {\em measure} on a topological space $X$ we understand a
countably additive function $\mu\colon\, \BB(X)\to[0,\infty)$ defined on
the $\sigma$-algebra of Borel subsets of $X$. A measure $\mu$ on $X$ is
called a {\em probability measure} if $\mu(X)=1$. A measure $\mu$ on $X$
is {\em Radon} if for every $\e>0$ and every Borel subset $B\subset X$
there is a compact subset $K\subset B$ with $\mu(B\setminus K)<\e$.

For a Hausdorff topological space $X$ by $P_r(X)$ we denote the space of
all probability Borel Radon measures on $X$ endowed with the topology
generated by the sub-basis consisting of the sets $\{\mu\in
P_r(X)\colon\, \mu(U)>a\}$ where $a\in\IR$ and $U$
runs over open subsets of~$X$.

It is well-known that for a Tychonoff  space $X$ the topology of $P_r(X)$ is
generated by the sub-basis consisting of the sets $\{\mu\in
P_r(X)\colon\, \int_Xf\,d\mu>0\}$ where $f$ runs over all bounded real-valued
continuous functions on $X$. For a compact Hausdorff space $X$ the space
$P_r(X)$ is known to be compact and Hausdorff \cite[8.5.1]{Bo},
\cite[Ch.~8]{B2003}, \cite[VII.3.5]{FF} or \cite{Fe1};
for a metrizable space $X$ the space $P_r(X)$ is metrizable
\cite[8.5.1]{Bo}, \cite[Ch.~8]{B2003}, or \cite{Ba}.

For a Borel function $f\colon\, X\to Y$ between topological spaces let
$P_r(f)\colon\, P_r(X)\to P_r(Y)$ be the function assigning to each measure
$\mu\in P_r(X)$ the measure $\eta=P_r(f)(\mu)$ such that
$\eta(B)=\mu(\pi^{-1}(B))$ for any Borel subset $B\subset Y$. It is
well-known that for every continuous map $f\colon\, X\to Y$ between Hausdorff spaces
the function $P_r(f)\colon\, P_r(X)\to P_r(Y)$ is continuous;
for an injective map
$f$ between Tychonoff  spaces the map $P_r(f)$ is injective;
for a surjective map $f$ between compact Hausdorff spaces the map $P_R(f)$
is surjective; for a topological embedding $f\colon\, X\to Y$ of
Tychonoff  spaces the map $P_r(f)$ is a topological embedding
(see \cite[Ch.~8,9]{B2003}).

A family $\MM\subset P_r(X)$ of measures on $X$ is called {\em uniformly
tight} if for every $\e>0$ there is a compact subset $K_\e\subset X$ such
that $\mu(X\setminus K_\e)<\e$ for all $\mu\in\MM$. Every uniformly tight
collection $\MM\subset P_r(X)$ of measures on a Tychonoff  space is
precompact in $P_r(X)$. The converse is not always true \cite{Preiss}. A
Tychonoff  space $X$ is called {\em Prohorov} if each compact subset of
$P_r(X)$ is uniformly tight. If each convergent sequence in $P_r(X)$ is
uniformly tight, then the space $X$ is called {\em sequentially Prohorov}.
The class of Prohorov spaces includes all \v Cech-complete spaces (hence
all locally compact spaces and all complete metric spaces) and all
$k_\omega$-spaces, see \cite[8.3.5]{Bo}, \cite[Ch.~8]{B2003}.
However, the space $\IQ$ of
rational numbers is not Prohorov, see \cite{Preiss}. On the other
hand, each Tychonoff  space of countable type (in particular, each metrizable
space) is sequentially Prohorov \cite[8.3.15]{Bo}. We recall that a
topological space $X$ has countable type if each point $x\in X$ lies in a
compact subset of $X$ possessing a countable fundamental system of
neighborhoods, see \cite[3.1.E]{En}.

\begin{theorem}\label{1.12}
Suppose a map $f\colon\, X\to Y$ between
Tychonoff  spaces admits a section $s\colon\, Y\to X$ that preserves precompact
sets. Then the map $P_r(f)\colon\, P_r(X)\to P_r(Y)$ admits a section $l\colon\, P_r(Y)\to
P_r(X)$ such that $l(\mu)(\overline{s(K)})=\mu(K)$ for every compact
subset $K\subset Y$ and every measure $\mu\in P_r(Y)$.
\end{theorem}

\begin{proof} Let $\beta f\colon\, \beta X\to\beta Y$ be the
continuous extension of the map $f$ onto the Stone-\v Cech compactifications
of $X$ and $Y$. Since the functor $P_r$ preserves embeddings \cite[2.4]{Ba}, we
can identify the space $P_r(X)$ with a subspace of $P_r(\beta X)$.
Given a probability Radon measure $\mu$ on $Y$ and a
compactum $K\in\KK(Y)$ consider the closed
subset
$$
M_K=\{\eta\in P_r(\beta X)\colon\,  P_r(\beta f)(\eta)=\mu\quad
\hbox{and}\quad \eta(\overline{s(K)})\ge\mu(K)\}
$$
of the compact space $P_r(\beta X)$. We claim that the intersection $\bigcap_{K\in\KK(Y)}M_K$ is not
empty. Since the sets $M_K$ are closed in the compact space $P_r(\beta X)$, it
suffices to verify that $\bigcap_{K\in\FF}M_K\ne\emptyset$ for every finite
collection $\FF\subset\KK(Y)$. Let $\A$ be the smallest algebra of subsets of
$X$ containing the family $\FF$ and let $\{A_1,\dots,A_m\}$ denote the set of
atoms (that is the minimal non-empty sets) in $\A$. It is clear that the
sets $A_i$
are pairwise disjoint Borel subsets of $X$.

Since the measure $\mu$ is Radon, for every $i\le m$ we can find a countable
collection $\KK_i$ of pairwise disjoint compact subsets of $A_i$ such that
$\mu(A_i)=\sum_{K\in\KK_i}\mu(K)$. Next, for every $K\in\KK_i$ find a measure
$\eta_K$ on the compactum $\overline{s(K)}$ whose image under the map $f$ is the
restriction $\mu|K$ of the measure $\mu$ onto the compactum $K$. Finally, let
$\eta=\sum_{i=1}^m\sum_{K\in\KK_i}\eta_K$. It is easy to see that
$P_r(\beta f)(\eta)=\mu$ and $\eta(\overline{s(K)})\ge \mu(K)$ for every $K\in\FF$.
Thus $\eta\in\bigcap_{K\in\FF}M_K\ne\emptyset$ and the intersection
$\bigcap_{K\in\KK(Y)}M_K$ contains some measure $l(\mu)$.

Thus we define a function $l\colon\, P_r(Y)\to P_r(\beta X)$.
It is clear that $P_r(\beta f)(l(\mu))=\mu$ and $l(\mu)
(f^{-1}(K))\ge l(\mu)(\overline{s(K)})\ge \mu(K)=l(\mu) (f^{-1}(K))$ for each
$K\in\KK(Y)$. Let us show that the measure $l(\mu)$ can be identified with
a Radon measure on $X$. Since the measure $\mu$ is Radon, for every
$\varepsilon>0$ there is a compact subset $K\subset Y$ with
$\mu(K)>1-\varepsilon$. Then
$l(\mu)(\overline{s(K)})\ge\mu(K)>1-\varepsilon$ and
thus $l(\mu)$ is a Radon measure on $X$.
\end{proof}

\begin{corollary}\label{1.13} Let $f\colon\, X\to Y$ be a subproper map between Tychonoff  spaces. Then the map
$P_r(f)\colon\, P_r(X)\to P_r(Y)$ admits a section that preserves uniformly tight
families of measures. Consequently, if the space $Y$ is Prohorov, then the
map $P_r(f)$ is subproper.
\end{corollary}

In spite of the fact that a metric space does not need to be Prohorov, we shall
show that the functor $P_r$ preserves subproper maps
between metrizable spaces.\begin{theorem}\label{1.14a} Suppose  $f\colon\, X\to Y$ is a map from a Tychonoff space $X$ onto a sequentially Prohorov space $Y$. If $f$ has a precompact-preserving section, then the map $P_r(f)\colon\, P_r(X)\to P_r(Y)$ has a $\cs^*$-continuous section.
\end{theorem}

\begin{proof} By Corollary~\ref{1.13}, the map $P_r(f)\colon\, P_r(X)\to P_r(Y)$
admits a section $s\colon\, P_r(Y)\to P_r(X)$ preserving uniformly tight families
of measures. Since $Y$ is sequentially Prohorov, each convergent sequence $(\mu_n)_{n\in\w}$ in $P_r(Y)$ is uniformly tight and hence its image $\{s(\mu_n)\}_{n\in\w}$, being uniformly tight, is precompact  in $P_r(X)$ and thus has an accumulation point in $P_r(X)$, which means that $s$ is $\cs^*$-continuous.
\end{proof}

Combining Theorem~\ref{1.14a} with Theorem~\ref{n1.4} we get

\begin{corollary}\label{1.14} Let $f\colon\, X\to Y$ be a subproper
map between Tychonoff  spaces. Then the map $P_r(f)\colon\, P_r(X)\to P_r(Y)$ is subproper provided the space $P_R(X)$ is $\mu$-complete, $Y$ is
sequentially Prohorov, and each compact subset of $P_R(Y)$ is sequentially
compact.
\end{corollary}

\begin{remark}
{\rm
It should be mentioned that for a $\mu$-complete Tychonoff
space $X$ the space $P_r(X)$ does not need to be $\mu$-complete:
  according to
\cite{Fe2} the space $P_r(\IR^{\omega_1})$ contains a closed topological
copy of the ordinal segment $[0,\omega_1)$ and hence is not
$\mu$-complete.
}\end{remark}

A topological space $X$ is called {\em submetrizable} if it admits a
continuous bijective map $f\colon\, X\to M$ onto a metrizable space $M$.

\begin{theorem}\label{1.16} If $f\colon\, X\to Y$ is a subproper map from a
metrizable space $X$ onto a submetrizable sequentially Prohorov Tychonoff
space $Y$,
then the map $P_r(f)\colon\, P_r(X)\to P_r(Y)$ is subproper.
\end{theorem}

\begin{proof}
Since the space $P_r(X)$ is metrizable, it is
$\mu$-complete. The space~$Y$, being submetrizable, admits an injective
continuous map $g\colon\, Y\to M$ into a metrizable space $M$. By \cite[2.1]{Ba}, the
map $P_r(g)\colon\, P_r(Y)\to P_r(M)$ is injective. Then for every compact subset
$\KK\subset P_r(Y)$ the restriction $P_r(g)|\KK$ is an embedding of $\KK$
into the metrizable space $P_r(M)$ which implies that the compactum $\KK$
is metrizable and thus sequentially compact. Now we can apply
Corollary~\ref{1.14} to conclude that the map $P_r(f)\colon\, P_r(X)\to P_r(Y)$ is subproper.
\end{proof}

The following simple assertion shows that the requirement on $Y$ to be
sequentially Prohorov is essential in the last theorem.

\begin{proposition}\label{1.17} Let $f\colon\, X\to Y$ be a map between Tychonoff
spaces. If the space $X$ is Prohorov and the map $P_r(f)\colon\, P_r(X)\to P_r(Y)$
is subproper, then the space $Y$ is Prohorov.
\end{proposition}

According to \cite{FGH} the Banach space $l^1$ endowed with the weak
topology is not sequentially Prohorov. Since each weakly convergent
sequence in $l^1$ converges in norm, the identity map
$\id\colon\, l^1\to(l^1,\weak)$ from the norm into the weak topology is proper. Yet, the map $P_r(\id)\colon\, P_r(l^1)\to
P_r(l^1,\weak)$ is not subproper (because the space $l^1$
is Prohorov while $(l^1,\weak)$ is not). Thus we get the
following example complementing Theorem~\ref{1.16}.

\begin{example} There is a bijective proper map $f\colon\, X\to
Y$ from a Polish space $X$ onto a Tychonoff  submetrizable space $Y$ such
that the map $P_r(f)\colon\, P_r(X)\to P_r(Y)$ is not subproper.
\end{example}

The results  on preservation of subproper maps by the functor of Radon measures will help us to detect $k^*$-metrizable spaces of probability measures.

\begin{theorem}\label{n10.9} Let $X$ be a Tychonoff sequentially Prohorov space and $P_r(X)$ be the space of probability Radon measures on $X$.
\begin{enumerate}
\item If $P_r(X)$ is $k^*$-metrizable or $\cs^*$-metrizable, then so is the space $X$;
\item If $X$ is $k^*$-metrizable, then $P_r(X)$ is $\cs^*$-metrizable;
\item If $X$ is submetrizable and $k^*$-metrizable, then $P_r(X)$ is $k^*$-metrizable;
\item $X$ is an $\aleph_0$-space if and only if $P_r(X)$ is an $\aleph_0$-space.
\end{enumerate}
\end{theorem}

\section{$k^*$-metrizable spaces and Skorohod properties}\label{s11}

Following \cite{BK} and \cite{BBK1} we say that a Tychonoff  space
$X$ has the {\em  strong Skorohod property} if to each measure
$\mu\in P_r(X)$ one can assign a Borel function $\xi_\mu\colon\,
[0,1]\to X$ such that $\mu$ is the image of  Lebesgue measure
under the function $\xi_\mu$ and for every  sequence
$(\mu_n)\subset P_r(X)$ convergent to a measure $\mu_0\in P_r(X)$
the function sequence $(\xi_{\mu_n})$ converges to $\xi_{\mu_0}$
almost surely on $[0,1]$. The {\em uniformly tight strong Skorohod
property} is defined by requiring the latter only for uniformly
tight weakly convergent sequences $(\mu_n)$.

A topological space $X$ has the {\em weak Skorohod property} if to
each measure $\mu\in P_r(X)$ one can assign a Borel function
$\xi_\mu\colon\, [0,1]\to X$ such that $\mu$ is the image of
Lebesgue measure under the function $\xi_\mu$ and for every
uniformly tight sequence $(\mu_n)\subset P_r(X)$ the function
sequence $(\xi_{\mu_n})$ contains a subsequence that converges
almost surely in $X$ with respect to Lebesgue measure $\lambda$ on
$[0,1]$.

It was shown in \cite[3.11]{BK} and \cite[2.4]{BBK1} that each
metrizable space has the  strong Skorohod property and each $k$-metrizable space has the uniformly tight strong Skorohod property.

It turns out that the weak Skorohod property is preserved by
subproper  maps.

\begin{theorem}\label{2.2}
A Tychonoff  space $Y$ has the weak Skorohod property if and only
if there is a subproper map $f\colon\, X\to Y$ from a space $X$
possessing the weak Skorohod property.
\end{theorem}

\begin{proof} The ``only if'' part is trivial. To prove the ``if'' part, let
$f\colon\, X\to Y$ be a subproper map. Assume that the space $X$
has the weak Skorohod property, i.e., to each probability Radon
measure $\eta$ on $X$ we can assign a Borel map $\xi_\eta\colon\,
[0,1]\to X$ so that $\eta$ is the image of Lebesgue measure
$\lambda$ under the map $\xi_\eta$ and for every uniformly tight
sequence $(\eta_n)$ of probability Radon measures on $X$ the
function sequence $(\xi_{\eta_n})$ contains a subsequence that
converges almost surely.

By Corollary~\ref{1.13}, the map $P_r(f)\colon\, P_r(X)\to P_r(Y)$
admits a section $s\colon\, P_r(Y)\to P_r(X)$ preserving uniformly
tight families of measures. Let us assign to each measure $\mu\in
P_r(Y)$ the Borel function $\zeta_\mu=f\circ\xi_{s(\mu)}\colon\,
[0,1]\to Y$. It is clear that $$
P_r(\zeta_\mu)(\lambda)=P_r(f\circ
\xi_{s(\mu)})(\lambda)=P_r(f)\circ
P_r(\xi_{s(\mu)})(\lambda)=P_r(f)(s(\mu))=\mu . $$
 Thus $\mu$ is the image of
 Lebesgue measure $\lambda$ under the Borel function $\zeta_\mu$.

Now assume that $(\mu_n)\subset P_r(Y)$ is a uniformly tight
sequence of probability Radon measures on $Y$. Since the section
$s$ preserves uniformly tight families of measures, the sequence
$(s(\mu_n))$ of measures on $X$ is uniformly tight. By our choice
of the maps $\xi_\mu$, the function sequence $(\xi_{s(\mu_n)})$
contains a subsequence $(\xi_{s(\mu_{n_k})})$ that converges
almost surely. Then by the continuity of $f$, the subsequence
$(\zeta_{\mu_{n_k}})$ of the sequence $(\zeta_{\mu_n})$ converges
almost surely in $Y$ which completes the proof of the weak
Skorohod property of the space $Y$.
\end{proof}

Since each $k$-metrizable space has the uniformly tight strong
Skorohod property \cite[2.4]{BBK1} and each metrizable space has
the weak Skorohod property \cite[\S4]{BK}, Theorems~\ref{n10.9} and
\ref{2.2} imply the following result.

\begin{theorem}\label{2.3} Let $X$ be a Tychonoff space.
\begin{enumerate}
\item If $X$ is metrizable, then $X$ has the strong Skorohod property;
\item If $X$ is $k$-metrizable, then $X$ has the uniformly tight strong Skorohod
property.
\item If $X$ is $k^*$-metrizable, then $X$ has the weak Skorohod property.
\end{enumerate}
\end{theorem}

\begin{remark}
{\rm It was shown in \cite[3.4]{BBK1} that the ordinal segment
$[0,\omega_1]$ endowed with the interval topology has the strong
and weak Skorohod properties. Since $[0,\omega_1]$ is a
non-metrizable sequentially compact space, it cannot be $\cs^*$-metrizable. This shows that the
class of Tychonoff  spaces with the weak Skorohod property is
strictly wider than the class of $k^*$-metrizable spaces. }\end{remark}

\section{$k^*$-metrizable locally convex spaces}\label{s12}

In this section we consider locally convex spaces that are $k^*$-metrizable.
 We recall that a locally convex space $X$ is the {\em strict inductive
limit} of a sequence $X_1\subset X_2\subset\dots$ of its subspaces (denoted by
$\ind X_n$) if $X=\bigcup_{n=1}^\infty X_n$ and $X$ carries the strongest
locally convex topology inducing the original topology on each space $X_n$.
It is well known that each bounded subset $B$ of the strict inductive limit
$X=\ind X_n$ lies in some $X_n$, see \cite[II.6.5]{Sch}. Also we will use the
well-known fact that the topology of the
locally convex sum $\oplus_{i\in\I}X_i$ of a
collection of locally convex spaces coincides with the corresponding
box-product topology,
see Exercises II.12 and I.1 in \cite{Sch}. Theorems~\ref{3.3}, \ref{3.3a},
and \ref{n3.10} imply that
the class of $k^*$-metrizable locally convex spaces has the following
permanence properties.

\begin{theorem}\label{4.1}
{\rm1.} Every subspace of a $k^*$-metrizable locally convex space $k^*$-metrizable.

{\rm2.} The countable product of $k^*$-metrizable locally convex is $k^*$-metrizable.

{\rm3.} The locally convex sum $\oplus_{i\in\I}X_i$ of arbitrary collection
of $k^*$-metrizable locally convex spaces is $k^*$-metrizable.

{\rm4.} The projective limit of a sequence of $k^*$-metrizable locally convex spaces is $k^*$-metrizable.

{\rm5.} The strict inductive limit $\ind X_n$ of an increasing sequence of
 $k$-closed $k^*$-metrizable spaces is $k^*$-metrizable.

{\rm6.}
A locally convex space $X$ is $k^*$-metrizable if there is
a sequence $\{X_n\}_{n=1}^\infty$ of compactly closed $k^*$-metrizable subspaces of $X$ such that every convergent sequence of $X$ lies in some
$X_n$.
\end{theorem}

Next, we look for $k^*$-metrizable spaces among operator spaces. Given two linear topological spaces $X$ and $Y$ let $\LL(X,Y)$
denote the linear space of all linear continuous operators from $X$ to $Y$.
For a collection $\SG$ of subsets of~$X$, the topology of uniform
convergence on elements of $\SG$ (or briefly,  the {\em $\SG$-topology})
is the locally convex topology whose base at the origin consists of
the sets
$$
\{L\in\LL(X,Y)\colon\, L(S_1\cup\dots\cup S_n)\subset U\},
$$
 where
$S_1,\dots,S_n\in\SG$ and $U$ is an open neighborhood of the origin in $Y$,
see \cite[III.3]{Sch}.
The space $\LL(X,Y)$ endowed with the $\SG$-topology will be denoted by $\LL_{\SG}(X,Y)$.
The following important theorem follows immediately from Theorem~\ref{n9.2}.

\begin{theorem}\label{n12.2} Let $X$ be a linear topological $\aleph_0$-space and $\mathcal S$ be a family of compact subsets of $X$ with dense union $\cup\mathcal S$ in $X$. For any linear topological $\aleph$-space $Y$ the operator space $\LL_{\mathcal S}(X,Y)$ endowed with the $\mathcal S$-topology is $k^*$-metrizable provided the evaluation operator
$$e:\LL_{\mathcal S}(X,Y)\times X\to Y,\; e:(T,x)\mapsto T(x)$$ is sequentially continuous.
\end{theorem}

A~subset $\FF\subset \LL(X,Y)$ is called {\em
equicontinuous} if for each neighborhood $U\subset Y$ of zero there is a
neighborhood $V\subset X$ of zero such that $f(V)\subset U$ for each $f\in \FF$.
It is easy to see that for any equicontinuous family $\FF\subset\LL(X,Y)$ the evaluation operator $\FF\times X\to Y$ is continuous and hence preserves precompact sets.

A {\em barrel} in a locally convex space $X$ is a closed convex symmetric
subset of $X$. A locally convex space $X$ is {\em barrelled} if each barrel
in $X$ is a neighborhood of zero, see \cite[II.7]{Sch}. It is well known
that each Baire locally convex space is barrelled. On the other hand, there
are separable normed barrelled spaces of the first Baire category, see
\cite[5.3.6]{Ku} or \cite{Ba3}.

\begin{corollary}\label{n12.3} Let $X$ be a linear topological $\aleph_0$-space and $\mathcal S$ be a family of compact subsets of $X$ with $\cup\mathcal S=X$. Assume that either $X$ is a Baire space or else $X$ is a barrelled locally convex space. Then for any linear topological $\aleph$-space $Y$ the space $\LL_{\mathcal S}(X,Y)$ is $k^*$-metrizable.
\end{corollary}

\begin{proof} It suffices to check that the evaluation operator $e:\LL_{\mathcal S}(X,Y)\times X\to Y$ is sequentially continuous. Take two compact sets $\mathcal K\subset \LL_{\mathcal S}(X,Y)$ and $K\subset X$. Since $\cup\mathcal S=X$ for each $x\in X$ the set $\mathcal K(x)=\{T(x):T\in\mathcal K\}$ is compact and hence bounded in $Y$. Now applying
Theorem 4.2 of \cite[Ch.III.4]{Sch} we conclude that the set $\mathcal K$ is equicontinuous, which implies that the restriction $e|\mathcal K\times K$ is continuous.
\end{proof}

 Another corollary of Theorem~\ref{n12.2} concerns the compact-open topology on the operator spaces.

\begin{corollary}\label{n12.4} Let $X$ be a linear topological $\aleph_0$-space and $\mathcal S$ be a family of compact subsets of $X$ containing all convergent sequences in $X$. For any linear topological $\aleph$-space $Y$ the operator space $\LL_{\mathcal S}(X,Y)$ is $k^*$-metrizable.
\end{corollary}

Observing that the dual space $X'$ to a locally convex space $X$ coincides
with the operator space $\LL(X,\IR)$ and taking into account that regular $k^*$-metrizable spaces with countable network are $\aleph_0$-spaces, we get

\begin{corollary}\label{4.3a}
Let $X$ be a linear topological $\aleph_0$-space and $\mathcal S$ be a family of compact subsets of $X$ such that $\cup\mathcal S$ is dense in $X$. The dual
space $X'$ endowed with the $\SG$-topology is an
$\aleph_0$-space if the duality map $X'\times X\to\IR$, $(f,x)\mapsto f(x)$, is sequentially continuous.
\end{corollary}

\begin{corollary}\label{4.5}
Let $X$ be a linear topological $\aleph_0$-space and either $X$ is
Baire or $X$ is locally convex and barrelled. Then the dual space
$X'$ endowed with the topology of pointwise convergence is an
$\aleph_0$-space.
\end{corollary}

\begin{remark}
{\rm It follows from Corollary~\ref{4.5} and
Proposition~\ref{monseq} that for any Baire (or barrelled locally
convex) linear topological $\aleph_0$-space $X$ the dual space
$X'$ endowed with the topology of pointwise convergence is
monotonically sequentially normal. On the other hand, by \cite{Ga}
for any infinite-dimensional locally convex space $X$ the dual
space $X'$ endowed with the topology of pointwise is not
monotonically normal. }\end{remark}

Next, we consider the compact-open topology on dual spaces.

\begin{corollary}\label{n12.8} For any linear topological $\aleph_0$-space $X$ and any collection $\mathcal S$ containing all convergent sequences in $X$ the dual space $X'$ endowed with the topology of $\SG$-convergence is an $\aleph_0$-space.
\end{corollary}

For metrizable separable spaces $X$ the compact-open topology of the dual space $X'$ can be described in more details.

We recall that a topological space $X$ is a {\em $k_\omega$-space} if
there is a countable collection $\C$ of compact subsets of $X$ such that
$X$ carries the strongest topology inducing the original topology on each
compactum $C\in\C$.
It is clear that each
compact subset of a submetrizable space is metrizable. The converse
statement is true for $k_\omega$-spaces:  if each compact subset of a
$k_\omega$-space $X$ is metrizable, then the space $X$ is submetrizable.
It is well-known that each submetrizable $k_\omega$-space is a stratifiable $\aleph_0$-space. 

\begin{theorem}\label{4.7} Let $X$ be a separable metrizable locally
convex space and let
$\SG$ be a collection of totally bounded subsets of~$X$
containing all sequences convergent to zero in~$X$. Then the dual space
$X'$ endowed with the $\SG$-topology is a submetrizable $k_\omega$-space.
\end{theorem}

\begin{proof} Let $\{U_n\}_{n\in\IN}$ be a neighborhood base at the origin
of $X$ such that $U_{n+1}\subset U_n$ for all $n$. Let $U_n^\circ=\{f\in
X'\colon\, |f(x)|\le 1$ for each $x\in U_n\}$, $n\in\IN$. It is clear that
$X'=\bigcup_{n\in\IN}U_n^\circ$ and each set $U_n^\circ$ is
equicontinuous. By the Alaoglu--Bourbaki theorem \cite[III.4.3]{Sch} and
 \cite[Ch.~III, \S4, Theorem 4.7]{Sch} the sets $U_n^\circ$, $n\in\IN$, are
metrizable compacta in the topology of simple convergence on~$X'$.
By the
Banach-Dieudonn\'e theorem \cite[IV.6.3]{Sch}, the $\SG$-topology
coincides with the strongest topology inducing the topology of simple
convergence on equicontinuous subsets of $X'$. It follows from the above
remarks that the dual space $X'$ endowed with the $\SG$-topology is a
submetrizable $k_\omega$-space.
\end{proof}

\begin{corollary}\label{4.8}
Let $\{X_n\colon\, n\in\IN\}$ be a countable cover of
a locally convex space $X$ by separable metrizable linear subspaces and
let $\SG$ be a collection of totally bounded subsets of $X$ such that
$\SG$ includes all  sequences convergent to zero
in the spaces $X_n$ and each
set $S\in\SG$ lies in some $X_n$. Then the dual space $X'$ endowed with
the $\SG$-topology is a stratifiable  $\aleph_0$-space.
\end{corollary}

\begin{proof} For every $n\in\IN$ let $\SG_n=\{S\in\SG\colon\, S\subset X_n\}$. It
follows that the dual space $X$ endowed with the $\SG$-topology is a
subspace of the product $\prod_{n\in\IN}X_n'$, where each dual space $X_n'$
is endowed with the $\SG_n$-topology. By Theorem~\ref{4.7}, each dual
space $X_n'$ is a submetrizable $k_\omega$-space and thus a stratifiable
$\aleph_0$-space. Since the class of stratifiable
$\aleph_0$-spaces is closed under countable products and passage to
 subspaces,
we conclude that $X'$ is a stratifiable $\aleph_0$-space.
\end{proof}

\begin{corollary}\label{4.9}
Let $X$ be a locally convex space that is a
countable union $X=\bigcup_{n=1}^\infty X_n$ of separable metrizable
linear subspaces such that each compact subset of $X$ lies
in some~$X_n$.
Then the dual space $X'$ endowed with the compact-open topology is
a stratifiable  $\aleph_0$-space.
\end{corollary}

We recall that {\em the strong dual topology} on the dual space $X'$ to a
locally convex space $X$ is the $\SG$-topology where $\SG$ is the
collection of all bounded subsets of $X$.
A locally convex space $X$ is called {\em boundedly-compact} if each
closed bounded subset of $X$ is compact, see \cite[8.4.7]{Ed}.
It is clear that for such a space $X$ the strong dual topology
on the dual space $X'$ coincides with the topology of precompact convergence.

\begin{corollary}\label{4.10}
Let $X$ be a locally convex space that is a
countable union $X=\bigcup_{n=1}^\infty X_n$ of separable metrizable
linear subspaces such that each compact subset of $X$ lies in
some~$X_n$. If the space $X$ is boundedly-compact, then
the dual space $X'$ endowed with the strong dual topology is
a stratifiable $\aleph_0$-space.
\end{corollary}

We recall that a locally convex space $X$ is called an
{\em {\rm(LF)}-space} if
$X$ is the strict inductive limit of a sequence of complete metrizable
linear subspaces of $X$.

\begin{corollary}\label{4.11}
If $X$ is a separable {\rm(LF)}-space, then the dual space $X'$ endowed with the topology of simple
convergence and the dual space $X'$ endowed with the compact-open topology
are $\aleph_0$-spaces. If the space $X$ is boundedly-compact,
then $X'$ endowed with the strong dual topology is a stratifiable
$\aleph_0$-space.
\end{corollary}

A classical example of a boundedly-compact (LF)-space is the space $\DD$
of test functions. Its dual $\DD'$ endowed with the strong dual topology is
the space of distributions.

\begin{corollary}\label{4.12}
The space $\DD$ of test functions and its
dual $\DD'$ endowed with the topology of simple convergence are 
$\aleph_0$-spaces. The dual $\DD'$ to $\DD$ endowed with the strong dual
topology is a stratifiable  $\aleph_0$-space.
\end{corollary}

It  was shown in \cite{Ba4}  that the space of distributions $\DD'$ endowed
with the strong dual topology is homeomorphic to the countable power
$(\IR^\infty)^\omega$ of the strict inductive limit $\IR^\infty$ of
finite-dimensional Euclidean spaces.

\begin{remark}
{\rm
It follows from Corollary~\ref{4.9} and \cite{Ga} that for an infinite-dimensional
{\rm(LF)}-space $X$ the topology of compact convergence on the dual $X'$ is
monotonically normal, but the topology of simple convergence is not. Yet,
both topologies are monotonically sequentially normal.
}\end{remark}

\section{The $k^*$-metrizability of the weak topology on Banach spaces}\label{s13}

In this section we detect Banach spaces whose weak topology is $k^*$-metrizable. For a Banach space $X$ by
$(X,\weak)$ we denote $X$ endowed with the weak topology. Since
$(X,\weak)$ is a subspace of the second dual space $X^{**}$ endowed with
the topology of simple convergence, we can apply Corollary~\ref{4.5} to get

\begin{proposition}\label{4.14} If $X$ is a normed space with a separable
dual, then $(X,\weak)$ is an $\aleph_0$-space.
\end{proposition}

There also exists a
Banach space $X$ with non-separable dual such that $(X,\weak)$
is an $\aleph_0$-space.
We recall that a normed space $X$ has the {\it Shur
property} if each weakly convergent sequence of $X$ converges in norm.
A standard example of a Banach space with the Shur property is $l^1$, the
Banach space of all absolutely convergent series. It is known that each
infinite-dimensional Banach space with the Shur property contains a
subspace isomorphic to $l^1$, see \cite[p.~212]{Di}. 
It follows from the definition that for a Banach space $X$ with the Shur property, the sequential coreflexion of $(X,\weak)$ coincides with $X$, which means that $(X,\weak)$ is $\cs$-metrizable. 
It turns out that the converse assertion also is true.


\begin{proposition}\label{4.15} For a normed space $X$ the following conditions are equivalent:
\begin{enumerate}
\item $(X,\weak)$ is $k$-metrizable;
\item $(X,\weak)$ is $\cs$-metrizable;
\item $X$ has the Shur property.
\end{enumerate}
\end{proposition}

\begin{proof} By Eberlein-Smulyan Theorem \cite[p.52]{HHZ} compact subsets of $(X,\weak)$ are sequentially compact. Now we see that the equivalence $(1)\Leftrightarrow(2)$ follows from Proposition~\ref{n4.2}. 

$(2)\Ra(3)$ Assume that the space $(X,\weak)$ is $\cs$-metrizable but $X$ fails to have the Shur property. Then we can find a sequence $(x_n)_{n\in\w}$ on the unit sphere of $X$ that weakly converges to zero. Observe that for every $m\in\IN$ the sequence $(m\cdot x_n)_{n\in\w}$ also weakly converges to zero. The $\cs$-metrizability of $(X,\weak)$ implies that the space $(X,\weak)$ has  Arkhangelski's property $(\alpha_4)$, which allows us to select two increasing number sequences $(m_k)$ and $(n_k)$ such that the sequence $(m_k\cdot x_{n_k})_{k\in\w}$ weakly converges to zero, which is impossible because this sequence is unbounded.

$(3)\Ra(2)$ If $X$ has the Shur property, then the identity map $\id:X\to (X,\weak)$ has sequentially continuous inverse witnessing the $\cs$-metrizability of the space $(X,\weak)$.
\end{proof}

Thus the Shur property is responsible of the $k$-metrizability of the weak topology of a Banach space. The problem of characterization of Banach spaces with $k^*$-metrizable weak topology is more delicate. 

\begin{theorem}\label{4.16}  Let $X$ be a normed space and $B$ be the closed unit ball of $X$. Then the following conditions are equivalent:
\begin{enumerate}
\item $(X,\weak)$ is $k^*$-metrizable;
\item $(X,\weak)$ is $\cs^*$-metrizable;
\item $(B,\weak)$ is $k^*$-metrizable;
\item $(B,\weak)$ is $\cs^*$-metrizable.
\end{enumerate}

Moreover, if the completion of $X$ contains no isomorphic copy of $l^1$, then the conditions (1)--(3) are equivalent to

\begin{enumerate}
\item[(5)] $(B,\weak)$ is $\cs$-metrizable;
\item[(6)] $(B,\weak)$ is $k$-metrizable;
\item[(7)] $(B,\weak)$ is metrizable;
\item[(8)] $X^*$ is separable.
\end{enumerate}
\end{theorem}

\begin{proof} The equivalences $(1)\Leftrightarrow(2)$, $(3)\Leftrightarrow(4)$, $(5)\Leftrightarrow(6)$ follow from Proposition~\ref{n4.2} combined with the Eberlein-Smulyan Theorem \cite[p.52]{HHZ}.
The equivalence $(2)\Leftrightarrow(4)$ can be easily derived from 
Theorem~\ref{n4.4}(1,3). The equivalence $(7)\Leftrightarrow(8)$ is well-known while $(7)\Ra(5)\Ra(4)$ are trivial.

So, it remains to prove that $(3)\Ra(7)$ under the assumption that $l^1$ does not embed into $X$. In this case we may apply the
famous Rosenthal $l^1$-Theorem to conclude that bounded subsets of $X$ and $X^2$ are Fr\'echet--Urysohn with respect to
the weak topology, see \cite[pp.~215, 216]{Di}. In particular, the unit ball $B$ of $X$ is
Fr\'echet--Urysohn in the weak topology. Assuming that the space
$(X,\weak)$ is $k^*$-metrizable but not metrizable, we would get that $(B,\weak)$ is a non-metrizable
Fr\'echet--Urysohn $k^*$-metrizable space which by
  \cite[Corollary~1.8]{Na} or Theorem~\ref{n8.1} is a La\v{s}nev space
containing a subspace
homeomorphic to the Fr\'echet--Urysohn fan.
Then the square $(B^2,\weak)$ is a Fr\'echet--Urysohn space containing
the square of the Fr\'echet--Urysohn fan. This implies that the square of the
Fr\'echet--Urysohn fan is a Fr\'echet--Urysohn space which is not true, see
\cite[1.6]{Na}.
\end{proof}

It follows from Theorem 1.c.9 of \cite{LT} that a Banach space $X$ with an
unconditional basis contains an isomorphic copy of $l^1$ if and only if
the dual $X^*$ is not separable.

\begin{question}\label{4.18}
Let $X$ be a Banach space with an
unconditional basis. Is $(X,\weak)$ $k^*$-metrizable?
Does $(X,\weak)$ have the weak Skorohod property?
\end{question}

A standard example of a separable Banach space having a
non-separable dual and
containing no isomorphic copy of $l^1$ is the James tree space~$JT$,
see
\cite[p.~215]{HHZ}, \cite{LSt}. By Theorem~\ref{4.16}, $(JT,\weak)$ fails to be $k^*$-metrizable. Since each separable Banach space (in particular,~$JT$) is
a quotient of~$l^1$, we arrive at the following result.

\begin{proposition}\label{4.19} There is a Banach space $X$
possessing a closed
linear subspace $Y\subset X$ such that $(X,\weak)$ is $k$-metrizable, but $(X/Y,\weak)$ is not even $\cs^*$-metrizable. Thus the $k$-, $k^*$-, $\cs$-, and $\cs^*$-metrizability is not preserves by open maps.
\end{proposition}

In the following proposition we describe a situation where such a
pathology cannot appear.

\begin{proposition}\label{4.20} Let $Y$ be a reflexive subspace of a
Banach space $X$. If $(X,\weak)$ is $k^*$-metrizable, then so is
the space $(X/Y,\weak)$.
\end{proposition}

\begin{proof} Let $Q\colon\, X\to X/Y$ be the quotient operator.
By  \cite[Proposition~1.19]{BL}
there is a continuous section $s\colon\, X/Y\to X$ of $Q$ such
that $s$ is positively homogeneous, i.e.,
 $s(\lambda x)=\lambda s(x)$ for each $x\in X/Y$ and $\lambda>0$.
It follows from the continuity and positive homogeneity of $s$ that
$s$ maps bounded subsets of $X/Y$ onto bounded subsets of $X$.

To finish the proof it remains to apply Theorem 3.3(4) and show that for each
weakly compact subset $K\subset X/Y$ the weak closure of $s(K)$ in $X$ is
weakly compact. By the Davis--Figiel--Johnson--Pe\l czy\'nski
factorization theorem
\cite{DFJP} there is an injective linear continuous operator $T\colon\, R\to X/Y$
from a reflexive Banach space $R$ such that the image $T(B)$ of the unit
ball $B\subset R$ contains the set $K$. In the product $R\times X$
consider the closed linear subspace $Z=\{(y,x)\in R\times X\colon\,  T(y)=Q(x)\}$.
Observe that the projection $\pr_X\colon\, Z\to X$ is injective while the kernel
of the projection $\pr_R\colon\, Z\to R$ is isomorphic to the reflexive space $Y$.
Since the reflexivity is a tree-space property \cite[4.1]{Ca}, we conclude that
the Banach space $Z$ is reflexive. Since $s$ preserves bounded sets, the
image $s(K)$ is bounded in $X$. Then the preimage $\pr_X^{-1}(s(K))\subset
Z$ is contained in the bounded set $(B\times s(K))\cap Z$ of $Z$. Take any
closed bounded convex set $D\subset Z$ with $D\supset \pr_Z^{-1}(s(K))$.
By the reflexivity of $Z$ the set $D$ is weakly compact and so is its
image $\pr_X(D)$ in $X$. Since $s(K)\subset\pr_X(D)$, we conclude that the
weak closure of $s(K)$ in $X$ is weakly compact. Thus
$Q\colon\, (X,\weak)\to(X/Y,\weak)$ is a subproper map and we can
apply Theorem~\ref{3.3}(4) to finish the proof.
\end{proof}

In light of Propositions~\ref{4.19} and \ref{4.20} the following
question appears naturally.

\begin{question} Let $X$ be a Banach space and let
$Y$ be a closed linear
subspace of $X$. Suppose that $(Y,\weak)$ and $(X/Y,\weak)$ are $k^*$-metrizable. Is $(X,\weak)$ $k^*$-metrizable?
\end{question}

In the following proposition we describe situations when an answer to
this question is affirmative. We recall that $c_0$ is the Banach space of
all real sequences convergent to zero.

\begin{proposition}\label{4.23}
Let $X$ be a normed space and let $Y$ be a closed linear
subspace of $X$ such that both $(Y,\weak)$ and $(X/Y,\weak)$ are $k^*$-metrizable. Then $(X,\weak)$ is $k^*$-metrizable if one of
the following conditions holds{\rm:}

{\rm1)} $Y$ is complemented in $X$;

{\rm2)}
there is a linear continuous operator $T\colon\, X\to Z$ into a normed
space $Z$ such that $T|Y$ is an isomorphic embedding and $(Z,\weak)$ is $k^*$-metrizable;

{\rm3)}
$X$ is separable and $Y$ is isomorphic to a subspace of $c_0$;

{\rm4)}
$X$ is complete and $X/Y$ has the Shur property.
\end{proposition}

\begin{proof} 1. The first statement follows immediately from the second.
\smallskip

2. Suppose that $T\colon\, X\to Z$ is a  continuous linear
 operator into a normed
space such that $T|Y$ is an isomorphic embedding and the space $(Z,\weak)$ is $k^*$-metrizable. By Theorem~\ref{3.3}, the product
$(X/Y,\weak)\times (Z,\weak)$ is  $k^*$-metrizable. Denote by
$Q\colon\, X\to X/Y$ the quotient operator. It can be shown that the operator
$(Q,T)\colon\, X\to(X/Y)\times Z$ is an isomorphic embedding. Then $(X,\weak)$,
being a subspace of $(X/Y,\weak)\times (Z,\weak)$, is $k^*$-metrizable.
\smallskip

3. Next, assume that the space $X$ is separable and there is an isomorphism
$h\colon\, Y\to Z$ of $Y$ onto a subspace $Z$ of $c_0$. By the
Sobczyk theorem
\cite[Theorem~94]{HHZ},
there is a  continuous linear operator $T\colon\, X\to c_0$ such that $T|Y=h$.
Since the space $c_0$ has the separable dual $c_0^*=l^1$, the space
$(c_0,\weak)$ is an  $\aleph_0$-space according to
Proposition~\ref{4.14}. Now the
previous statement yields that $(X,\weak)$ is $k^*$-metrizable.
\smallskip

4. Finally assume that $X$ is a Banach space and the quotient $X/Y$ has
the Shur property. It follows from the Michael selection theorem
that there is
a continuous section $s\colon\, X/Y\to X$ of the quotient operator $Q\colon\, X\to X/Y$.
By Theorem~\ref{3.3}, the product $(Y,\weak)\times (X/Y)$ is $k^*$-metrizable. Consider the bijective continuous map $f\colon\, (Y,\weak)\times
(X/Y)\to X$ defined by $f(y,x)=y+s(x)$. To show that $(X,\weak)$ is $k^*$-metrizable, it suffices to verify that the map $f$ is proper.
Let $K$ be a weakly compact subset of $X$. Then $Q(K)$ is a weakly compact
subset of $X/Y$ and by Eberlein-\v Smulyan Theorem \cite[p.52]{HHZ}, $Q(K)$ is sequentially compact in the weak topology of $X/Y$. Since the space $X/Y$ has the Shur property, the set
$Q(K)$ is norm compact. Consequently, $s(Q(K))$ is a norm compact subset
of $X$ and $C=(K-s(Q(K)))\cap Y$ is a weakly compact subset of $Y$.
Observing that $f^{-1}(K)\subset C\times Q(K)$ we see that $f$ is proper.
\end{proof}

Finally, we show that the equivalence of the conditions $(1)\Leftrightarrow(8)$ in Theorem~\ref{4.16} cannot be proved without any restrictions on $X$. A suitable counterexample can be constructed using the operation of the $l_1$-sum $\big(\sum_{i\in\w}X_i\big)_{l_1}$ of a sequence of Banach spaces $X_i$ with norms $\|\cdot\|_i$. We recall that
$$\big(\sum_{i\in\w}X_i\big)_{l_1}=\{(x_i)_{i\in\w}\in\prod_{i=1}^\infty X_i:\sum_{i\in\w}\|x_i\|<\infty\}$$
has norm $\|(x_i)\|=\sum_{i\in\w}\|x_i\|<\infty$.

\begin{proposition}\label{n13.11} For any sequence $\{(X_i,\|\cdot\|_i)$ of normed spaces whose weak topology is $k^*$-metrizable the weak topology of the $l_1$-sum $X=\big({\sum_{i\in\w}}X_i\big)_{l_1}$ is $k^*$-metrizable. If infinitely many spaces $X_i$ fail to have the Shur property, then the weak topology of the unit ball of $X$ is not $\cs$-metrizable.
\end{proposition}

\begin{proof} For every $i\in\IN$ fix a $\sigma$-compact-finite $k$-network $\mathcal N_i$ for the $k^*$-metrizable space $(X_i,\weak)$. Let $B$ denote the closed unit ball of the normed space $X=\big({\sum_{i\in\w}}X_i\big)_{l_1}$. Given a number $\e\in\{2^{-m}:m\in\w\}$, a finite subset $F\subset\w$, and a sequence of sets $(N_i)_{i\in F}\in\prod_{i\in F}\mathcal N_i$, consider the set
$$N[(N_i)_{i\in F},\e]=\{(x_i)\in X:\mbox{$\forall i\in F\; x_i\in N_i$ and $\sum_{i\in\w\setminus F}\|(x_i)\|\le \e\}$}\}.$$
It is easy to check that the family $\mathcal N$ of all such sets $N[(N_i)_{i\in F},\e]$
is $\sigma$-compact-finite in $(X,\weak)$. It remains to check that $\mathcal N$ is a $k$-network for $(X,\weak)$. According to Proposition~\ref{n7.2} this will follow as soon as we show that all compact subsets of $(X,\weak)$ are sequentially compact and $\mathcal N$ is a $\wcs^*$-network for $(X,\weak)$.

First we check that each weakly compact subset $K$ of $X$ is metrizable. Given a number $i\in\w$ consider the projection $K_i$  of $K$ onto the factors $X_i$. Then $K_i$, being a compact subset of the $k^*$-metrizable space $(X_i,\weak)$, is metrizable and separable. Consequently, the  linear hull $L_i$ of $K_i$ is separable in $(X_i,\weak)$ and also in $X_i$ because weak and norm separability is equivalent for convex sets. Then $K$, being a subset of the separable normed space $\big(\sum_{i\in\w}L_i)_{l_1}$, has countable network both in norm and weak topologies. Consequently, the weak topology of $K$ is metrizable because it has countable network.

Next, we check that $\mathcal N$ is a $\wcs^*$-network for $(X,\weak)$. Fix an open set $W$ in $(X,\weak)$ and a sequence $\{x^n\}_{n\in\w}\subset W$ weakly convergent to a point $x^\infty\in W$. 
It will be more convenient to think of elements of $X$ as functions $x:\w\to \bigcup_{i\in\w}X_i$ with $x(i)\in X_i$ for all $i\in\w$. 

The compactness of the set $S=\{x^n:n\le\infty\}$ in $W$ implies the existence of an open neighborhood $W_0$ of the origin in $(X,\weak)$ such that $S+W_0+W_0\subset W$. Next, find $\e\in \{2^{-m}:m\in\w\}$ such that $W_0$ contains the closed $2\e$-ball centered at the origin.
 
 We claim that for any $\e>0$ there is $m\in\w$ such that $\sum_{i=m}^\infty \|x^n(i)\|_i\le\e$ for all but finitely many numbers $n$. Assume the converse: there is $\e>0$ such that for any $m\in\w$ there are infinitely many numbers $n\in\w$ such that $\sum_{i=m}^\infty \|x^n(i)\|_i> \e$. Inductively we can construct two increasing number sequences $(m_k)$ and $(n_k)$ such that for any $k\in\w$ 
$$\sum_{i=m_k+1}^{m_{k+1}}\|x^{n_k}(i)\|_i>\e.$$
By the Hahn-Banach Theorem for every $k\in\w$ we can find a linear functional $f_k$ on $\big(\sum_{i=m_k+1}^{m_{k+1}}X_i\big)_{l_1}$ with norm $\|f_k\|=1$ such that $f_k\big(x^{n_k}|[m_k, m_{k+1})\big)>\e$. These functionals form a continuous operator $$f:X\to l_1,\; f:x\mapsto (f_k\big(x|[m_k,m_{k+1})\big)_{k\in\w}$$ on $X$. The weak convergence of $(x^{n_k})$ to $x^\infty$ implies the weak convergence of $\big(f(x^{n_k})\big)$ to $f(x^\infty)$. Since $l_1$ has the Shur property, this sequence converges in norm, which yields a number $i_0\in\w$ with 
$$\sum_{i\ge i_0}|f_i\big(x^{n_k}|[m_i,m_{i+1})\big)|<\e\mbox{ \ and hence \ } |f_k\big(x^{n_k}|[m_k,m_{k+1})\big)|<\e$$ for all $k\ge i_0$. But this contradicts the choice of the functional $f_k$.

So we may take a number $m\in\w$ such that $\sum_{i=m}^\infty \|x^n(i)\|_i\le\e$ for all numbers $n\le \infty$. Let $\pr_m:X\to \big(\sum_{i<m}X_i\big)_{l_1}$ be the projection of $X$ onto the subspace $\big(\sum_{i<m}X_i\big)_{l_1}=\{(z_i)\in X:z_i=0$ for all $i\ge m\}$. The sequence $(\pr_m(x^n))_{n\in\w}$ weakly converges to $\pr_m(x^\infty)$. Since the norm distance between $x^n$ and its projection $\pr_m(x^n)$ does not exceed $\e$, we get $\{\pr_m(x^n):n\in\w\}\subset S+W_0$. Since $\mathcal N_i$ are $k$-networks in $(X_i,\weak)$, we can select sets $N_i\in\mathcal N_i$ for $i<m$ whose product $\prod_{i<m}X_i$ lies in $(S+W_0)\cap \big(\sum_{i<m}X_i)_{l_1}$ and contains infinitely many points of the sequence $(\pr_m(x^n))$. Then $N[(N_i)_{i<m},\e]\in\mathcal N$ lies in $S+W_0+\e\, B\subset W$ and contains infinitely many points of the sequence $(x^n)$, witnessing that $\mathcal N$ is a $\wcs^*$-network for $(X,\weak)$.

By Proposition~\ref{n7.2}, the $\sigma$-compact finite $\wcs^*$-network $\mathcal N$ is a $k$-network for $(X,\weak)$ and by Theorem~\ref{n7.4}, the space $(X,\weak)$ is $k^*$-metrizable.
\medskip

Now assuming that infinitely many spaces $X_i$ fail to have the Shur property, we shall show that the weak unit ball $(B,\weak)$ of $X$ is not $cs$-metrizable. Without loss of generality, all the spaces $X_k$ fail the Shur property, which allows us in each space $X_k$ to find a sequence $(x^k_n)_{n\in\w}$ on the unit sphere that weakly converges to zero. Identify each space $X_k$ with the subspace $\{(z_i)_{i\in\w}\in X:z_i=0$ for all $i\ne k\}$ of $X$ and consider the set $V=\{0,x^k_n:k,n\in\w\}\subset (B,\weak)$. It can be shown that the space $V$ fails to have the Arkhangelski's property $(\alpha_4)$,  which implies that $(B,\weak)$ cannot be $cs$-metrizable.
\end{proof}

Using the above Proposition one can show that the weak unit ball of the space $L_1[0,1]$ is not $k$-metrizable.

\begin{problem} Is the space $L_1[0,1]$ $k^*$-metrizable in its weak topology?
\end{problem}

More generally, we can pose a

\begin{problem} Characterize Banach spaces $X$ whose weak unit ball is $k$-metrizable (resp. $k^*$-metrizable). 
\end{problem}

Let us note that in the realm of nonseparable Banach spaces the $k^*$-metrizability of the weak topology is "orthogonal" to the WCG-property.
We recall that a Banach space $X$ is {\em weakly compact generated} (briefly, WCG) if there is a weakly compact subset $K$ whose linear hull is dense in $X$.

\begin{proposition} If a WCG Banach space $X$ has $k^*$-metrizable weak topology, then $X$ is separable.
\end{proposition}

The proof follows from the fact that compact subsets of $k^*$-metrizable spaces are metrizable and hence separable.

\begin{remark}
{\rm
Propositions~\ref{4.14}, \ref{4.15}, and \ref{4.23}, \ref{n13.11} give
 examples of infinite-dimensional Banach spaces $X$ which being endowed
with the weak topology are $k^*$-metrizable spaces and hence have the
weak Skorohod property. In contrast, the space $(X,\weak)$ has the {\em strong}
Skorohod property if and only if $X$ is finite-dimensional, see \cite{BBK2}.
}\end{remark}

\section{The structure of sequential $k^*$-metrizable groups}\label{s14}

Results of the preceding section give us many examples of non-metrizable
$k^*$-metrizable topological groups. However such
topological groups rarely are sequential in view of the following result \cite[2.7]{LST}.

\begin{theorem}[Liu-Sakai-Tanaka]\label{n14.1} A sequential topological group $G$ with point-countable cosmic $k$-network contains a cosmic open subgroup.
\end{theorem}

A $k$-network $\mathcal N$ will be called {\em cosmic} if each space $N\in\mathcal N$ is {\em cosmic} in the sense that $N$ is the continuous image of a separable metrizable space.  This theorem of C.Liu, M.Sakai, Y.Tanaka can be completed by the following result of T.Banakh and L.Zdomskyy \cite{BZ}.
 
\begin{theorem}[Banakh-Zdomskyy]\label{n14.2} A sequential topological group $G$ with countable $\cs^*$-character is a stratifiable $\aleph$-space. More precisely, $G$ is either metrizable or contains an open $k_\omega$-subgroup.
\end{theorem}

A topological group $G$ is called a {\em $k_\omega$-group} if its underlying topological space is a $k_\omega$-space.

We shall use these two theorems to prove a structure theorem for sequential $k^*$-metrizable topological groups.

\begin{theorem}\label{n14.3} Each sequential $k^*$-metrizable group $G$ has countable $\cs^*$-character and hence is a stratifiable $\aleph$-space. More precisely, $G$ is either metrizable or contains an open $k_\w$-subgroup. 
\end{theorem}

\begin{proof} This theorem will follow from Theorem~\ref{n14.2} as soon as we check that $G$ has countable $\cs^*$-character. Assuming the converse and applying Theorem~\ref{n6.1}, we will get that $\alpha_4(G)>\aleph_0$ and hence $\alpha_4^+(G)>\aleph_1$, which means that $G$ contains a closed copy of the uncountable sequential fan $S_{\w_1}$. In particular, $G$ contains a closed  topological copy of $S_\w$. 
By Lemma 4 from \cite{BZ}, a sequential topological group cannot simultaneously contain closed copies of the sequential fan $S_\w$ and its metrizable counterpart, the sequential hedgehog $$H_\w=\{0,2^{-m}e_n:n,m\in\w\}\subset l^2,$$ where $(e_n)$ is the standard orthonormal base of the separable Hilbert space $l^2$.  The space $H_\w$ is not locally compact at its unique non-isolated point $0$, and is the smallest non-locally compact space among first countable non-locally compact spaces: each first countable non-locally compact space contains a closed copy of $H_\w$.

Therefore, the group $G$ contains no closed copy of $H_\w$. We shall use this fact to show that $G$ has a point-countable cosmic $k$-network.

Fix a subproper map $\pi:M\to X$ of a metrizable space $M$ onto $X$ and let $s:X\to M$ be a section of $\pi$ that preserves precompact sets. 

\begin{claim}\label{n14.6} Each point $x\in M$ has a neighborhood $U\subset M$ whose image $\pi(U)$ is precompact in $G$.
\end{claim}

 Assuming the converse, fix a countable neighborhood base $(U_n)$ at $x$. For every $U_n$ the image $\pi(U_n)$ is not precompact, which means that the closure $\overline{\pi(U_n)}$ is not compact. The sequentiality of the $k^*$-metrizable space $G$ implies that each countably compact space is compact, see Proposition~\ref{n3.7}. Hence $\overline{\pi(U_n)}$ is not countably compact and contains a countable infinite closed discrete subset $D_n$ not containing the point $y=\pi(x)$. By induction we can construct an increasing number sequence $(n_k)$ such that $\overline{\pi(U_{n_k})}\cap \bigcup_{i\le n_{k-1}}D_i=\emptyset$ for all $k$. The ``convergence'' of the sets $U_n$ to $x$ implies that $(D_{n_k})$ converges to $y=\pi(x)$ in the sense that each neighborhood of $y$ contains all but finitely many sets $D_n$. This fact can be used to show that the space $D=\{y\}\cup\bigcup_k D_{n_k}$ is a closed copy of the hedgehog $H_\w$ in $G$, which is not possible.\hfill $\Box$
\smallskip

So we can find a $\sigma$-locally-finite base $\mathcal B$ for $M$ consisting of set $U\subset M$ with precompact images $\pi(U)$ in $G$. Since each compact subset of $G$ is metrizable, the sets $s^{-1}(U)\subset \pi(U)$, $U\in\mathcal B$, are metrizable and separable. The precompact preserving property of $s$ implies that $\mathcal N=\{s^{-1}(U):U\in\mathcal B\}$ is a $\sigma$-compact-finite $k$-network consisting of metrizable separable subsets of $G$. In particular, $\mathcal N$ is a point-countable cosmic $k$-network for $G$. By Theorem~\ref{n14.1}, the group $G$ contains an open cosmic subgroup $H$. This subgroup has countable network and, being a $k^*$-metrizable space, has countable $k$-network by Theorem~\ref{n8.2}. By Proposition~\ref{n6.3}, $\cs^*(G)=\cs^*_\chi(H)\le\alpha_4(H)\le \ext(H)\le knw(H)\le\aleph_0$. So $G$ has countable $\cs^*$-character. Application of Theorem~\ref{n14.2} completes the proof.
\end{proof}

Theorem~\ref{n14.3} will be applied to show that many cardinal characteristics of a sequential $k^*$-metrizable group coincide with a specific group-topological invariant $ib(G)$ called the {\em index of boundedness} of $G$ and equal  to the smallest cardinal $\kappa$ such that for any open neighborhood $U\subset G$ of the neutral element $e\in G$ there is a subset $F\subset G$ of size $|F|\le \kappa$ with $G=F\cdot U=\{x\cdot y:x\in F,\; y\in U\}$. It is known that $ib(G)\le\min\{\ext(G),d(G)\}$ for any topological group.

\begin{theorem}\label{n14.9} If $G$ is an infinite sequential $k^*$-metrizable group, then $ib(G)=d(G)=\ext(G)=s(G)=l(G)=nw(G)=knw(G)$.
\end{theorem}

\begin{proof} The assertion is trivial if $G$ is metrizable.
So, assume that $G$ is not metrizable.
By Theorem~\ref{n14.3}, the group $G$ contains an open $k_\w$-subgroup $H$.
Since all compact subsets in $G$ are metrizable, we get $knw(H)=\aleph_0$.
By the definition of $ib(G)$, there is a subset $F\subset G$ of size $|F|\le ib(G)$ with $G=F\cdot H$. Then for any countable $k$-network $\mathcal N$ in $H$ the family $F\cdot\mathcal N=\{f\cdot N: f\in F,\;N\in\mathcal N\}$ is a $k$-network for $G$ having size $|F\cdot \mathcal N|\le|F|\cdot |\mathcal N|=ib(G)$, which yields $knw(G)\le ib(G)$. Combining this inequality with $ib(G)\le\min\{\ext(G),d(G)\}\le nw(G)\le knw(G)$, we get the desired equalities.
\end{proof} 

Theorem~\ref{n14.3} will be also applied to the problem of topological classification of non-metrizable sequential $k^*$-metrizable groups. This can be done for two classes of groups: punctiform groups and locally convex spaces.

A topological space $X$ is called {\em punctiform} if each compact subset of $X$ is zero-dimensional. A topological classification of punctiform $k^*$-metrizable group involves the notion of the compact scatteredness rank.

Given
a topological space $X$  let $X_{(1)}\subset X$ denote the set of
all non-isolated points of $X$. For each ordinal $\alpha$ define
the $\alpha$-th derived set $X_{(\alpha)}$ of $X$ by transfinite
induction:
$X_{(\alpha)}=\bigcap_{\beta<\alpha}(X_{(\beta)})_{(1)}$. By
the {\em scatteredness height} $\mathrm{sch}(X)$ of $X$ we
understand the smallest ordinal $\alpha$ such that
$X_{(\alpha+1)}=X_{(\alpha)}$. A topological space $X$ is {\em
scattered} if $X_{(\alpha)}=\emptyset$ for some ordinal $\alpha$.
By the {\em compact scatteredness rank}\/  of a topological
space $X$ we understand the ordinal
$\mathrm{scr}(X)=\sup\{\mathrm{sch}(K): K$ is a scattered compact
subspace of $X\}$. In particular, $\mathrm{scr}(X)=\w_1$ for any uncountable compact metrizable space $X$ (such a space $X$ contains a copy of the Cantor set and hence a copy of each countable compactum).

\begin{theorem}\label{n15.2} Two non-metrizable sequential punctiform
$k^*$-metrizable groups $G,H$ are homeomorphic if and only if $d(G)=d(H)$ and
$\mathrm{scr}(G)=\mathrm{scr}(H)$.
\end{theorem}

This theorem was proved in \cite{BZ} for sequential groups with countable $\cs^*$-character. By Theorem~\ref{n14.3} each sequential $k^*$-metrizable group has countable $\cs^*$-character, so the theorem follows.

Also we can classify non-metrizable sequential $k^*$-metrizable locally convex spaces. In this cases all of them are homeomorphic either to $\IR^\infty=\underset{\longrightarrow}{\lim}\,\IR^n$, the strict inductive limit of Euclidean spaces, or to the product $\IR^\infty\times Q$ of $\IR^\infty$ and the Hilbert cube $Q=[0,1]^\w$.

\begin{theorem}\label{n15.3} Each non-metrizable sequential $k^*$-metrizable locally convex
space is homeomorphic
to $\mathbb{R}^\infty$ or $\mathbb{R}^\infty\times Q$.
\end{theorem}

This theorem follows from the corresponding classification of sequential locally convex spaces with countable $\cs^*$-character proved in \cite{BZ}.

\end{document}